\documentclass[11pt,reqno]{amsart}

\usepackage{amsmath}
\usepackage{amssymb}
\usepackage{amsthm}

\newtheorem{thm}{Theorem}[section]
\newtheorem{prop}[thm]{Proposition}
\newtheorem{lem}[thm]{Lemma}
\newtheorem*{cor}{Corollary}
\newtheorem*{assumption}{Assumption}
\theoremstyle{definition}
\newtheorem{defn}[thm]{Definition}
\theoremstyle{remark}
\newtheorem*{rem}{Remark}
\numberwithin{equation}{section}

\begin{document}

\title[$C^*$-algebraic quantum groupoid]
{Construction of a $C^*$-algebraic quantum groupoid from a weak multiplier Hopf algebra}

\author{Byung-Jay Kahng}
\date{}
\address{Department of Mathematics and Statistics\\ Canisius College\\
Buffalo, NY 14208, USA}
\email{kahngb@canisius.edu}

\keywords{Weak Hopf algebra, Weak multiplier Hopf algebra, Separability idempotent, 
Locally compact quantum groupoid, Multiplier Hopf algebra, Locally compact quantum group}

\begin{abstract}
Van Daele and Wang developed a purely algebraic notion of {\em weak multiplier Hopf algebras\/}, which extends 
the notions of Hopf algebras, multiplier Hopf algebras, and weak Hopf algebras. With an additional requirement 
of an existence of left or right integrals, this framework provides a self-dual class of algebraic quantum groupoids.  
The aim of this paper is to show that from this purely algebraic data, with only some relatively minimal additional requirements 
(``quasi-invariance''), one can construct a {\em $C^*$-algebraic quantum groupoid of separable type\/}, recently 
defined by the author, with Van Daele. The $C^*$-algebraic quantum groupoid is represented as an operator algebra 
on the Hilbert space constructed from the left integral, and the comultiplication is determined by means of a certain 
multiplicative partial isometry $W$, which is no longer unitary.  In the last section (Appendix), we obtain some results 
in the purely algebraic setting, which have not appeared elsewhere.
\end{abstract}
\maketitle

\setcounter{section}{-1} 
\section{Introduction}

In a series of papers, Van Daele and Wang introduced a purely algebraic notion of {\em weak multiplier Hopf algebras\/} 
\cite{VDWangwha0}, \cite{VDWangwha1}, \cite{VDWangwha2}, \cite{VDWangwha3}.  In short, a weak multiplier 
Hopf algebra is a pair $(A,\Delta)$, where $A$ is a non-degenerate idempotent algebra and $\Delta$ is a comultiplication, 
satisfying some number of conditions.  They are natural generalizations of Hopf algebras (when $A$ is unital and 
$\Delta$ is non-degenerate), multiplier Hopf algebras (when $A$ is non-unital and $\Delta$ is non-degenerate), and 
weak Hopf algebras (when $A$ is unital, but $\Delta(1)\ne1\otimes1$). For a weak multiplier Hopf algebra, the algebra 
is not assumed to be unital and the comultiplication is no longer assumed to be non-degenerate.

Going further, in \cite{VDWangwha3}, they considered the situation in which a weak multiplier Hopf algebra possesses (a faithful 
family of) left/right invariant functionals (or ``integrals'').  While in a purely algebraic setting, this framework is known to include 
all compact and discrete quantum groups, as well as all weak Hopf algebras and finite quantum groupoids.  It is true that not all 
quantum groups/groupoids are contained in this framework, and some classical groups are left out.  Nevertheless, it is shown that 
a dual object can be constructed within this framework, thereby giving rise to a nice self-dual class of {\em algebraic quantum groupoids\/}. 

This purely algebraic framework provided a strong motivational basis for a $C^*$-algebraic framework of {\em locally compact 
quantum groupoids of separable type\/}, by Van Daele and the author \cite{BJKVD_qgroupoid1}, \cite{BJKVD_qgroupoid2}.  
There is a strong resemblance between the purely algebraic framework (= weak multiplier Hopf algebras) and the $C^*$-framework 
(= locally compact quantum groupoids of separable type).  Having said this, it has never been made explicit whether there is indeed 
a direct pathway from the purely algebraic setting to the $C^*$-setting without too many additional requirements. Naively speaking, 
this is about constructing a $C^*$-completion of the algebra.  But there are some subtle issues to consider.

The main purpose of this paper is to clarify that we can indeed carry out the construction of a $C^*$-algebraic locally 
compact quantum groupoid (in the sense of  \cite{BJKVD_qgroupoid1}, \cite{BJKVD_qgroupoid2}) out of a weak multiplier 
Hopf ${}^*$-algebra equipped with a faithful left (or right) integral.  Just as the work of Kustermans and Van Daele 
of similar nature (\cite{KuVD}), constructing a $C^*$-algebraic quantum group from a purely algebraic object of 
a multiplier Hopf algebra, made fundamental contributions to the development of the theory of locally compact 
quantum groups, we hope that the current work can help us understand better the theory of quantum groupoids.

The paper is organized as follows: In Section~\ref{sec1}, we give an overview of the purely algebraic framework 
of {\em weak multiplier Hopf algebras\/}.  This is needed, not only for motivational purposes but for clarifying the ingredients 
necessary for the construction of the $C^*$-algebraic quantum groupoid in what follows.  The definition and the properties 
of weak multiplier Hopf algebras are given in this section, though most of the detailed proofs are skipped (Instead, 
the relevant theorems elsewhere are referred to.)   We do not, however,  describe the original definition (as in 
\cite{VDWangwha1}).  Instead, we take a more recent but equivalent approach, as in \cite{BJKVD_LSthm}: 
In that paper, it is shown that a regular weak multiplier bialgebra (in the sense of \cite{BohmGomezLopez}) becomes 
a weak multiplier Hopf algebra, if it has sufficient number of left and right integrals.  Being more recent is one thing, 
but actually, for our purposes of working with the integrals and studying the duality, this characterization turns out to be 
more convenient.

One aspect of note is that we get to consider the case of a weak multiplier Hopf ${}^*$-algebra here.  While the discussion 
about the case with an involution has appeared in the original literatures on weak multiplier Hopf algebras and separability 
idempotents, typically they have appeared in a scattered way. 

In Section~\ref{sec2}, from the purely algebraic data given above, we construct the ``base'' $C^*$-algebras $B$ and $C$, 
as well as their multiplier algebras $M(B)$ and $M(C)$.  They are essentially the source algebra and the target algebra.  
They are equipped with certain KMS-weights $\nu$ and $\mu$.

The construction of the $C^*$-algebra $A$ and the comultiplication $\Delta$ is carried out in Section~\ref{sec3}, whose 
representation is given in terms of a certain partial isometry $W$. The construction of the canonical idempotent 
$E\in M(A\otimes A)$ is also given in this section.

In Section~\ref{sec4}, we discuss the polar decomposition of the antipode map, which allows it to be properly defined at the 
operator algebra level.  We also collect some relevant technical results. 

The aim of Section~\ref{sec5}, is to carry out a construction of a left invariant weight and a right invariant weight.  For technical 
reasons, the weights are considered at the von Neumann algebra level first, before the $C^*$-algebraic weights are constructed. 
To make things work, we we require a certain {\em quasi-invariance condition \/} at the ${}^*$-algebra level.  This allows an 
existence of a Radon--Nikodym relationship between our two weights in terms of a certain modular operator, which plays a central 
role.  Eventually, we are able to construct two KMS weights, $\varphi$ and $\psi$, on $A$, satisfying the left invariance condition 
and the right right invariant condition, respectively.

In this way, in Section~\ref{sec6}, we can verify that we have successfully constructed a $C^*$-algebraic quantum groupoid, that fits 
well in the $C^*$-algebraic framework developed in \cite{BJKVD_qgroupoid1}, \cite{BJKVD_qgroupoid2}, as expected.

We chose not to pursue the construction of the dual object here, which should be more or less similar.  The reason for not doing this 
is because the $C^*$-algebraic framework for the duality of the locally compact quantum groupoids of separable type is still in the works 
\cite{BJKVD_qgroupoid3}.  In that paper, we plan to give a clarification of the duality picture in the $C^*$-algebraic framework (for a related 
work, refer also to \cite{BJK_mpi}).  We will postpone to a future occasion to verify the expected result that the $C^*$-algebraic counterpart 
of the algebraic dual $(\widehat{\mathcal A},\widehat{\Delta})$ is indeed isomorphic to the $C^*$-algebraic dual of $(A,\Delta)$ obtained here.

In Appendix (Section~\ref{appx}), we gathered some results in the purely algebraic framework regarding the modular element $\delta$. 
Even though this is done in the purely algebraic setting, the author could not find a suitable reference, and some of the results here 
may be new.  As such, all proofs are given for the results in the Appendix section.

\bigskip

\noindent{\sc Acknowledgments:}
The author wishes to thank Alfons Van Daele (Leuven), who inspired him to pursue research projects on weak multiplier Hopf algebras 
and locally compact quantum groupoids. The author is always indebted to his constant support and guidance. The initial seed for 
the current work arose from talking with Thomas Timmermann (M{\"u}nster), while both of us were visiting Albert Sheu (Kansas) in 2017, 
who has always been supportive of the author's projects.

\section{The algebraic framework: Weak multiplier Hopf algebra with integrals}\label{sec1}

\subsection{Preliminaries}\label{sub1.1}

As indicated in Introduction, we will primarily follow the description in \cite{BJKVD_LSthm}.  We will only consider associative algebras 
over $\mathbb{C}$.  We require that the algebras are non-degenerate, which means that the product on an algebra is non-degenerate 
as a bilinear form.  So ${\mathcal A}$ is a non-degenerate algebra, if $[ab=0,\forall b\in{\mathcal A}]\,\Rightarrow\,[a=0]$; also 
$[ba=0,\forall b\in{\mathcal A}]\,\Rightarrow\,[a=0]$.  We also require that the algebras are idempotent, written ${\mathcal A}^2={\mathcal A}$, 
meaning that every element in the algebra can be written as a sum of products of two elements.  It is evident that if an algebra ${\mathcal A}$ 
is unital or has local units, it is automatically non-degenerate and idempotent.  Typically, however, we do not expect our algebras to be unital. 
We will denote by ${\mathcal A}^*$ the dual vector space of ${\mathcal A}$, consisting of the linear functionals on ${\mathcal A}$.

We will mostly consider ${}^*$-algebras, equipped with an involution.  For a non-degenerate ${}^*$-algebra ${\mathcal A}$, we can define 
its multiplier algebra $M({\mathcal A})$, which is a unital ${}^*$-algebra containing ${\mathcal A}$ as an essential self-adjoint ideal.  It is 
the largest such, and is unique.  It can be characterized in terms of double centralizers.  If ${\mathcal A}$ is unital, then we have 
$M({\mathcal A})={\mathcal A}$. We can also consider ${\mathcal A}\odot{\mathcal A}$, the algebraic tensor product, and its multiplier 
algebra $M({\mathcal A}\odot{\mathcal A})$.

\subsection{Comultiplication}\label{sub1.2}

Let ${\mathcal A}$ be a non-degenerate idempotent ${}^*$-algebra. By a {\em comultiplication\/} on ${\mathcal A}$, 
we mean a ${}^*$-homomorphism $\Delta$ from ${\mathcal A}$ into $M({\mathcal A}\odot{\mathcal A})$ such that 
\begin{equation}\label{(comult1)}
(\Delta a)(1\otimes b)\in{\mathcal A}\odot{\mathcal A}, \quad {\text { and }} 
\quad (c\otimes1)(\Delta a)\in{\mathcal A}\odot{\mathcal A}, \quad {\text { for all 
$a,b,c\in{\mathcal A}$,}}
\end{equation}
which is also required to satisfy the following ``weak coassociativity'' condition:
\begin{equation}\label{(weakcoassociativity)}
(c\otimes1\otimes1)(\Delta\otimes\operatorname{id})\bigl((\Delta a)(1\otimes b)\bigr)
=(\operatorname{id}\otimes\Delta)\bigl((c\otimes1)(\Delta a)\bigr)(1\otimes1\otimes b), 
\quad {\text { for $a,b,c\in{\mathcal A}$}}.
\end{equation}
Note that condition~\eqref{(comult1)} is needed to formulate the weak coassociativity \eqref{(weakcoassociativity)}.

As ${\mathcal A}$ is a ${}^*$-algebra and $\Delta$ is a ${}^*$-homomorphism, it automatically 
follows from condition~\eqref{(comult1)} that we also have:
\begin{equation}\label{(comult2)}
(\Delta a)(c\otimes 1)\in{\mathcal A}\odot{\mathcal A}, \quad {\text { and }} 
\quad (1\otimes b)(\Delta a)\in{\mathcal A}\odot{\mathcal A}, \quad 
{\text { for all $a,b,c\in{\mathcal A}$.}}
\end{equation}
There is also another version of the weak coassociativity, as follows:
\begin{equation}\label{(weakcoassociativity_alt)}
(\Delta\otimes\operatorname{id})\bigl((1\otimes b)(\Delta a)\bigr)(c\otimes1\otimes 1)
=(1\otimes1\otimes b)(\operatorname{id}\otimes\Delta)\bigl((\Delta a)(c\otimes 1)\bigr), 
\quad {\text { for $a,b,c\in{\mathcal A}$}}.
\end{equation}

\begin{rem} 
Having $({\mathcal A},\Delta)$ further satisfy \eqref{(comult2)} and \eqref{(weakcoassociativity_alt)} means that 
our comultiplication is ``regular'', in the sense of \cite{VDWangwha1} (see Definition~1.1 of that paper).  
\end{rem}

The comultiplication is also assumed to be ``full''.  This means that the left and the right legs of $\Delta({\mathcal A})$ 
are all of ${\mathcal A}$ (see \cite{VDWangwha0}, \cite{VDWangwha1}, also see \cite{BohmGomezLopez}). 
A way to characterize the fullness of $\Delta$ is as follows:
$$
\operatorname{span}\bigl\{(\operatorname{id}\otimes\omega)((\Delta a)(1\otimes b)):
a,b\in{\mathcal A},\omega\in{\mathcal A}^*\bigr\}={\mathcal A},
$$
$$
\operatorname{span}\bigl\{(\omega\otimes\operatorname{id})((b\otimes1)(\Delta a)):
a,b\in{\mathcal A},\omega\in{\mathcal A}^*\bigr\}={\mathcal A}.
$$

For $({\mathcal A},\Delta)$, we have the following result:

\begin{lem} \label{canonicalE} 
Suppose there exists a self-adjoint idempotent element $E\in M({\mathcal A}\odot{\mathcal A})$ 
such that 
$$
\Delta({\mathcal A})({\mathcal A}\odot{\mathcal A})=E({\mathcal A}\odot{\mathcal A}), 
\quad {\text { and }} \quad
({\mathcal A}\odot{\mathcal A})\Delta({\mathcal A})=({\mathcal A}\odot{\mathcal A})E.
$$
Then this idempotent is unique.  It is the smallest idempotent $E\in M({\mathcal A}\odot{\mathcal A})$ satisfying
$$
E(\Delta a)=\Delta a, \quad (\Delta a)E=\Delta a,\quad \forall a\in{\mathcal A}.
$$
\end{lem}

\begin{proof}

See Lemmas~3.3 and 3.5 of \cite{VDWangwha0}.

\end{proof}

\begin{rem}
The requirement to have $E$ self-adjoint is natural, because we are working with ${}^*$-algebras. The existence of the idempotent 
$E\in M({\mathcal A}\odot{\mathcal A})$ as above is referred to as $\Delta$ being {\em weakly non-degenerate\/}.  Note that when 
$E=1\otimes1$, we would indeed have the non-degeneracy of the comultiplication.  Moreover, if such an idempotent $E$ exists 
(called the {\em canonical idempotent\/}), then it can be shown that $\Delta$ has a unique extension to a ${}^*$-homomorphism 
$\widetilde{\Delta}:M({\mathcal A})\to M({\mathcal A}\odot{\mathcal A})$, such that 
$\widetilde{\Delta}(m)=E\widetilde{\Delta}(m)=\widetilde{\Delta}(m)E$, for all $m\in M({\mathcal A})$. 
For proof of this result and more details, see Appendix of \cite{VDWangwha0}. For convenience, we will just 
denote the extension map also by $\Delta$.
\end{rem}

It is also the case that we can make sense of the maps $\Delta\otimes\operatorname{id}$ and $\operatorname{id}\otimes\Delta$ 
as ${}^*$-homomorphisms naturally extended to $M({\mathcal A}\odot{\mathcal A})$, such that 
$$
(\Delta\otimes\operatorname{id})(m)=(E\otimes1)\bigl((\Delta\otimes\operatorname{id})(m)\bigr)
=\bigl((\Delta\otimes\operatorname{id})(m)\bigr)(E\otimes1),\quad\forall m\in M({\mathcal A}
\odot{\mathcal A}),
$$
and similarly for $\operatorname{id}\otimes\Delta$.  These results mean that when extended to the multiplier 
algebra level, we have $\Delta(1)=E$, and that $(\Delta\otimes\operatorname{id})(1\otimes1)=E\otimes1$ and 
$(\operatorname{id}\otimes\Delta)(1\otimes1)=1\otimes E$.

As a consequence of the weak coassociativity, namely Equations~\eqref{(weakcoassociativity)} and 
\eqref{(weakcoassociativity_alt)}, now knowing that the maps $\Delta\otimes\operatorname{id}$ and 
$\operatorname{id}\otimes\Delta$ are extended, we obtain the following {\em coassociativity\/} property:
\begin{equation}
(\Delta\otimes\operatorname{id})(\Delta a)=(\operatorname{id}\otimes\Delta)(\Delta a),
\quad \forall a\in{\mathcal A}.
\end{equation}
See again Appendix of \cite{VDWangwha0}.  

We will require one more condition on $\Delta$, which is also a part of the axioms for a weak multiplier Hopf 
algebra (see Definition~1.14 of \cite{VDWangwha1}). The following condition is referred to as the {\em weak 
comultiplicativity of the unit\/}:
\begin{equation}\label{(weakcomultiplicativity)}
(\operatorname{id}\otimes\Delta)(E)=(E\otimes1)(1\otimes E)=(1\otimes E)(E\otimes1).
\end{equation}
This condition already appeared in the theory of weak Hopf algebras (see Definition~2.1 in \cite{BNSwha1}).

\subsection{Separability idempotent}\label{sub1.3}

The canonical idempotent is further required to be a {\em separability idempotent\/}, in the sense of \cite{VDsepid}. 
See also Appendix~B of \cite{BJKVD_LSthm}.  In the below is a short summary. 

Suppose ${\mathcal B}$ and ${\mathcal C}$ are non-degenerate ${}^*$-algebras.  Consider a self-adjoint 
idempotent $E\in M({\mathcal B}\odot{\mathcal C})$ such that 
$$
E(1\otimes c)\in{\mathcal B}\odot{\mathcal C}, \ (b\otimes1)E\in{\mathcal B}\odot{\mathcal C}, 
\quad \forall b\in{\mathcal B},\forall c\in{\mathcal C}.
$$
As $E$ is self-adjoint, we also have $(1\otimes c)E\in{\mathcal B}\odot{\mathcal C}$, $E(b\otimes1)\in{\mathcal B}
\odot{\mathcal C}$, for $b\in{\mathcal B},c\in{\mathcal C}$.  Assume also that $E$ is ``full'', which means that 
its left and the right legs are all of ${\mathcal B}$ and ${\mathcal C}$, respectively. Or equivalently, we have 
$$
\operatorname{span}\bigl\{(\operatorname{id}\otimes\omega)(E(1\otimes c)):c\in{\mathcal C},
\omega\in{\mathcal C}^*\bigr\}={\mathcal B},
$$
$$
\operatorname{span}\bigl\{(\omega\otimes\operatorname{id})((b\otimes1)E):b\in{\mathcal B},
\omega\in{\mathcal B}^*\bigr\}={\mathcal C}.
$$

Then, it can be shown that $E\in M({\mathcal B}\odot{\mathcal C})$ automatically becomes a {\em regular\/} 
separability idempotent, which means that we have: 
\begin{equation}
E({\mathcal B}\odot1)=E(1\odot{\mathcal C}) \quad {\text { and }} \quad ({\mathcal B}\odot1)E
=(1\odot{\mathcal C})E.
\end{equation}
See Proposition~3.7 of \cite{VDsepid}.  The regularity condition implies ${\mathcal B}$ and ${\mathcal C}$ have 
local units (see \cite{VDsepid},\,v1). The other aspect of the regularity of $E$ is that there exist two anti-isomorphisms 
$S_{\mathcal B}:{\mathcal B}\to{\mathcal C}$ and $S_{\mathcal C}:{\mathcal C}\to{\mathcal B}$, characterized by 
\begin{equation}\label{(S_Bmap)}
E(b\otimes1)=E\bigl(1\otimes S_{\mathcal B}(b)\bigr), \quad b\in{\mathcal B},
\end{equation}
\begin{equation}\label{(S_Cmap)}
(1\otimes c)E=\bigl(S_{\mathcal C}(c)\otimes1\bigr)E, \quad c\in{\mathcal C}.
\end{equation}
It can be also shown that $(S_{\mathcal B}\otimes S_{\mathcal C})(E)=\varsigma E$, where $\varsigma$ is the flip 
map between ${\mathcal B}\odot{\mathcal C}$ and ${\mathcal C}\odot{\mathcal B}$. The maps $S_{\mathcal B}$ 
and $S_{\mathcal C}$ are in general not involutive, but they satisfy the following relations:
\begin{equation}
S_{\mathcal C}\bigl(S_{\mathcal B}(b)^*\bigr)^*=b \quad {\text { and }} \quad
S_{\mathcal B}\bigl(S_{\mathcal C}(c)^*\bigr)^*=c, \quad 
\forall b\in{\mathcal B},\forall c\in{\mathcal C}.
\end{equation}

The existence of such a self-adjoint separability idempotent element $E\in M({\mathcal B}\odot{\mathcal C})$ restricts 
the possible structure of the algebras ${\mathcal B}$ and ${\mathcal C}$, which are typically direct sums of 
finite-dimensional matrix algebras.  We will not go too deep into this discussion here.  See section~4 of \cite{VDsepid}.

By the general theory of regular separability idempotents \cite{VDsepid}, we have the existence of the following 
{\em distinguished linear functionals\/}, namely $\nu$ on ${\mathcal B}$ and $\mu$ on ${\mathcal C}$, such that 
$$
(\nu\otimes\operatorname{id})(E)=1, \quad {\text { and }} \quad 
(\operatorname{id}\otimes\mu)(E)=1.
$$
[The notations used in \cite{VDsepid} for the distinguished linear functionals are actually $\varphi_{\mathcal B}$ 
and $\varphi_{\mathcal C}$, but we are denoting them as $\nu$ and $\mu$ here, mainly to avoid a future confusion 
with the total linear functional $\varphi$ on ${\mathcal A}$.]

The distinguished functionals $\nu$ and $\mu$ are uniquely determined and are faithful.  Being a faithful functional 
means that we have $[\nu(bk)=0,\forall k\in{\mathcal B}]\,\Rightarrow\,[b=0]$, and similarly for $\mu$.  Meanwhile, with the 
algebras ${\mathcal B}$ and ${\mathcal C}$ being ${}^*$-algebras, and $E$ being self-adjoint, it also turns out that 
$\nu$ and $\mu$ become positive linear functionals (see section~4 of \cite{VDsepid}).

The general theory also tells us that the functional $\nu$ is equipped with a KMS-type automorphism $\sigma^{\nu}$ 
on ${\mathcal B}$, such that $\nu\circ\sigma^{\nu}=\nu$ and 
\begin{equation}\label{(nu_weakKMS)}
\nu(bb')=\nu\bigl(b'\sigma^{\nu}(b)\bigr), \quad \forall b,b'\in{\mathcal B},
\end{equation}
given by $\sigma^{\nu}={S_{\mathcal B}}^{-1}\circ{S_{\mathcal C}}^{-1}$.  As for $\mu$ also, there is a KMS-type 
automorphism $\sigma^{\mu}$ on ${\mathcal C}$, such that $\mu\circ\sigma^{\mu}=\mu$ and 
\begin{equation}\label{(mu_weakKMS)}
\mu(cc')=\mu\bigl(c'\sigma^{\mu}(c)\bigr), \quad \forall c,c'\in{\mathcal C},
\end{equation}
given by $\sigma^{\mu}=S_{\mathcal B}\circ S_{\mathcal C}$. The existence of such automorphisms is referred to 
as the {\em weak KMS-property\/}.  See Proposition~2.8 of \cite{VDsepid} for more details. 

In addition, we can show that 
$$
\mu=\nu\circ S_{\mathcal C} \quad {\text { and }} \quad \nu=\mu\circ S_{\mathcal B},
$$
by using the result $(S_{\mathcal B}\otimes S_{\mathcal C})E=\varsigma E$ and the uniqueness of the functionals.

Here is a Radon--Nikodym type theorem on the functionals on the algebras ${\mathcal B}$ and ${\mathcal C}$:

\begin{prop} \label{sepid_RN}
\begin{enumerate}
\item Consider the distinguished linear functional $\nu$ on ${\mathcal B}$. For any other linear functional $g$ on 
${\mathcal B}$, there is a unique element $y\in M({\mathcal B})$ such that $g(b)=\nu(by)$ for all $b\in{\mathcal B}$. 
Moreover, the functional $g$ is faithful if and only if $y$ is invertible in $M({\mathcal B})$.

\item Similar results hold true for linear functionals on ${\mathcal C}$. 
\end{enumerate}
\end{prop}

\begin{proof}
This is a consequence of the regularity property of $E$.  See Proposition~1.2 of \cite{VD_aqg}, which is actually 
more general than the result given here.
\end{proof}

As a consequence of Proposition~\ref{sepid_RN} above, we can show that any faithful linear functional on 
${\mathcal B}$ (or on ${\mathcal C}$) is equipped with a KMS-type automorphism:

\begin{prop} \label{sepid_KMS}
\begin{enumerate}
\item Any faithful linear functional on ${\mathcal B}$ has the weak KMS property.
\item Any faithful linear functional on ${\mathcal C}$ has the weak KMS property.
\end{enumerate}
\end{prop}

\begin{proof}
See Proposition~1.3 of \cite{VD_aqg}.  The case for ${\mathcal B}$ is below.  The case for ${\mathcal C}$ is similar.

If $f$ is a faithful linear functional on ${\mathcal B}$, then by Proposition~\ref{sepid_RN} we can write it as 
$f(\ ):b\mapsto\nu(by)$, where $y$ is invertible in $M({\mathcal B})$.  Then for any $b,b'\in{\mathcal B}$, we have:
$$
f(bb')=\nu(bb'y)=\nu\bigl(b'y\sigma^{\nu}(b)\bigr)=f\bigl(b'y\sigma^{\nu}(b)y^{-1}\bigr).
$$
This shows that $b\mapsto y\sigma^{\nu}(b)y^{-1}$ determines a KMS-type automorphism for $f$. 
\end{proof}

\subsection{Algebraic quantum groupoid}\label{sub1.4}

Let ${\mathcal A}$ be a non-degenerate idempotent ${}^*$-algebra, and $\Delta$ a regular, full comultiplication 
that is weakly non-degenerate (See \S\ref{sub1.2}). This means the existence of a unique canonical idempotent 
element $E\in M({\mathcal A}\odot{\mathcal A})$, which is also assumed to satisfy the weak comultiplicativity 
of the unit, or Equation~\eqref{(weakcomultiplicativity)}.

We will further assume that $E$ is a regular, full separability idempotent (See \S\ref{sub1.3}), such that there 
exist two ${}^*$-subalgebras ${\mathcal B}$ and ${\mathcal C}$ of $M({\mathcal A})$ sitting in a non-degenerate 
way, and we have $E\in M({\mathcal B}\odot{\mathcal C})$.  When we say the subalgebras ${\mathcal B}$ and 
${\mathcal C}$ sit in a non-degenerate way inside $M({\mathcal A})$, it means that ${\mathcal B}{\mathcal A}
={\mathcal A}{\mathcal B}={\mathcal A}$, and ${\mathcal C}{\mathcal A}={\mathcal A}{\mathcal C}={\mathcal A}$. 
Then it is easy to see that ${\mathcal B}$ and ${\mathcal C}$ are non-generate algebras, and that $M({\mathcal B})$ 
and $M({\mathcal C})$ can be regarded as subalgebras in $M({\mathcal A})$ as well. 

\begin{rem}
It turns out that the subalgebras ${\mathcal B}$ and ${\mathcal C}$ are completely determined by the conditions 
given above. See Proposition~3.1 of \cite{BJKVD_LSthm}.  As such, it would be all right to just say that ``$E$ is 
a regular separability idempotent'', without having to specify ${\mathcal B}$ and ${\mathcal C}$ explicitly.
\end{rem}  

On $M({\mathcal B})$ and $M({\mathcal C})$, regarded as subalgebras in $M({\mathcal A})$, it turns out that 
the comultiplication acts the following way (which tuns out to characterize the subalgebras):

\begin{prop}\label{DeltaonBandC}
If $x\in M({\mathcal B})$, then we have $\Delta x=E(1\otimes x)=(1\otimes x)E$.

If $y\in M({\mathcal C})$, then we have $\Delta y=(y\otimes 1)E=E(y\otimes 1)$.
\end{prop}

\begin{proof}
See Proposition~3.4 in \cite{BJKVD_LSthm}. 
\end{proof}

Let us next give the definition for left and right integrals.  See Definition~3.5 of \cite{BJKVD_LSthm}. 

\begin{defn}\label{invariantintegrals}
A linear functional $\varphi$ on ${\mathcal A}$ is said to be {\em left invariant\/}, if 
$$
(\operatorname{id}\otimes\varphi)(\Delta a)\in M({\mathcal C}), \quad \forall a\in{\mathcal A}.
$$
Similarly, a linear functional $\psi$ on $A$ is said to be {\em right invariant\/}, if 
$$
(\psi\otimes\operatorname{id})(\Delta a)\in M({\mathcal B}), \quad \forall a\in{\mathcal A}.
$$
Any non-zero left-invariant linear functional on ${\mathcal A}$ is called a {\em left integral\/}, and any non-zero 
right-invariant linear functional on ${\mathcal A}$ is called a {\em right integral\/}.

\end{defn}

\begin{rem}
While $\Delta a\in M({\mathcal A}\odot{\mathcal A})$, we can still make sense of $(\operatorname{id}\otimes\varphi)(\Delta a)$  
as an element in $M({\mathcal A})$.  To see this, remember that $(b\otimes1)(\Delta a)\in{\mathcal A}\odot{\mathcal A}$, for all $a,b\in{\mathcal A}$. 
For any $b\in{\mathcal A}$, we may regard 
$b(\operatorname{id}\otimes\varphi)(\Delta a)=(\operatorname{id}\otimes\varphi)\bigl((b\otimes1)(\Delta a)\bigr)\in{\mathcal A}$. 
What the left invariance of $\varphi$ means is that $(\operatorname{id}\otimes\varphi)(\Delta a)$ is actually contained in 
$M({\mathcal C})\,\subset M({\mathcal A})$. Similarly, we can also make sense of $(\psi\otimes\operatorname{id})(\Delta a)\in M({\mathcal B})$.
\end{rem}

By the result of Proposition~\ref{DeltaonBandC}, we can actually sharpen the statements in Definition~\ref{invariantintegrals}:

\begin{prop}\label{sharpinvariance}
Let $\varphi$ and $\psi$ are, respectively, left and right integrals on ${\mathcal A}$.  Then we have:
$$
(\operatorname{id}\otimes\varphi)(\Delta x)\in{\mathcal C}, \quad \forall x\in{\mathcal A},
$$
$$
(\psi\otimes\operatorname{id})(\Delta x)\in{\mathcal B}, \quad \forall x\in{\mathcal A}.
$$
\end{prop}

\begin{proof}
Let $a\in{\mathcal A}$ and let $c\in{\mathcal C}$.  By a reasoning similar to the one given in Remark above, it is easy to see that we can make sense 
of $(\operatorname{id}\otimes\varphi)\bigl((\Delta a)(c\otimes1)\bigr)$.  By the left invariance, we know $(\operatorname{id}\otimes\varphi)(\Delta a)
\in M({\mathcal C})$.  Therefore, we have: $(\operatorname{id}\otimes\varphi)\bigl((\Delta a)(c\otimes1)\bigr)
=(\operatorname{id}\otimes\varphi)(\Delta a)c\in{\mathcal C}$.

But by Proposition~\ref{DeltaonBandC}, we know $\Delta c=E(c\otimes1)$.  So we have $(\Delta a)(c\otimes1)=(\Delta a)E(c\otimes1)=
(\Delta a)(\Delta c)=\Delta(ac)$.  Considering the result from the above, we conclude that 
$(\operatorname{id}\otimes\varphi)\bigl(\Delta(ac)\bigr)=(\operatorname{id}\otimes\varphi)\bigl((\Delta a)(c\otimes1)\bigr)\in{\mathcal C}$.
This is true for any $a\in{\mathcal A}$ and any $c\in{\mathcal C}$.  But then, as ${\mathcal C}$ sits in a non-degenerate way inside 
$M({\mathcal A})$, or ${\mathcal A}{\mathcal C}={\mathcal A}$, we know that any $x\in{\mathcal A}$ can be expressed as a finite sum 
of the elements of the form $ac$, where $a\in{\mathcal A}$, $c\in{\mathcal C}$.  It follows that 
$(\operatorname{id}\otimes\varphi)(\Delta x)\in{\mathcal C}$, for any $x\in{\mathcal A}$.

The proof for $(\psi\otimes\operatorname{id})(\Delta x)\in{\mathcal B}$, $x\in{\mathcal A}$, can be done in exactly the same way for a right invariant 
functional $\psi$.
\end{proof}

Here are some additional consequences of the left/right invariance of $\varphi$ and $\psi$:

\begin{prop}\label{invarianceF}
Denote
$$
F_1=(\operatorname{id}\otimes S)(E), \quad F_2=(S\otimes\operatorname{id})(E), \quad
F_3=(\operatorname{id}\otimes S^{-1})(E), \quad F_4=(S^{-1}\otimes\operatorname{id})(E),
$$
which are elements in $M({\mathcal A}\odot{\mathcal A})$. 
If $\varphi$ is a left integral and if $\psi$ is a right integral, we have:
$$
(\operatorname{id}\otimes\varphi)(\Delta a)=(\operatorname{id}\otimes\varphi)\bigl(F_2(1\otimes a)\bigr)
=(\operatorname{id}\otimes\varphi)\bigl((1\otimes a)F_4\bigr),
$$
$$
(\psi\otimes\operatorname{id})(\Delta a)=(\psi\otimes\operatorname{id})\bigl((a\otimes1)F_1\bigr)
=(\psi\otimes\operatorname{id})\bigl(F_3(a\otimes1)\bigr),
$$
for all $a\in{\mathcal A}$.
\end{prop}

\begin{proof}
See Proposition~1.4 of \cite{VDWangwha3} and Proposition~3.7 of \cite{BJKVD_LSthm}.
\end{proof}

Proposition~\ref{muphinupsi} below gives a relationship between the integrals $\varphi$, $\psi$ and 
the expressions $(\psi\otimes\operatorname{id})(\Delta x)$ and $(\operatorname{id}\otimes\varphi)(\Delta x)$, 
for $x\in{\mathcal A}$.  These result hold true for any left integral and any right integral.  The proof is fundamentally 
no different than the one given in Proposition~4.9 of \cite{BJKVD_qgroupoid1}, but now done at the ${}^*$-algebra level.

\begin{prop}\label{muphinupsi}
For any right integral $\psi$ and any left integral $\varphi$, we have:
\begin{itemize}
  \item $\nu\bigl((\psi\otimes\operatorname{id})(\Delta x)\bigr)=\psi(x)$, for $x\in{\mathcal A}$.
  \item $\mu\bigl((\operatorname{id}\otimes\varphi)(\Delta x)\bigr)=\varphi(x)$, for $x\in{\mathcal A}$.
\end{itemize}
\end{prop}

\begin{proof}
By Definition~\ref{invariantintegrals} and Proposition~\ref{sharpinvariance}, we know $(\psi\otimes\operatorname{id})(\Delta x)\in{\mathcal B}$ 
and $(\operatorname{id}\otimes\varphi)(\Delta x)\in{\mathcal C}$, for any $x\in{\mathcal A}$.  As such, both expressions on the 
left sides above make sense.

Let $a\in{\mathcal A}$.  Consider $(\psi\otimes\operatorname{id})(\Delta a)$, then apply $\Delta$.  On the one hand, we have:
$$
\Delta\bigl((\psi\otimes\operatorname{id})(\Delta a)\bigr)
=(\psi\otimes\operatorname{id}\otimes\operatorname{id})\bigl((\operatorname{id}\otimes\Delta)(\Delta a)\bigr)
=(\psi\otimes\operatorname{id}\otimes\operatorname{id})\bigl((\Delta\otimes\operatorname{id})(\Delta a)\bigr),
$$
where we used the coassociativity of $\Delta$.  Meanwhile, by Proposition~\ref{DeltaonBandC}, 
we have: 
$$
\Delta\bigl((\psi\otimes\operatorname{id})(\Delta a)\bigr)
=E\bigl(1\otimes(\psi\otimes\operatorname{id})(\Delta a))
=(\psi\otimes\operatorname{id}\otimes\operatorname{id})\bigl((1\otimes E)\Delta_{13}(a)\bigr).
$$
It thus follows that 
\begin{equation}\label{(muphinupsi1)}
(\psi\otimes\operatorname{id}\otimes\operatorname{id})\bigl((\Delta\otimes\operatorname{id})
(\Delta a)\bigr)=(\psi\otimes\operatorname{id}\otimes\operatorname{id})
\bigl((1\otimes E)\Delta_{13}(a)\bigr).
\end{equation}

Let $y=\tilde{y}c$, where $\tilde{y}\in{\mathcal A}$ and $c\in{\mathcal C}$.  Recall that ${\mathcal A}{\mathcal C}
={\mathcal A}$, so such elements span ${\mathcal A}$.  Multiply $1\otimes y=1\otimes\tilde{y}c$ to both sides 
of Equation~\eqref{(muphinupsi1)}, from left.  Then the equation becomes: 
$$
(\psi\otimes\operatorname{id}\otimes\operatorname{id})
\bigl((\Delta\otimes\operatorname{id})((1\otimes\tilde{y}c)(\Delta a))\bigr)
=(\psi\otimes\operatorname{id}\otimes\operatorname{id})
\bigl((1\otimes1\otimes\tilde{y}c)(1\otimes E)\Delta_{13}(a)\bigr).
$$
Let $\omega\in{\mathcal A}^*$, and apply $\operatorname{id}\otimes\omega$ to the equation above.  Then 
it becomes: 
\begin{align}
&(\psi\otimes\operatorname{id})\bigl(\Delta((\operatorname{id}\otimes\omega)
[(1\otimes\tilde{y}c)(\Delta a)])\bigr)   \notag \\
&=(\psi\otimes\operatorname{id})\bigl((\operatorname{id}\otimes\operatorname{id}
\otimes\omega)[(1\otimes1\otimes\tilde{y})(1\otimes1\otimes c)(1\otimes E)
\Delta_{13}(a)]\bigr).
\notag
\end{align}
Apply $\nu$ to both sides.  By the property of $\nu$, we know $(\nu\otimes\operatorname{id})(E)=1$. 
So we have:
\begin{align}
&\nu\bigl((\psi\otimes\operatorname{id})(\Delta((\operatorname{id}
\otimes\omega)[(1\otimes\tilde{y}c)(\Delta a)]))\bigr)   \notag \\
&=\psi\bigl((\operatorname{id}\otimes\omega)[(1\otimes\tilde{y})(1\otimes
(\nu\otimes\operatorname{id})((1\otimes c)E))(\Delta a)]\bigr)   \notag \\
&=\psi\bigl((\operatorname{id}\otimes\omega)[(1\otimes\tilde{y}c)(\Delta a)]\bigr).
\label{(muphinupsi2)}
\end{align}

For convenience, write: $x=(\operatorname{id}\otimes\omega)[(1\otimes\tilde{y}c)(\Delta a)]$.  
By the ``fullness'' of the comultiplication, such elements span all of ${\mathcal A}$. Then 
Equation~\eqref{(muphinupsi2)} can be expressed as $\nu\bigl((\psi\otimes\operatorname{id})(\Delta x)\bigr)
=\psi(x)$, which would be true for all $x\in{\mathcal A}$.

Similarly, we can show that for any left integral $\varphi$, we have 
$\mu\bigl((\operatorname{id}\otimes\varphi)(\Delta x)\bigr)=\varphi(x)$, for $x\in{\mathcal A}$.
\end{proof}

In general, while we can have a faithful left integral and a faithful right integral on ${\mathcal A}$, it may have neither.  Instead, it may be 
possible to have a faithful set of left integrals $\{\varphi_{\alpha}\}$, in the sense that if $a\in{\mathcal A}$ is such that $\varphi_{\alpha}(ax)=0$ 
for all $x\in{\mathcal A}$ and for all left integrals $\phi_{\alpha}$, then we must have $a=0$.  Similarly, if $\varphi_{\alpha}(xa)=0$ for all 
$x\in{\mathcal A}$ and for all $\phi_{\alpha}$, then we must have $a=0$. We can make sense of a faithful set of right integrals in a similar way.

The following proposition is a consequence of having a faithful set of left integrals and a faithful set of right integrals. Again, 
proof is skipped.

\begin{prop}\label{faithfulintegrals}
\begin{enumerate}
\item If there is a faithful set of left integrals, then 
$$
{\mathcal A}=\operatorname{span}\bigl\{(\operatorname{id}\otimes\varphi)((\Delta a)(1\otimes b)):{\text { $\varphi$ is a left integral}},
a,b\in{\mathcal A}\bigr\}.
$$
\item If there is a faithful set of right integrals, then 
$$
{\mathcal A}=\operatorname{span}\bigl\{(\psi\otimes\operatorname{id})((b\otimes1)(\Delta a)):{\text { $\psi$ is a left integral}},
a,b\in{\mathcal A}\bigr\}.
$$
\item If there is a faithful set of left integrals, then
$$
\Delta({\mathcal A})(1\odot{\mathcal A})=E({\mathcal A}\odot{\mathcal A}) \quad {\text { and }} \quad
(1\odot{\mathcal A})\Delta({\mathcal A})=({\mathcal A}\odot{\mathcal A})E.
$$
\item If there is a faithful set of right integrals, then
$$
({\mathcal A}\odot1)\Delta({\mathcal A})=({\mathcal A}\odot{\mathcal A})E \quad {\text { and }} \quad
\Delta({\mathcal A})({\mathcal A}\odot1)=E({\mathcal A}\odot{\mathcal A}).
$$
\end{enumerate}
\end{prop}

The main result of \cite{BJKVD_LSthm} is that given data $({\mathcal A},\Delta,E)$ as above, if there is a faithful set of left integrals 
and a faithful set of right integrals, then we have a regular weak multiplier Hopf algebra, in the sense of \cite{VDWangwha1}.  In particular, 
the existence of the antipode, $S$, can be proved, using the result of Proposition~\ref{faithfulintegrals}. One can also construct the counit, 
$\varepsilon$.  It then becomes equivalent to the situation of a regular weak multiplier Hopf algebra equipped with a faithful set of left integrals, 
whose right integrals can be obtained using the antipode map. 

In \cite{VDWangwha3}, \cite{VD_aqg}, regular weak Hopf algebras with a faithful set of integrals is referred to as {\em algebraic quantum groupoids\/}. 
It is shown there that they form a self-dual category.  Note, however, that this notion is different from Timmermann's notion of an algebraic quantum 
groupoid (see \cite{Timm_aqgintegral}, \cite{Timm_aqgdual}), which is based on the framework of multiplier Hopf algebroids.  Some discussion 
on the relationship between these two frameworks can be found in \cite{TimmVD_wmhaalgebroid}.

\subsection{Weak multiplier Hopf ${}^*$-algebra with a single faithful integral}\label{sub1.5}

According to the general theory on weak multiplier Hopf algebras, the existence of a faithful family of (left) integrals is required 
for the duality picture to be complete. See \cite{VDWangwha3}.  Unlike in the case of multiplier Hopf algebras (see \cite{VD_multiplierHopf}, 
\cite{VD_multiplierHopfduality}), having a left or right integral does not necessarily mean that it is also faithful.  There are known examples 
of weak (multiplier) Hopf algebras where enough integrals exist to form a faithful family, but not a single faithful one \cite{nosinglefaithfulintegral}.

Having said this, in many examples there exists a single faithful integral.  In particular, it has been observed that for finite-dimensional 
weak Hopf algebras, a single faithful integral exists if and only if the underlying algebra is a Frobenius algebra, for instance a finite-dimensional 
$C^*$-algebra (See Theorem~3.16 in \cite{BNSwha1}.).  Infinite-dimensional case is not fully understood.  Nonetheless, 
considering that our aim is to eventually construct a $C^*$-algebraic version, this observation seems to suggest that it may 
not be too restrictive to require the existence of a single faithful integral. We will do so here.  Also, as we are working 
in the ${}^*$-algebra setting, we will further require that the integrals are positive linear functionals. 

\begin{defn}\label{aqgdefn}
Let $({\mathcal A},\Delta,E)$ be as above, and assume that there exists a single positive faithful left integral $\varphi$ 
and a single positive faithful right integral $\psi$. We will call $({\mathcal A},\Delta)$ a {\em weak multiplier Hopf 
${}^*$-algebra with a faithful integral\/}.
\end{defn}

\begin{rem}
The positivity of the functionals $\varphi$ and $\psi$ is not automatic, even under the ${}^*$-structure and the self-adjointness 
of $E$.  There are examples even in the case of ${}^*$-algebraic quantum groups of compact type, for which the left invariant 
functional is not positive.  As such, we explicitly require here the positivity of the invariant functionals.  
\end{rem}

As noted earlier, having the ${}^*$-structure means this is a regular weak multiplier Hopf algebra.  Meanwhile, 
having a single faithful left integral and a single faithful right integral just means that we have one-element 
families of left/right integrals, so the main results outlined in the previous subsection still hold, including the 
existence of the counit and the antipode: 

\begin{thm}\label{LS}
Let ${\mathcal A}$ be a non-degenerate idempotent ${}^*$-algebra, with a full comultiplication $\Delta:{\mathcal A}\to M({\mathcal A}\odot
{\mathcal A})$ as defined in \S\ref{sub1.2}. Assume also that there exists the canonical idempotent $E\in M({\mathcal A}\odot{\mathcal A})$, 
satisfying the properties given in \S\ref{sub1.2} and \S\ref{sub1.3}. In addition, we assume that there exists a single positive faithful left integral 
$\varphi$ and a single positive faithful right integral $\psi$.

Then $({\mathcal A},\Delta)$ becomes a regular weak multiplier Hopf ${}^*$-algebra, in the sense of Van Daele and Wang \cite{VDWangwha1}. 
In particular, the counit $\varepsilon$ and the antipode $S$ exisits.
\end{thm} 

\begin{proof}
See Theorem 3.15 of \cite{BJKVD_LSthm}.  See also the last two paragraphs of Section~\S\ref{sub1.4} above.
\end{proof}

We will gather some additional results before we end this subsection, which we will use down the road.
Next proposition summarizes a characterization of the antipode map, with its properties:

\begin{prop}\label{antipodeS}
\begin{enumerate}
\item The antipode map, $S$, is anti-multiplicative and bijective on $A$, and can be characterized as follows:
$$
S:(\operatorname{id}\otimes\varphi)\bigl((\Delta a)(1\otimes b)\bigr)
\mapsto(\operatorname{id}\otimes\varphi)\bigl((1\otimes a)(\Delta b)\bigr), \quad \forall a,b\in{\mathcal A}.
$$
\item Similarly, the antipode $S$ can be also characterized as follows:
$$
S:(\psi\otimes\operatorname{id})\bigl((a\otimes1)(\Delta b)\bigr)
\mapsto(\psi\otimes\operatorname{id})\bigl((\Delta a)(b\otimes1)\bigr), \quad \forall a,b\in{\mathcal A}.
$$
\item The antipode $S$ of $({\mathcal A},\Delta)$ can be extended to the multiplier algebra $M({\mathcal A})$, and when 
restricted to ${\mathcal B}$ and ${\mathcal C}$, it coincides with the maps $S_{\mathcal B}$ on ${\mathcal B}$ and $S_{\mathcal C}$ 
on ${\mathcal C}$.
\item We have: $(S\otimes S)(\Delta x )=\Delta^{\operatorname{cop}}\bigl(S(x)\bigr)=\varsigma\Delta\bigl(S(x)\bigr)$, for $x\in{\mathcal A}$, 
where $\varsigma$ is the flip map on $M({\mathcal A}\odot{\mathcal A})$.
\item With respect to the ${}^*$-structure, we have: $S\bigl(S(x)^*\bigr)^*=x$, for $x\in{\mathcal A}$.
\end{enumerate}
\end{prop}

\begin{proof}
For the characterizations (1), (2), see Proposition~1.5 of \cite{VDWangwha3}, which are true not just for faithful integrals but for any 
left integral and any right integral.  See also Proposition~3.16 of  \cite{BJKVD_LSthm}.  In view of the earlier Proposition~\ref{faithfulintegrals}, 
we see that the antipode is bijective. 

For (3), see Proposition~3.18 of \cite{BJKVD_LSthm}.  For (4) and (5), see Propositions~4.4 and 4.11 of \cite{VDWangwha1}.
For more results on the antipode on a regular multiplier Hopf algebra, see Section~4 of \cite{VDWangwha1}.
\end{proof}

Even though a characterization for the antipode, $S$, is given above in terms of the left integral $\varphi$ (and alternatively 
in terms of the right integral $\psi$), it is actually intrinsic.  Indeed, the antipode $S$ is independent of any particular choice 
of the integrals $\varphi$ and $\psi$.  This seemingly ambiguous aspect is because the invariant functionals are not uniquely determined. 
Nonetheless, the results (1) and (2) of Proposition~\ref{antipodeS} mean that there exists a weak sense of uniqueness for the left/right
integrals.  More discussions are given below.

Here are some additional consequences of having a single faithful integral $\varphi$. The following are general results 
(see section~1 of \cite{VDWangwha3}), which do not require the positivity of the functionals.  As these are purely algebraic results that can be 
found in \cite{VDWangwha3}, we will skip the details and the proofs.

\begin{prop}\label{phi_modularautomorphism}
There exists an automorphism $\sigma$ of ${\mathcal A}$ such that 
$$
\varphi(ab)=\varphi\bigl(b\sigma(a)\bigr), \quad \forall a,b\in{\mathcal A}.
$$
We also have $\varphi\circ\sigma=\varphi$.
\end{prop}

\begin{proof}
See Proposition~1.7 of \cite{VDWangwha3} for the existence of $\sigma$.
\end{proof}

The automorphism $\sigma$ above will be referred to as the  {\em modular automorphism for $\varphi$\/}.  The terminology is motivated by 
the theory of weights on $C^*$-algebras.  Note that while any faithful linear functional on a finite-dimensional algebra admits such 
an automorphism, that is not always the case in the infinite-dimensional case. So the above result indicates that there is more going on 
with $\varphi$ than it being just a faithful linear functional. 

The next result gives a relation between two left integrals:

\begin{prop}\label{phiandotherleft}
Let $\varphi$ be a faithful left integral, and let $\varphi_1$ be a left integral (not necessarily faithful).  Then there is an element $y\in M({\mathcal B})$ 
such that $\varphi_1(x)=\varphi(xy)$, for all $x\in{\mathcal A}$.
\end{prop}

\begin{proof}
See Proposition~1.8 of \cite{VDWangwha3}.
\end{proof}

Going the other way, it is not difficult to show that for any left integral $\varphi$ and $y\in M({\mathcal B})$, the functionals $\varphi(\,\cdot\,y)$ 
and $\varphi(y\,\cdot\,)$ again become left invariant.  As a consequence, we can see that $\sigma$ leaves $M({\mathcal B})$ invariant. 

By a similar argument used to prove the above results, we obtain the following proposition, relating any right integral $\psi_1$ 
with the faithful left integral $\varphi$: 

\begin{prop}\label{modular}
Let $\varphi$ be a faithful left integral, and let $\psi_1$ be a right integral (not necessarily faithful). 
Then there exists an element $\delta_1\in M({\mathcal A})$ such that $\psi_1(x)=\varphi(x\delta_1)$ for all $x\in{\mathcal A}$.

In particular, if $\psi_1$ is also faithful, then $\delta_1$ is invertible in $M({\mathcal A})$.
\end{prop}

\begin{proof}
See Proposition~1.9 of \cite{VDWangwha3}.
\end{proof}

Of particular interest is the {\em modular element\/} $\delta$, which relates the left integral $\varphi$ with the functional $\varphi\circ S$, 
which turns out to be right invariant.  Namely, we have $(\varphi\circ S)(x)=\varphi(x\delta)$, for $x\in{\mathcal A}$.  By the faithfulness of $\varphi$, 
it is not difficult to show that the modular element is invertible and is uniquely determined.  See Section~\ref{appx} (Appendix), where some 
more results on $\delta$ are gathered.

\subsection{The dual algebra}\label{sub1.6}
 
Let $({\mathcal A},\Delta)$ be a (regular) weak multiplier Hopf ${}^*$-algebra, with its counit $\varepsilon$ and 
the antipode $S$, as in \cite{VDWangwha1}.  Suppose there exists a faithful left integral $\varphi$.  Then by 
the general theory, we can construct its dual object.  For details, see Section~2 of \cite{VDWangwha3}. First, 
we consider $\widehat{\mathcal A}$, the space of linear functionals on ${\mathcal A}$ spanned by the elements 
of the form $\varphi(\,\cdot\,a)$, for $a\in{\mathcal A}$.  It can be given a weak multiplier Hopf ${}^*$-algebra 
structure, as follows. 

For $\omega,\omega'\in\widehat{\mathcal A}$ and $x\in{\mathcal A}$, define the multiplication 
$\omega\omega'\in\widehat{\mathcal A}$ by 
$$
(\omega\omega')(x):=(\omega\otimes\omega')(\Delta x).
$$
For $\omega\in\widehat{\mathcal A}$ and $x\in{\mathcal A}$, define the involution on $\widehat{\mathcal A}$ 
by $\omega^*(x):=\overline{\omega\bigl(S(x)^*\bigr)}$. One can show that $\widehat{\mathcal A}$ becomes 
a non-degenerate idempotent ${}^*$-algebra.

As for the comultiplication, we define $\widehat{\Delta}$ on $\widehat{\mathcal A}$ in such a way that 
$$
\widehat{\Delta}(\omega)(x\otimes y)=\omega(xy), \quad {\text { for $\omega
\in\widehat{\mathcal A}$, $x,y\in{\mathcal A}$.}}
$$
It becomes a full coassociative comultiplication.  However, making sense of this needs some care, as we need 
to consider $M(\widehat{\mathcal A}\odot\widehat{\mathcal A})$ inside the dual space $({\mathcal A}\odot{\mathcal A})^*$ 
in a proper way.  See Propositions~2.7, 2.8, 2.9 of \cite{VDWangwha3}.

The antipode map, $\widehat{S}:\widehat{\mathcal A}\to\widehat{\mathcal A}$, is given by
$$
\widehat{S}(\omega)(x)=\omega\bigl(S(x)\bigr), \quad {\text { for $\omega
\in\widehat{\mathcal A}$, $x\in{\mathcal A}$.}}
$$
Meanwhile, the canonical idempotent $\widehat{E}\in M({\mathcal A}\odot{\mathcal A})$ should be such that 
$\widehat{E}=\widehat{\Delta}\bigl(1_{M(\widehat{A})}\bigr)$, so it is defined by 
$$
\widehat{E}(x\otimes y)=\varepsilon(xy), \quad {\text { for $x,y\in{\mathcal A}$,}}
$$
where $\varepsilon$ is the counit.  It can be shown that these structure maps make $(\widehat{\mathcal A},\widehat{\Delta})$ 
a regular weak multiplier Hopf ${}^*$-algebra.  See Theorem~2.15 of \cite{VDWangwha3}.

Finally, as we are assuming that $({\mathcal A},\Delta)$ is equipped with a single faithful integral, it can be shown 
that $(\widehat{\mathcal A},\widehat{\Delta})$ is also equipped with a single faithful integral. (See Theorem~2.21 
of \cite{VDWangwha3}.) We will skip the details, but if $\varphi$ is a faithful left integral for $({\mathcal A},\Delta)$ 
and $\omega=\varphi(\,\cdot\,c)\in\widehat{\mathcal A}$, $c\in{\mathcal A}$, then one can consider the functional 
$\widehat{\psi}$ on $\widehat{\mathcal A}$, such that $\widehat{\psi}(\omega):=\varepsilon(c)$.  It can be shown that 
such a functional $\widehat{\psi}$ becomes a faithful right integral on $(\widehat{\mathcal A},\widehat{\Delta})$.

To summarize, from a regular multiplier Hopf algebra $({\mathcal A},\Delta)$ equipped with a faithful left integral $\varphi$, 
one can construct its dual $(\widehat{\mathcal A},\widehat{\Delta})$, which is also a regular multiplier Hopf algebra, 
with a faithful right integral $\widehat{\psi}$.  A main theorem is that if we consider the dual of $(\widehat{\mathcal A},
\widehat{\Delta})$, then the resulting object $(\widehat{\widehat{\mathcal A}},\widehat{\widehat{\Delta}})$ is canonically 
isomorphic to the original $({\mathcal A},\Delta)$, which is a generalized Pontryagin duality.  See Section~2 of \cite{VDWangwha3}.

\section{The base $C^*$-algebras}\label{sec2}

From this point on, we are going to systematically construct a {\em $C^*$-algebraic quantum groupoid of separable type\/} 
in the sense of \cite{BJKVD_qgroupoid1}, \cite{BJKVD_qgroupoid2}, out of our purely algebraic object of a weak multiplier 
Hopf ${}^*$-algebra $({\mathcal A},\Delta)$ equipped with a faithful left integral $\varphi$ (and also a right integral $\psi$).  
We will use the notations and properties summarized in Section~\ref{sec1}, including the subalgebras ${\mathcal B}$ and 
${\mathcal C}$, the distinguished linear functionals $\nu$ and $\mu$ on them, the canonical idempotent $E$, and the antipode $S$. 

\subsection{Construction of the  $C^*$-algebras $B$ and $C$}\label{sub2.1}

We will begin with the construction of the {\em base\/} $C^*$-algebras $B$ and $C$, which will be essentially the ``source algebra'' and 
the ``target algebra''.  This will be done by completing the algebras ${\mathcal B}$ and ${\mathcal C}$ in an appropriate sense.

First, recall (from Section~\ref{sub1.3}) that there exists a distinguished linear functional $\nu$ on the ${}^*$-algebra 
${\mathcal B}$, which is positive and faithful.  Using $\nu$, we can provide ${\mathcal B}$ with an inner product, 
as follows:
$$
\langle x_1,x_2\rangle:=\nu(x_2^*x_1), \quad {\text { for $x_1,x_2\in{\mathcal B}$.}}
$$
Form the completion of ${\mathcal B}$ with respect to the induced norm, and obtain a Hilbert space ${\mathcal H}_B$ 
with the natural inclusion $\Lambda_B:{\mathcal B}\to{\mathcal H}_B$.

In a similar way, using the distinguished linear functional $\mu$ on the ${}^*$-algebra ${\mathcal C}$, we can equip 
${\mathcal C}$ with an inner product, by $\langle y_1,y_2\rangle:=\mu(y_2^*y_1)$, for $y_1,y_2\in{\mathcal C}$. 
As above, we obtain a Hilbert space ${\mathcal H}_C$ with the natural inclusion 
$\Lambda_C:{\mathcal C}\to{\mathcal H}_C$.

Between the Hilbert spaces ${\mathcal H}_B$ and ${\mathcal H}_C$, there exists a unitary map given by the 
anti-homomorphism $S_{\mathcal B}$. More precisely, note that for any $b_1,b_2\in{\mathcal B}$, we have:
\begin{align}
\mu\bigl(S_{\mathcal B}(b_1)^*S_{\mathcal B}(b_2)\bigr)
&=\mu\bigl(S_{\mathcal C}^{-1}(b_1^*)S_{\mathcal B}(b_2)\bigr)
=\mu\bigr(S_{\mathcal B}(b_2)S_{\mathcal B}(b_1^*)\bigr)
\notag \\
&=(\mu\circ S_{\mathcal B})(b_1^*b_2)=\nu(b_1^*b_2),
\notag
\end{align}
because $S_{\mathcal C}\bigl(S_{\mathcal B}(b_1)^*\bigr)^*=b_1$ and $\sigma_{\mathcal C}=S_{\mathcal B}\circ S_{\mathcal C}$, 
while $\mu\circ S_{\mathcal B}=\nu$.  This shows that $S_{\mathcal B}:{\mathcal B}\to{\mathcal C}$ provides an isometry 
with respect to the inner products on ${\mathcal H}_B$ and ${\mathcal H}_C$.   So it lifts to a unitary $\widehat{J}_B:{\mathcal H}_B
\to{\mathcal H}_C$, by $\widehat{J}_B\Lambda_B(x)=\Lambda_C\bigl(S_{\mathcal B}(x)\bigr)$, $\forall x\in{\mathcal B}$. 

We can formulate the Hilbert space tensor product ${\mathcal H}_B\otimes{\mathcal H}_C$.  On this Hilbert space, note that our canonical 
idempotent $E\in M({\mathcal B}\odot{\mathcal C})$ naturally defines an operator, $\Pi(E)$, by
$$
\Pi(E)\bigl(\Lambda_B(x)\otimes\Lambda_C(y)\bigr)=(\Lambda_B\otimes\Lambda_C)
\bigl(E(x\otimes y)\bigr), \quad {\text { for $x\in{\mathcal B}$, $y\in{\mathcal C}$.}}
$$
As ${\mathcal B}$ and ${\mathcal C}$ are dense in ${\mathcal H}_B$ and ${\mathcal H}_C$, respectively, it is clear that 
$\Pi(E)$ is a densely-defined operator acting on ${\mathcal H}_B\otimes{\mathcal H}_C$.  By the property of $E$, it is 
also clear that $\Pi(E)$ is self-adjoint and idempotent.  This means $\Pi(E)$ extends to an orthogonal projection, so 
it is a bounded self-adjoint operator in ${\mathcal B}({\mathcal H}_B\otimes{\mathcal H}_C)$.

Consider an element of the form $x=(\operatorname{id}\otimes\omega)(E)\in{\mathcal B}$, where $\omega\in{\mathcal C}^*$ 
is defined by $\omega=\mu(c_1^*\,\cdot\,c_2)$, for $c_1,c_2\in{\mathcal C}$. We saw in \S\ref{sub1.3} that such elements 
span the algebra ${\mathcal B}$.  In terms of the bounded operator $\Pi(E)$ above, it is clear that for this $x$, we can consider 
the operator $(\operatorname{id}\otimes\omega_{\Lambda_C(c_2),\Lambda_C(c_1)})\bigl(\Pi(E)\bigr)\in{\mathcal B}
({\mathcal H}_B)$, such that
$$
(\operatorname{id}\otimes\omega_{\Lambda_C(c_2),\Lambda_C(c_1)})\bigl(\Pi(E)\bigr)\Lambda_B(b)
=\Lambda_B(xb), \quad {\text { for all $b\in{\mathcal B}$.}}
$$
Here, we are using the standard notation that $\omega_{\xi,\eta}(T)=\langle T\xi,\eta\rangle$, for $T\in{\mathcal B}({\mathcal H}_C)$ 
and $\xi,\eta\in{\mathcal H}_C$.  In other words, any element of the form $x=(\operatorname{id}\otimes\omega)(E)\in{\mathcal B}$ can be 
regarded as a bounded operator on ${\mathcal H}_B$ by left multiplication.  Since such elements span all of ${\mathcal B}$, the same can be 
said true for any element in ${\mathcal B}$. This allows us to define the GNS-representation $\pi_B$ of $\nu$:

\begin{defn}
Define $\pi_B$ from ${\mathcal B}$ into ${\mathcal B}({\mathcal H}_B)$, by
$$
\pi_B(x)\Lambda_B(b)=\Lambda(xb), \quad {\text { for all $x,b\in{\mathcal B}$.}}
$$
Then $\pi_B$ is an injective ${}^*$-homomorphism, which is the GNS-representation of ${\mathcal B}$, such that 
$\pi_B({\mathcal B}){\mathcal H}_B$ is dense in ${\mathcal H}_B$.
\end{defn}

\begin{rem}
In general, even if we have a positive linear functional on a ${}^*$-algebra ${\mathcal B}$, resulting in an inner product 
and a Hilbert space ${\mathcal H}_B$, it is not always possible to represent the algebra as an algebra of left multiplication 
operators. Some elements may become unbounded operators.  Observe that in our case, the existence of our self-adjoint 
idempotent $E$ allowed the construction of the GNS-representation.  Meanwhile, note that the density statement 
in the last sentence of the definition is a quick consequence of the fact that ${\mathcal B}$ is a non-degenerate 
idempotent algebra.  We can also extend $\pi_B$ to the level of $M({\mathcal B})$.
\end{rem}

By a similar argument, using $E$ and considering its other leg, we can also define the GNS-representation $\pi_C$ 
of $\mu$:

\begin{defn}
Define $\pi_C$ from ${\mathcal C}$ into ${\mathcal B}({\mathcal H}_C)$, by
$$
\pi_C(y)\Lambda_C(c)=\Lambda(yc), \quad {\text { for all $y,c\in{\mathcal C}$.}}
$$
Then $\pi_C$ is an injective ${}^*$-homomorphism, which is the GNS-representation of ${\mathcal C}$, such that 
$\pi_C({\mathcal C}){\mathcal H}_C$ is dense in ${\mathcal H}_C$.
We can also extend $\pi_C$ to the level of $M({\mathcal C})$.
\end{defn}

The GNS-representations allow us to properly define our $C^*$-algebras $B$ and $C$:

\begin{defn}\label{BandC}
\begin{enumerate}
\item Define $B:=\overline{\pi_B({\mathcal B})}^{\|\ \|}$, as a non-degenerate 
$C^*$-subalgebra of ${\mathcal B}({\mathcal H}_B)$. Or, equivalently, 
$$
B=\overline{\bigl\{(\operatorname{id}\otimes\omega)(E):\omega\in{\mathcal B}
({\mathcal H}_C)_*\bigr\}}^{\|\ \|}\,\bigl(\subseteq{\mathcal B}({\mathcal H}_B)\bigr),
$$
where we wrote $E=\Pi(E)\in{\mathcal B}({\mathcal H}_B\otimes{\mathcal H}_C)$, for convenience.
\item Define $C:=\overline{\pi_C({\mathcal C})}^{\|\ \|}$, as a non-degenerate $C^*$-subalgebra of 
${\mathcal B}({\mathcal H}_C)$.  Or, equivalently, 
$$
C=\overline{\bigl\{(\omega\otimes\operatorname{id})(E):\omega\in{\mathcal B}
({\mathcal H}_B)_*\bigr\}}^{\|\ \|}\,\bigl(\subseteq{\mathcal B}({\mathcal H}_C)\bigr),
$$
\end{enumerate}
\end{defn}

We may also get to work with their enveloping von Neumann algebras, namely, 
$$
N:=\pi_B({\mathcal B})''\,\subseteq{\mathcal B}({\mathcal H}_B) 
\quad {\text { and }} \quad
L:=\pi_C({\mathcal C})''\,\subseteq{\mathcal B}({\mathcal H}_C) .
$$

\subsection{The left Hilbert algebras}\label{sub2.2}

Eventually, we wish to lift the functionals $\nu$ and $\mu$ to the $C^*$-algebra level.  Observe first that 
$\Lambda_B({\mathcal B})\subseteq{\mathcal H}_B$ and $\Lambda_C({\mathcal C})\subseteq{\mathcal H}_C$ 
are left Hilbert algebras, as in the Tomita--Takesaki modular theory (see \cite{Tk2}):

\begin{prop}\label{hilbertalgebra}
The subspaces $\Lambda_B({\mathcal B})\subseteq{\mathcal H}_B$ and $\Lambda_C({\mathcal C})
\subseteq{\mathcal H}_C$ are left Hilbert algebras, with respect to the multiplications and the ${}^*$-structures 
inherited from ${\mathcal B}$ and ${\mathcal C}$, respectively.
\end{prop}

\begin{proof}
For any $x\in{\mathcal B}$, we have already shown that $\pi_B(x):\Lambda_B(b)\mapsto\Lambda_B(xb)$, 
$b\in{\mathcal B}$, given by the multiplication, determines a bounded operator.  The involution, $x\mapsto x^*$, 
$x\in{\mathcal B}$, is such that
$$
\bigl\langle\Lambda_B(xb),\Lambda_B(b')\bigr\rangle=\nu({b'}^*xb)=\nu({b'}^*(x^*)^*b)=\nu({x^*b'}^*b)
=\bigl\langle\Lambda_B(b),\Lambda(x^*b')\bigr\rangle, 
$$
for $b,b'\in{\mathcal B}$.  This shows that for any $\xi,\eta\in\Lambda_B({\mathcal B})$, we have 
$\bigl\langle\Lambda_B(x)\xi,\eta\bigr\rangle=\bigl\langle\xi,\Lambda_B(x^*)\eta\bigr\rangle$.

To see that the involution is pre-closed, note that for any fixed $b\in{\mathcal B}$ and any $x_n\in{\mathcal B}$, 
we have:
$$
\bigl\langle\Lambda_B(b),\Lambda_B(x_n^*)\bigr\rangle=\nu(x_nb)=\nu\bigl([\sigma^{\nu}]^{-1}(b)x_n\bigr)
=\bigl\langle\Lambda_B(x_n),\Lambda_B([\sigma^{\nu}]^{-1}(b)^*)\bigr\rangle.
$$
Since $\Lambda_B({\mathcal B})$ is dense in ${\mathcal H}_B$, we can quickly see that if $x_n\to 0$  and 
$x_n^*\to z$ in ${\mathcal B}$, then $z=0$.  So the involution is pre-closed. We already know that 
${\mathcal B}^2={\mathcal B}$, and $\Lambda_B({\mathcal B})^2$ is dense in  ${\mathcal H}_B$.

In this way, we showed that $\Lambda_B({\mathcal B})$ becomes a left Hilbert algebra (see Definition~1.1 
in \cite{Tk2}).  Similarly, we can show that $\Lambda_C({\mathcal C})$ is also a left Hilbert algebra.
\end{proof}

The modular theory associates to the left Hilbert algebra $\Lambda_B({\mathcal B})$ a von Neumann algebra, which should be 
none other than $N=\pi_B({\mathcal B})''$.  Also by the general theory on left Hilbert algebras, we obtain a normal semi-finite 
faithful (n.s.f.) weight $\tilde{\nu}$ on $N$.  Consider the associated spaces ${\mathfrak N}_{\tilde{\nu}}
=\bigl\{x\in N:\tilde{\nu}(x^*x)<\infty\bigr\}$ and ${\mathfrak M}_{\tilde{\nu}}
={\mathfrak N}_{\tilde{\nu}}^*{\mathfrak N}_{\tilde{\nu}}$.  The general theory provides us with the following properties:

We have a closed linear map $\Lambda_{\tilde{\nu}}:{\mathfrak N}_{\tilde{\nu}}\to{\mathcal H}_B$ (the same Hilbert 
space), which is the GNS-map for the weight $\tilde{\nu}$.  The map $\Lambda_{\tilde{\nu}}$ extends $\Lambda_B$, 
such that $\pi_B({\mathcal B})\subseteq{\mathfrak N}_{\tilde{\nu}}$ and $\Lambda_{\tilde{\nu}}\circ\pi_B=\Lambda_B$. 
Note that the weight $\tilde{\nu}$ extends the functional $\nu$.  In particular, we have 
$\tilde{\nu}\bigl(\pi_B(x)^*\pi_B(x)\bigr)=\nu(x^*x)$, for all $x\in{\mathcal B}$. For any $b\in{\mathfrak N}_{\tilde{\nu}}$, 
there exists a sequence $(x_n)_n$ in ${\mathcal B}$ such that 
$\Lambda_B(x_n)\xrightarrow{{\text { (in ${\mathcal H}_B$) }}}\Lambda_{\tilde{\nu}}(b)$ and 
$\pi_B(x_n)\xrightarrow{{\text { ($\sigma$-strong-${}^*$) }}}b$.

Denote by $T_{\tilde{\nu}}$ the closure of the involution $\Lambda_B(x)\mapsto\Lambda_B(x^*)$ on 
$\Lambda_B({\mathcal B})$. There exists a polar decomposition, $T_{\tilde{\nu}}=J_{\tilde{\nu}}\nabla_{\tilde{\nu}}^{\frac12}$,
where $\nabla_{\tilde{\nu}}$ is the modular operator, given by $\nabla_{\tilde{\nu}}=T_{\tilde{\nu}}^*T_{\tilde{\nu}}$, 
and $J_{\tilde{\nu}}$ is the modular conjugation, which is anti-unitary.  

According to the modular theory in the von Neumann algebra setting, the modular operator defines a strongly 
continuous one-parameter group of automorphisms $\sigma^{\tilde{\nu}}$, by $\sigma^{\tilde{\nu}}_t(b)
=\nabla_{\tilde{\nu}}^{it}b\nabla_{\tilde{\nu}}^{-it}$, for $b\in{\mathcal B}$, $t\in\mathbb{R}$, leaving the von Neumann algebra 
$N$ invariant.  We have $\tilde{\nu}\circ\sigma^{\tilde{\nu}}_t=\tilde{\nu}$ for $t\in\mathbb{R}$, and $\tilde{\nu}$ satisfies a certain 
KMS boundary condition.  In particular, the weak KMS property at the ${}^*$-algebra level, $\nu(bb')
=\nu\bigl(b'\sigma^{\nu}(b)\bigr)$, $b,b'\in{\mathcal B}$, extends to the von Neumann algebra as $\tilde{\nu}(xx')
=\tilde{\nu}\bigl(x'\sigma^{\tilde{\nu}}_{-i}(x)\bigr)$, $x\in{\mathfrak M}_{\tilde{\nu}}$, $x'\in{\mathcal D}(\sigma^{\tilde{\nu}}_{-i})$.  
Meanwhile, the modular conjugation $J_{\nu}$ can be characterized by 
$J_{\nu}\Lambda_{\nu}(x)=\Lambda_{\nu}\bigl(\sigma^{\tilde{\nu}}_{\frac{i}{2}}(x)^*\bigr)$, for $x\in{\mathfrak N}_{\nu}$. 

Observe that the elements in ${\mathcal B}$ are actually analytic elements for $\nu$.

\begin{prop}\label{Banalyticelements}
Consider the KMS weights $\nu$ on $B$.  Then any element $b\in{\mathcal B}\,(\subseteq B)$ is an analytic element for 
$\nu$.
\end{prop}

\begin{proof}
Recall the weak KMS property of $\nu$, or Equation~\eqref{(nu_weakKMS)}, such that there exists an automorphism 
$\sigma^{\nu}$ on ${\mathcal B}$ satisfying $\nu(bb')=\nu\bigl(b'\sigma^{\nu}(b)\bigr)$, for all $b,b'\in{\mathcal B}$.  
As such, we can see that for any $b\in{\mathcal B}$, we have $b\in{\mathcal D}(\sigma^{\nu}_i)$ and that $\sigma^{\nu}_i(b)
=\nabla_{\nu}^{-1}b\nabla_{\nu}=\sigma^{\nu}(b)$.  
Then we can see quickly that we also have $b\in{\mathcal D}(\sigma^{\nu}_{mi})$ and that $\sigma^{\nu}_{mi}(b)
=\nabla_{\nu}^{-m}b\nabla_{\nu}^m$, for all $m\in\mathbb{Z}$.  Continuing, it is not difficult to see that 
$b\in{\mathcal D}(\sigma^{\nu}_{z})$ and that 
$\sigma^{\nu}_z(b)=\nabla_{\nu}^{iz}b\nabla_{\nu}^{-iz}$, for any $z\in\mathbb{C}$.
\end{proof}

Meanwhile, we can also do the same with the left Hilbert algebra $\Lambda_C({\mathcal C})$, obtaining another 
n.s.f.~weight $\tilde{\mu}$ on $L=\pi_C({\mathcal C})$, as well as the modular operator $\nabla_{\tilde{\mu}}$, the 
modular conjugation $J_{\tilde{\mu}}$, the modular automorphism group $\sigma^{\tilde{\mu}}$. The elements in 
${\mathcal C}$ will be analytic elements for $\mu$.

However, having gathered all these results from the left Hilbert algebra theory and the modular theory, we have to 
point out that they are not quite sufficient for our purposes.  Since we wish to develop a $C^*$-algebraic framework, 
two main issues arise: (i).~The modular automorphism group $\sigma^{\tilde{\nu}}$ leaves the von Neumann algebra $N$ 
invariant, but we also want it to leave the $C^*$-algebra $B$ invariant; (ii).~While $(\sigma^{\tilde{\nu}}_t)_{t\in\mathbb{R}}$ 
is strongly continuous, we want it to be norm continuous as well.  As (i) and (ii) are not automatic consequences of the modular 
theory, some more work is needed. To remedy this situation, let us gather below some additional results on the 
canonical idempotent $E$.

\subsection{The idempotent $E$}\label{sub2.3}

We saw earlier that we may regard our canonical idempotent $E\in M({\mathcal B}\odot{\mathcal C})$ as a bounded operator 
$\Pi(E)\in{\mathcal B}({\mathcal H}_B\otimes{\mathcal H}_C)$.  In fact, we can now see that $\Pi=\pi_B\otimes\pi_C$, 
and that $(\pi_B\otimes\pi_C)(E)$ is an element of the tensor product von Neumann algebra $N\otimes L$.  It is also 
apparent that $(\pi_B\otimes\pi_C)(E)\in M(B\otimes C)$, where $\otimes$ is now a (spatial) $C^*$-tensor product. 
For convenience, we will regard $E=(\pi_B\otimes\pi_C)(E)$ in what follows. 

Also for convenience, we may regard $x\in{\mathcal B}$ as $x=\pi_B(x)\in B\,\subseteq N\,\subseteq{\mathcal B}({\mathcal H}_B)$, 
and regard $y\in{\mathcal C}$ as $y=\pi_C(y)\in C\,\subseteq L\,\subseteq{\mathcal B}({\mathcal H}_C)$.

Recall from Equation~\eqref{(S_Bmap)} that for $b\in{\mathcal B}$, we have:
$$
(\nu\otimes\operatorname{id})\bigl(E(b\otimes1)\bigr)=(\nu\otimes\operatorname{id})
\bigl(E(1\otimes S_{\mathcal B}(b))\bigr)=(\nu\otimes\operatorname{id})(E)
S_{\mathcal B}(b)=S_{\mathcal B}(b).
$$
As the weight $\tilde{\nu}$ extends the functional $\nu$, we can use the above observation to define the map 
$\gamma_N:N\to L$, by 
\begin{equation}\label{(gamma_n)}
\gamma_N(b):=(\tilde{\nu}\otimes\operatorname{id})\bigl(E(b\otimes1)\bigr), \quad 
b\in{\mathcal B}.
\end{equation}
This map may be unbounded, but as ${\mathcal B}=\pi_B({\mathcal B})$ is dense in $N$ and $S_{\mathcal B}({\mathcal B})
={\mathcal C}$ is dense in $L$, we see that $\gamma_N:N\to L$ is a densely-defined map having a dense range, which is 
an injective anti-homomorphism because $S_{\mathcal B}$ is. However it is not a ${}^*$-map, as $S_{\mathcal B}$ is not.

Consider instead $\widetilde{R}:=\gamma_N\circ\sigma^{\tilde{\nu}}_{-\frac{i}2}$, where $\sigma^{\tilde{\nu}}_{-\frac{i}2}$ 
is the analytic generator for $(\sigma^{\tilde{\nu}}_t)_{t\in\mathbb{R}}$, at $z=-\frac{i}2$.  Since $\sigma^{\tilde{\nu}}_t$ 
is an automorphism and since $\gamma_N$ is an anti-homomorphism, we can see quickly that $\widetilde{R}$ is 
anti-multiplicative.   As all elements of ${\mathcal B}$ are analytic, we see that $\sigma^{\tilde{\nu}}_{-\frac{i}2}$ is 
densely-defined, and so is $\widetilde{R}$.  In addition, the following observation shows that $\widetilde{R}$ is involutive. 
Note that for $b\in{\mathcal T}_{\tilde{\nu}}$, we have:
\begin{align}
\widetilde{R}(b^*)
&=(\tilde{\nu}\otimes\operatorname{id})\bigl(E(\sigma^{\tilde{\nu}}_{-\frac{i}2}(b^*)
\otimes1)\bigr)=(\tilde{\nu}\otimes\operatorname{id})\bigl((\sigma^{\tilde{\nu}}_i
(\sigma^{\tilde{\nu}}_{-\frac{i}2}(b^*))\otimes1)E\bigr)   \notag \\
&=(\tilde{\nu}\otimes\operatorname{id})\bigl(\sigma^{\tilde{\nu}}_{\frac{i}2}(b^*)\otimes1)E\bigr),
\end{align}
because of the KMS property of $\tilde{\nu}$.  At the same time, 
$$
\widetilde{R}(b)^*=\bigl[(\tilde{\nu}\otimes\operatorname{id})
(E(\sigma^{\tilde{\nu}}_{-\frac{i}2}(b)\otimes1))\bigr]^*
=(\tilde{\nu}\otimes\operatorname{id})\bigl(\sigma^{\tilde{\nu}}_{\frac{i}2}(b^*)
\otimes1)E\bigr),
$$
because $E$ is self-adjoint.  Comparing, we see that $\widetilde{R}(b^*)=\widetilde{R}(b)^*$.  This shows that 
$\widetilde{R}$ is a ${}^*$-map, which means that it is actually a ${}^*$-anti-homomorphism, so bounded.  Therefore, 
we can extend $\widetilde{R}$ to all of $N$. In fact, as $\widetilde{R}$ is a bounded map from $N$ to $L$, injective, 
densely-defined, having a dense range, it extends to a ${}^*$-anti-isomorphism $\widetilde{R}:N\to L$.  Meanwhile, 
from the definition of $\tilde{R}$, it is immediate that we have $\gamma_N=\tilde{R}\circ\sigma^{\tilde{\nu}}_{\frac{i}2}$, 
which is essentially like a polar decomposition.

Next, consider the n.s.f. weight $\tilde{\mu}$ on $L$, extending the functional $\mu$ on ${\mathcal C}$.  In an 
analogous way as above, we can consider the extension of the map $S_{\mathcal C}:{\mathcal C}\to{\mathcal B}$ 
to the von Neumann algebra level, namely the densely-defined anti-homomorphism $\gamma_L:L\to N$.  Analogous 
to Equation~\eqref{(gamma_n)} for $\gamma_N$, we can characterize it by 
\begin{equation}\label{(gamma_l)}
\gamma_L(c)=(\operatorname{id}\otimes\tilde{\mu})\bigl((1\otimes c)E\bigr), 
\quad c\in{\mathcal C}.
\end{equation}
In the following proposition, we gather some useful relationships between the weights $\tilde{\nu}$, $\tilde{\mu}$, 
and the maps $\gamma_N$ and $\gamma_L$.

\begin{prop}\label{vngammamaps}
Let the weights $\tilde{\nu}$ on $N$ and $\tilde{\mu}$ on $L$ be the extensions of the functionals $\nu$ and $\mu$, 
and let $\gamma_N:N\to L$ and $\gamma_L:L\to N$ be the densely-defined anti-homomorphisms as in 
Equations~\eqref{(gamma_n)} and \eqref{(gamma_l)}, extending the maps $S_{\mathcal B}$ and $S_{\mathcal C}$.  
Also let $\widetilde{R}=\gamma_N\circ\sigma^{\tilde{\nu}}_{-\frac{i}2}$ be the ${}^*$-anti-isomorphism from $N$ to $L$ 
obtained above. Then
\begin{enumerate}
\item $\tilde{\nu}=\tilde{\mu}\circ\gamma_N$ and $\tilde{\nu}=\tilde{\mu}\circ\widetilde{R}$.
\item $\gamma_N=\widetilde{R}\circ\sigma^{\tilde{\nu}}_{\frac{i}2}
=\sigma^{\tilde{\mu}}_{-\frac{i}2}\circ\widetilde{R}$.
\item $\gamma_L=\sigma^{\tilde{\nu}}_{\frac{i}2}\circ\widetilde{R}^{-1}
=\widetilde{R}^{-1}\circ\sigma^{\tilde{\mu}}_{-\frac{i}2}$.
\item For any $t\in\mathbb{R}$, we have 
$(\sigma^{\tilde{\nu}}_t\otimes\sigma^{\tilde{\mu}}_{-t})(E)=E$.
\item $(\gamma_N\otimes\gamma_L)(E)=\varsigma E$ and $(\gamma_L\otimes\gamma_N)(\varsigma E)=E$.
\item $(\widetilde{R}\otimes\widetilde{R}^{-1})(E)=\varsigma E$ and $(\widetilde{R}^{-1}\otimes\widetilde{R})
(\varsigma E)=E$.
\end{enumerate}
\end{prop}

\begin{proof}
(1).
Recall that at the ${}^*$-algebra level, we have $\nu=\mu\circ S_{\mathcal B}:{\mathcal B}\to{\mathcal C}$. 
Extending this to the von Neumann algebra level, we have $\tilde{\nu}=\tilde{\mu}\circ\gamma_N$, or equivalently, 
$\tilde{\nu}\circ\gamma_N^{-1}=\tilde{\mu}$. 

From $\widetilde{R}=\gamma_N\circ\sigma^{\tilde{\nu}}_{-\frac{i}2}$, we can write $\gamma_N^{-1}
=\bigl(\widetilde{R}\circ\sigma^{\tilde{\nu}}_{\frac{i}2}\bigr)^{-1}=\sigma^{\tilde{\nu}}_{-\frac{i}2}\circ\widetilde{R}^{-1}$.  
So from $\tilde{\nu}\circ\gamma_N^{-1}=\tilde{\mu}$, we have 
$\tilde{\nu}\circ\sigma^{\tilde{\nu}}_{-\frac{i}2}\circ\widetilde{R}^{-1}=\tilde{\mu}$. 
Since we know $\tilde{\nu}\circ\sigma^{\tilde{\nu}}_t=\tilde{\nu}$, $\forall t$, it follows that 
$\tilde{\nu}\circ\widetilde{R}^{-1}=\tilde{\mu}$.  Or equivalently, $\tilde{\nu}=\tilde{\mu}\circ\widetilde{R}$.

(2). Since $\tilde{\nu}=\tilde{\mu}\circ\widetilde{R}$, it is easy to see that the modular automorphism groups have 
the following relation:
\begin{equation}\label{(sigma^nusigma^mu)}
\sigma^{\tilde{\nu}}_t=\widetilde{R}^{-1}\circ\sigma^{\tilde{\mu}}_{-t}\circ\widetilde{R}, 
\quad \forall t\in\mathbb{R}.
\end{equation}
We already know $\gamma_N=\widetilde{R}\circ\sigma^{\tilde{\nu}}_{\frac{i}2}$, which is immediate from the definition 
of $\widetilde{R}$.  Moreover, by Equation~\eqref{(sigma^nusigma^mu)}, we have $\gamma_N
=\widetilde{R}\circ\bigl(\widetilde{R}^{-1}\circ\sigma^{\tilde{\mu}}_{-\frac{i}2}\circ\widetilde{R}\bigr)
=\sigma^{\tilde{\mu}}_{-\frac{i}2}\circ\widetilde{R}$.

(3). At the ${}^*$-algebra level, from Equation~\eqref{(nu_weakKMS)}, we know about the KMS-type automorphism 
$\sigma^{\nu}=S_{\mathcal B}^{-1}\circ S_{\mathcal C}^{-1}$.  At the von Neumann algebra level, this extends to 
$\sigma^{\tilde{\nu}}_{-i}=\gamma_N^{-1}\circ\gamma_L^{-1}$.  We thus have 
$$
\gamma_L=\bigl(\gamma_N\circ\sigma^{\tilde{\nu}}_{-i}\bigr)^{-1}=\sigma^{\tilde{\nu}}_{i}\circ\gamma_N^{-1}
=\sigma^{\tilde{\nu}}_{i}\circ\bigl(\widetilde{R}\circ\sigma^{\tilde{\nu}}_{\frac{i}2}\bigr)^{-1}
=\sigma^{\tilde{\nu}}_{i}\circ\sigma^{\tilde{\nu}}_{-\frac{i}2}\circ\widetilde{R}^{-1}
=\sigma^{\tilde{\nu}}_{\frac{i}2}\circ\widetilde{R}^{-1}.
$$
Alternatively, by Equation~\eqref{(sigma^nusigma^mu)}, we have 
$\gamma_L=\bigl(\widetilde{R}^{-1}\circ\sigma^{\tilde{\mu}}_{-\frac{i}2}\circ
\widetilde{R}\bigr)\circ\widetilde{R}^{-1}
=\widetilde{R}^{-1}\circ\sigma^{\tilde{\mu}}_{-\frac{i}2}$.

(4). For arbitrary $b\in{\mathcal D}(\sigma^{\tilde{\nu}}_{\frac{i}2})$ and $t\in\mathbb{R}$, observe that
\begin{align}
&(\tilde{\nu}\otimes\operatorname{id})\bigl((\sigma^{\tilde{\nu}}_t\otimes
\sigma^{\tilde{\mu}}_{-t})(E)(b\otimes1)\bigr)
\notag \\
&=(\tilde{\nu}\otimes\operatorname{id})\bigl((\sigma^{\tilde{\nu}}_t\otimes\sigma^{\tilde{\mu}}_{-t})
(E(\sigma^{\tilde{\nu}}_{-t}(b)\otimes1))\bigr)=(\tilde{\nu}\otimes\operatorname{id})
\bigl((\operatorname{id}\otimes\sigma^{\tilde{\mu}}_{-t})
(E(\sigma^{\tilde{\nu}}_{-t}(b)\otimes1))\bigr)  
\notag \\
&=\sigma^{\tilde{\mu}}_{-t}\bigl(\gamma_N(\sigma^{\tilde{\nu}}_{-t}(b))\bigr)
=\bigl(\sigma^{\tilde{\mu}}_{-t}\circ\widetilde{R}\circ\sigma^{\tilde{\nu}}_{\frac{i}2}
\circ\sigma^{\tilde{\nu}}_{-t}\bigr)(b)
=\bigl(\sigma^{\tilde{\mu}}_{-t}\circ\widetilde{R}\circ\sigma^{\tilde{\nu}}_{-t}\circ
\sigma^{\tilde{\nu}}_{\frac{i}2}\bigr)(b)
\notag \\
&=\bigl((\widetilde{R}\circ\sigma^{\tilde{\nu}}_{t}\circ\widetilde{R}^{-1})\circ
\widetilde{R}\circ\sigma^{\tilde{\nu}}_{-t}\circ\sigma^{\tilde{\nu}}_{\frac{i}2}\bigr)(b)
=(\widetilde{R}\circ\sigma^{\tilde{\nu}}_{\frac{i}2})(b)   \notag \\
&=\gamma_N(b) 
=(\tilde{\nu}\otimes\operatorname{id})\bigl(E(b\otimes1)\bigr).
\notag
\end{align}
We used the fact that $\sigma^{\tilde{\nu}}_t$ is an automorphism for the first equality; for the second, 
we used $\tilde{\nu}\circ\sigma^{\tilde{\nu}}_t=\tilde{\nu}$; and in the rest, we used the definition of $\gamma_N$ 
and Equation~\eqref{(sigma^nusigma^mu)}.  This is true for any $b\in{\mathcal D}(\sigma^{\tilde{\nu}}_{\frac{i}2})$ 
and $\tilde{\nu}$ is faithful, so we see that $(\sigma^{\tilde{\nu}}_t\otimes\sigma^{\tilde{\mu}}_{-t})(E)=E$, 
for any $t\in\mathbb{R}$.

(5). At the ${}^*$-algebra level, it is known that $(S_{\mathcal B}\otimes S_{\mathcal C})(E)=\varsigma E$, 
where $\varsigma$ is the flip map between ${\mathcal B}\odot{\mathcal C}$ and ${\mathcal C}\odot{\mathcal B}$.
It thus follows that at the von Neumann algebra level, we have $(\gamma_N\otimes\gamma_L)(E)=\varsigma E$. 
Also $(\gamma_L\otimes\gamma_N)(\varsigma E)=E$.

(6). Combine the results of (4) and (5).  Since $\widetilde{R}=\gamma_N\circ\sigma^{\tilde{\nu}}_{-\frac{i}2}$ 
and $\widetilde{R}^{-1}=\gamma_L\circ\sigma^{\tilde{\mu}}_{\frac{i}2}$, we have:
$$
(\widetilde{R}\otimes\widetilde{R}^{-1})(E)=(\gamma_N\otimes\gamma_L)
\bigl((\sigma^{\tilde{\nu}}_{-\frac{i}2}\otimes\sigma^{\tilde{\mu}}_{\frac{i}2})(E)\bigr)
=(\gamma_N\otimes\gamma_L)(E)=\varsigma E.
$$
We also have $(\widetilde{R}^{-1}\otimes\widetilde{R})(\varsigma E)=E$.
\end{proof}

\subsection{The KMS weights on the $C^*$-algebras $B$ and $C$}\label{sub2.4}

Note that by restricting the weight $\tilde{\nu}$ on $N=\pi_B({\mathcal B})''$ to the $C^*$-algebra 
$B=\overline{\pi_B({\mathcal B})}^{\|\ \|}$, represented on the same Hilbert space ${\mathcal H}_B$, we obtain 
a faithful lower semi-continuous weight. For convenience, we will denote this weight by $\nu$, as it is also 
an extension of the functional $\nu$ at the ${}^*$-algebra level.  We can consider the associated spaces 
${\mathfrak N}_{\nu}=\bigl\{x\in B:\nu(x^*x)<\infty\bigr\}$ and ${\mathfrak M}_{\nu}
={\mathfrak N}_{\nu}^*{\mathfrak N}_{\nu}$.  

We can also consider the operator $T_{\nu}$, the closure of $\Lambda_B(x)\mapsto\Lambda_B(x^*)$, 
$x\in{\mathcal B}$.  It is apparent that it will exactly coincide with $T_{\tilde{\nu}}$ earlier, and the polar 
decomposition will also remain exactly same, $T_{\nu}=J_{\nu}\nabla_{\nu}^{\frac12}$, with 
$\nabla_{\nu}=\nabla_{\tilde{\nu}}$ and $J_{\tilde{\nu}}=J_{\nu}$.  However, the stumbling issue 
in \S\ref{sub2.2} was the question whether the associated modular automorphism group $\sigma^{\nu}_t:x
\mapsto\nabla_{\nu}^{it}x\nabla_{\nu}^{-it}$ leaves the $C^*$-algebra $B$ invariant, and whether 
$\bigl(\sigma^{\nu}_t\bigr)_{t\in\mathbb{R}}$ forms a norm-continuous one-parameter group.  With 
the results gathered in \S\ref{sub2.3}, we are now in a position to resolve this issue in the affirmative.

\begin{prop}\label{KMSweightnu}
Consider the weight $\nu$ on $B$, restricted from $\tilde{\nu}$ on $N$. Then
\begin{enumerate}
\item The automorphism group $(\sigma^{\tilde{\nu}}_t)_{t\in\mathbb{R}}$ leaves $B$ invariant. So we can 
consider $\sigma^{\nu}_t:=\sigma^{\tilde{\nu}}_t|_B$, for $t\in\mathbb{R}$.
\item $\nu$ becomes a KMS weight on $B$, equipped with the automorphism group 
$(\sigma^{\nu}_t)_{t\in\mathbb{R}}$, which is norm-continuous.
\end{enumerate}
\end{prop}

\begin{proof}
(1). Consider $\omega\in{\mathcal B}({\mathcal H}_C)_*$ and consider $(\operatorname{id}\otimes\omega)(E)\in B$. 
Such elements are dense in $B$.  For any $t\in\mathbb{R}$ we know from Proposition~\ref{vngammamaps}\,(4) 
that $(\sigma^{\tilde{\nu}}_{-t}\otimes\sigma^{\tilde{\mu}}_{t})(E)=E$.  we thus have
\begin{equation}\label{(KMSweightnu_eqn)}
\sigma^{\tilde{\nu}}_t\bigl((\operatorname{id}\otimes\omega)(E)\bigr)
=\sigma^{\tilde{\nu}}_t\bigl((\operatorname{id}\otimes\omega)
[(\sigma^{\tilde{\nu}}_{-t}\otimes\sigma^{\tilde{\mu}}_{t})(E)]\bigr)
=\bigl(\operatorname{id}\otimes(\omega\circ\sigma^{\tilde{\mu}}_t)\bigr)(E)\in B.
\end{equation}
This shows that $\sigma^{\tilde{\nu}}_t(B)=B$, for all $t\in\mathbb{R}$.  We will just write 
$\sigma^{\nu}_t:=\sigma^{\tilde{\nu}}_t|_B$.

(2). As we noted above, it is clear that $\nu$ is a faithful lower semi-continuous weight because $\tilde{\nu}$ is 
an n.s.f.~weight.  In addition, we know that $\nu$ is semi-finite because it extends the distinguished functional $\nu$, 
which is defined on a dense subalgebra ${\mathcal B}\subseteq B$.  

Meanwhile, since $\sigma^{\tilde{\nu}}_t(B)=B$, we can consider the one-parameter group of automorphisms 
$(\sigma^{\nu}_t)_{t\in\mathbb{R}}$.  At present, we only know that it is strongly continuous.  But, the strong 
continuity together with Equation~\eqref{(KMSweightnu_eqn)} show us that $t\mapsto\sigma^{\nu}_t(b)$ 
is indeed norm-continuous.

Finally, it is evident that $\nu$ satisfies appropriate KMS properties, by inheriting the properties of the n.s.f. weight 
$\tilde{\nu}$ and the automorphism group $(\sigma^{\tilde{\nu}}_t)_{t\in\mathbb{R}}$.  In particular, it is obvious 
that $\nu\circ\sigma^{\nu}_t=\nu$, and that for any $x\in{\mathcal D}(\sigma^{\nu}_{\frac{i}2})$, we have 
$$
\nu(x^*x)=\nu\bigl(\sigma^{\nu}_{\frac{i}2}\sigma^{\nu}_{\frac{i}2}(x)^*\bigr).
$$
In this way, we show that $\nu$ is a KMS weight (see \cite{Tk2}, \cite{KuVaweightC*}, 
\cite{KuKMS}).
\end{proof}

Note, by the way, that the ${}^*$-anti-isomorphism $\widetilde{R}:N\to L$ can be restricted to the $C^*$-algebra 
level.  So consider $R:=\widetilde{R}|_B$.  Then for $(\operatorname{id}\otimes\omega)(E)\in B$, we have:
\begin{align}
R\bigl((\operatorname{id}\otimes\omega)(E)\bigr)&=\widetilde{R}\bigl((\omega\otimes\operatorname{id})(\varsigma E)\bigr)
\notag \\
&=\widetilde{R}\bigl((\omega\otimes\operatorname{id})[(\widetilde{R}\otimes\widetilde{R}^{-1})(E)]\bigr)
=\bigl((\omega\circ\widetilde{R})\otimes\operatorname{id}\bigr)(E)\in C, 
\notag
\end{align}
where we used the result of Proposition~\ref{vngammamaps}\,(6).  This shows that $\widetilde{R}:N\to L$ restricts 
to $R:B\to C$.  It becomes a $C^*$-anti-isomorphism.

This means that together with the ${}^*$-anti-isomorphism $R:B\to C$ and the KMS weight $\nu$ on $B$, it turns out 
that $(E,B,\nu)$ forms a {\em separability triple}, in the sense of \cite{BJKVD_SepId}:

\begin{prop}\label{Esepid}
The restriction $R=\widetilde{R}|_B:B\to C$ is a $C^*$-anti-isomorhism. The self-adjoint idempotent $E\in M(B\otimes C)$ 
is a separability idempotent, in the sense that 
\begin{enumerate}
\item $(\nu\otimes\operatorname{id})(E)=1$
\item For $b\in{\mathcal D}(\sigma^{\nu}_{\frac{i}2})$ we have: 
$(\nu\otimes\operatorname{id})\bigl(E(b\otimes1)\bigr)=(R\circ\sigma^{\nu}_{\frac{i}2})(b)$.
\end{enumerate}
\end{prop}

\begin{proof}
We showed that $R$ and $\sigma^{\nu}$ are now valid at the level of the $C^*$-algebra $B$. Then (1) is just 
recognizing the fact that $\nu$ extends the distinguished functional $\nu$ on ${\mathcal B}$, and (2) is just noting 
that $R\circ\sigma^{\nu}_{\frac{i}2}=\gamma_B=\gamma_L|_B$, extending $S_{\mathcal B}$.
\end{proof}

\begin{rem}
While we gave results only regarding the weight $\nu$ on $B$, a very much similar argument can 
be given for the weight $\mu$ on the $C^*$-algebra $C$, as a restriction of $\tilde{\mu}$ on $L$. It would extend 
the distinguished functional $\mu$, and become a KMS weight on the $C^*$-algebra $C$,  equipped with the 
norm-continuous one-parameter group $(\sigma^{\mu}_t)_{t\in\mathbb{R}}$ given by the modular operator.
\end{rem}

In Proposition~\ref{Banalyticelements}, we saw that the elements in ${\mathcal B}$ are analytic elements for $\nu$. 
Similarly, the elements of ${\mathcal C}$ are analytic elements for $\mu$.  These results suggest that 
the elements in $\Lambda_B({\mathcal B})$ and $\Lambda_C({\mathcal C})$ are {\em right bounded vectors\/} 
in ${\mathcal H}_B$ and ${\mathcal H}_C$, respectively. (See Definition~1.7 of \cite{Tk2} for the notion of right bounded 
vectors.) To make this point clearer, see the proposition below. 

\begin{prop}\label{BCrightbounded}
\begin{itemize}
\item For any $b\in{\mathcal B}$, the vector $\Lambda_B(b)\in{\mathcal H}_B$ is right bounded. This means that 
the map $\pi^R_B(b):\Lambda_B(x)\mapsto\Lambda_B(xb)$ is a bounded operator on ${\mathcal H}_B$.
\item For any $c\in{\mathcal C}$, the vector $\Lambda_C(c)\in{\mathcal H}_C$ is right bounded. Or the map 
$\pi^R_C(c):\Lambda_C(y)\mapsto\Lambda_C(yc)$ is a bounded operator on ${\mathcal H}_C$. 
\end{itemize}
\end{prop}

\begin{proof}
(1). For any $x\in{\mathcal B}$, we know $\pi_B(x)\Lambda_B(b)=\Lambda_B(xb)$.  Recall next the unitary 
operator $\widehat{J}_B:{\mathcal H}_B\to{\mathcal H}_C$ defined earlier (see \S\ref{sub2.1}), given by 
$\widehat{J}_B\Lambda_B(x)=\Lambda_C\bigl(S_{\mathcal B}(x)\bigr)$.  Note that we can write:
\begin{align}
[\widehat{J}_B]^*\pi_C\bigl(S_{\mathcal B}(b)\bigr)\widehat{J}_B\Lambda_B(x)
&=[\widehat{J}_B]^*\pi_C\bigl(S_{\mathcal B}(b)\bigr)\Lambda_C\bigl(S_{\mathcal B}(x)\bigr)
=[\widehat{J}_B]^*\Lambda_C\bigl(S_{\mathcal B}(b)S_{\mathcal B}(x)\bigr)  \notag \\
&=[\widehat{J}_B]^*\Lambda_C\bigl(S_{\mathcal B}(xb)\bigr)
=\Lambda_B\bigl((S_{\mathcal B}^{-1}\circ S_{\mathcal B})(xb)\bigr)
=\Lambda_B(xb).
\notag
\end{align}
Combining, we observe that $\pi_B(x)\Lambda_B(b)=\Lambda_B(xb)=\pi^R_B(b)\Lambda_B(x)$, where 
$\pi^R_B(b)$ is the bounded operator $[\widehat{J}_B]^*\pi_C\bigl(S_{\mathcal B}(b)\bigr)\widehat{J}_B$. 
This proves that $\Lambda_B(b)$, $b\in{\mathcal B}$, is right bounded in ${\mathcal H}_B$.

(2). Similarly, we can show that $\pi_C(y)\Lambda_C(c)=\Lambda_C(yc)=\pi^R_C(c)\Lambda_C(y)$, 
where $\pi^R_C(c)=\widehat{J}_B\pi_B\bigl(S_{\mathcal B}^{-1}(c)\bigr)[\widehat{J}_B]^*$, 
a bounded operator.  So $\Lambda_C(c)$, $c\in{\mathcal C}$, is right bounded in 
${\mathcal H}_C$.
\end{proof}

\section{The $C^*$-bialgebra $(A,\Delta)$}\label{sec3}

Recall that our weak multiplier Hopf ${}^*$-algebra $({\mathcal A},\Delta)$ is equipped with a faithful positive 
left integral $\varphi$.  As $\varphi$ is a positive linear functional, we can equip ${\mathcal A}$ with an inner product:
$$
\langle x,y\rangle:=\varphi(y^*x), \quad {\text { for $x,y\in{\mathcal A}$.}}
$$
As usual, complete ${\mathcal A}$ with respect to the induced norm, and obtain a Hilbert space ${\mathcal H}$ 
with the natural inclusion $\Lambda:{\mathcal A}\to{\mathcal H}$. (Note that $\Lambda$ is injective because 
$\varphi$ is faithful.)  We are planning to represent our $C^*$-algebra as an operator algebra in 
${\mathcal B}({\mathcal H})$, but at present it is not clear if the left multiplication of the elements of $A$ are 
bounded.  Some work is needed.

\subsection{Representations of $B$ and $C$ on ${\mathcal H}$}\label{sub3.1}

Note that ${\mathcal C}{\mathcal A}={\mathcal A}{\mathcal C}={\mathcal A}$.  This suggests us to define 
the map $\rho_C:{\mathcal A}\to{\mathcal L}({\mathcal H}_C,{\mathcal H})$, by 
$$
\rho_C(a)\Lambda_C(y)=\Lambda(ya), \quad {\text { for $a\in{\mathcal A}$, $y\in{\mathcal C}$.}}
$$
The next proposition shows that $\rho_C(a)$, $a\in{\mathcal A}$,  is bounded.

\begin{prop}\label{rho_C(a)bounded}
Consider $\rho_C:{\mathcal A}\to{\mathcal L}({\mathcal H}_C,{\mathcal H})$ above.  Then $\rho_C(a)$ is 
a bounded element in ${\mathcal L}({\mathcal H}_C,{\mathcal H})$, for any $a\in{\mathcal A}$. 

\end{prop}

\begin{proof}
Let $a\in{\mathcal A}$ and $y\in{\mathcal C}$.  Then
$$
\bigl\|\rho_C(a)\Lambda_C(y)\bigr\|_{\mathcal H}^2=\bigl\langle\Lambda(ya),\Lambda(ya)\bigr\rangle
=\varphi(a^*y^*ya)=\varphi\bigl(y^*ya\sigma(a^*)\bigr),
$$
where $\sigma$ is the modular automorphism for $\varphi$, as noted in Proposition~\ref{phi_modularautomorphism}.  
Apply here the result of Proposition~\ref{muphinupsi}, knowing that the weight $\mu$ extends the functional $\mu$ 
on ${\mathcal C}$. Then we have:
$$
\bigl\|\rho_C(a)\Lambda_C(y)\bigr\|_{\mathcal H}^2
=\varphi\bigl(y^*ya\sigma(a^*)\bigr)=\mu\bigl((\operatorname{id}\otimes\varphi)(\Delta(y^*ya\sigma(a^*)))\bigr),
$$

Note that by Proposition~\ref{DeltaonBandC}, since $y^*y\in{\mathcal C}$, we have $\Delta(y^*y)=(y^*y\otimes1)E$. 
So we have 
$$
\Delta\bigl(y^*ya\sigma(a^*)\bigr)=\Delta(y^*y)\Delta\bigl(a\sigma(a^*)\bigr)
=(y^*y\otimes1)E\Delta\bigl(a\sigma(a^*)\bigr)=(y^*y\otimes1)\Delta\bigl(a\sigma(a^*)\bigr).
$$
Putting this in the previous equation, we see that
$$
\bigl\|\rho_C(a)\Lambda_C(y)\bigr\|_{\mathcal H}^2=\mu\bigl(y^*y(\operatorname{id}
\otimes\varphi)(\Delta(a\sigma(a^*)))\bigr)=\mu(y^*yc)
=\bigl\langle\Lambda_C(yc),\Lambda_C(y)\bigr\rangle_{{\mathcal H}_C},
$$
where $c=(\operatorname{id}\otimes\varphi)\bigl(\Delta(a\sigma(a^*))\bigr)\in M({\mathcal C})$, by the left 
invariance property of $\varphi$.

By Proposition~\ref{BCrightbounded}, we can write $\Lambda_C(yc)=\pi_C^R(c)\Lambda_C(y)$, where 
$\pi_C^R(c)$ is a bounded operator.  So the previous equation becomes:
$$
\bigl\|\rho_C(a)\Lambda_C(y)\bigr\|_{\mathcal H}^2
=\bigl\langle\pi_C^R(c)\Lambda_C(y),\Lambda_C(y)\bigr\rangle_{{\mathcal H}_C}\,
\le\,\bigr\|\pi_C^R(c)\bigr\|\bigl\|\Lambda_C(y)\bigr\|_{{\mathcal H}_C}^2,
$$
showing that $\bigl\|\rho_C(a)\bigr\|\le\bigr\|\pi_C^R(c)\bigr\|^{\frac12}$.
\end{proof}

\begin{cor}
For any $a,b\in{\mathcal A}$, we have 
$\rho_C(b)^*\rho_C(a)\in{\mathcal B}({\mathcal H}_C)$, a bounded operator on ${\mathcal H}_C$.
\end{cor}

\begin{proof}
By the previous proposition, we know $\rho_C(a),\rho_C(b)\in{\mathcal L}({\mathcal H}_C,{\mathcal H})$ are 
bounded, which also means $\rho_C(b)^*$ is a bounded element in ${\mathcal L}({\mathcal H},{\mathcal H}_C)$. 
It follows that $\rho_C(b)^*\rho_C(a)$ is also bounded, such that $\rho_C(b)^*\rho_C(a)\in{\mathcal B}({\mathcal H}_C)$.
\end{proof}

See below that any operator of the form $\rho_C(b)^*\rho_C(a)\in{\mathcal B}({\mathcal H}_C)$, 
$a,b\in{\mathcal A}$, commutes with the elements of the $C^*$-algebra 
$C=\overline{\pi_C({\mathcal C})}^{\|\ \|}\,\bigl(\subseteq{\mathcal B}({\mathcal H}_C)\bigr)$.

\begin{prop}\label{Ccommutant}
For $a,b\in{\mathcal A}$, consider $\rho_C(b)^*\rho_C(a)\in{\mathcal B}({\mathcal H}_C)$ as above. 
It commutes with the elements of the $C^*$-algebra $C$, regarded as an operator algebra contained 
in ${\mathcal B}({\mathcal H}_C)$.
\end{prop}

\begin{proof}
For any $y_1,y_2\in{\mathcal C}$, we have
\begin{align}
\bigl\langle\rho_C(b)^*\rho_C(a)\Lambda_C(y_1),\Lambda_C(y_2)\bigr\rangle_{{\mathcal H}_C}
&=\bigl\langle\rho_C(a)\Lambda_C(y_1),\rho_C(b)\Lambda_C(y_2)\bigr\rangle_{\mathcal H}
=\bigl\langle\Lambda(y_1a),\Lambda(y_2b)\bigr\rangle_{\mathcal H}   \notag \\
&=\varphi\bigl(b^*y_2^*y_1a)=\varphi\bigl(y_2^*y_1a\sigma(b^*)\bigr).
\notag
\end{align}
By the same argument as in the proof of Proposition~\ref{rho_C(a)bounded}, where we used 
Propositions~\ref{muphinupsi} and \ref{DeltaonBandC}, we can write 
$$
\varphi\bigl(y_2^*y_1a\sigma(b^*)\bigr)=\cdots=\mu(y_2^*y_1\tilde{c}),
$$
where $\tilde{c}=(\operatorname{id}\otimes\varphi)\bigl(\Delta(a\sigma(b^*))\bigr)\in M({\mathcal C})$. 
So we can write $\rho_C(b)^*\rho_C(a)\Lambda_C(y_1)=\Lambda_C(y_1\tilde{c})$, which means that 
$\rho_C(b)^*\rho_C(a)\in{\mathcal B}({\mathcal H}_C)$ is none other than $\pi_C^R(\tilde{c})$.

As such, for $\pi_C(c)\in C$ and any $y\in{\mathcal C}$, we have:
$$
\rho_C(b)^*\rho_C(a)\pi_C(c)\Lambda_C(y)=\rho_C(b)^*\rho_C(a)\Lambda_C(cy)
=\pi_C^R(\tilde{c})\Lambda_C(cy)=\Lambda_c(cy\tilde{c}),
$$
$$
\pi_C(c)\rho_C(b)^*\rho_C(a)\Lambda_C(y)=\pi_C(c)\pi_C^R(\tilde{c})\Lambda_C(y)
=\pi_C(c)\Lambda_C(y\tilde{c})=\Lambda_C(cy\tilde{c}),
$$
showing that $\rho_C(b)^*\rho_C(a)$ commutes with any $\pi_C(c)\in C$.
\end{proof}

We are now ready to construct a ${}^*$-representation of the $C^*$-algebra $C$ into ${\mathcal B}({\mathcal H})$. 
See below:

\begin{prop}\label{ConH}
Consider any $c\in{\mathcal C}$, which we regard as $c=\pi_C(c)$, an element of the $C^*$-algebra $C$. 
Define $\alpha\bigl(\pi_C(c)\bigr)\in{\mathcal L}({\mathcal H})$, by 
$$\alpha\bigl(\pi_C(c)\bigr)\Lambda(a)=\Lambda(ca),\quad a\in{\mathcal A}.$$
Then
\begin{enumerate}
\item $\alpha\bigl(\pi_C(c)\bigr)$, $c\in{\mathcal C}$, is a bounded operator on 
${\mathcal H}$.
\item $\alpha$ extends to a (bounded) $C^*$-representation $\alpha:C\to{\mathcal B}
({\mathcal H})$.
\item $\alpha:C\to{\mathcal B}({\mathcal H})$ becomes a non-degenerate ${}^*$-representation.  
It also extends to the ${}^*$-representation at the level of the multiplier algebra $M(C)$.
\end{enumerate}
\end{prop}

\begin{proof}
(1). Without loss of generality, we may consider the vectors of the type $\Lambda(ya)\in{\mathcal H}$, where 
$y\in{\mathcal C}$, $a\in{\mathcal A}$, because ${\mathcal C}{\mathcal A}={\mathcal A}$. Note that we can write
$$
\alpha(\pi_C(c))\Lambda(ya)=\Lambda(cya)=\rho_C(a)\Lambda_C(cy)=\rho_C(a)\pi_C(c)\Lambda_C(y).
$$
We know that $\rho_C(a)$ is bounded.  We thus have:
\begin{align}
&\bigl\|\alpha(\pi_C(c))\Lambda(ya)\bigr\|^2  \notag \\
&=\bigl\langle\rho_C(a)\pi_C(c)\Lambda_C(y),\rho_C(a)\pi_C(c)\Lambda_C(y)\bigr\rangle_{\mathcal H}
=\bigl\langle\rho_C(a)^*\rho_C(a)\pi_C(c)\Lambda_C(y),\pi_C(c)\Lambda_C(y)\bigr\rangle_{{\mathcal H}_C}
\notag \\
&=\bigl\langle\pi_C(c)\rho_C(a)^*\rho_C(a)\Lambda_C(y),\pi_C(c)\Lambda_C(y)\bigr\rangle_{{\mathcal H}_C}  
\notag \\
&\le\bigl\|\pi_C(c)\bigr\|^2\bigl\langle\rho_C(a)^*\rho_C(a)\Lambda_C(y),\Lambda_C(y)\bigr\rangle_{{\mathcal H}_C}
=\bigl\|\pi_C(c)\bigr\|^2\bigl\langle\rho_C(a)\Lambda_C(y),\rho_C(a)\Lambda_C(y)\bigr\rangle_{\mathcal H}   
\notag \\
&=\bigl\|\pi_C(c)\bigr\|^2\bigl\langle\Lambda(ya),\Lambda(ya)\bigr\rangle_{\mathcal H}
=\bigl\|\pi_C(c)\bigr\|^2\bigr\|\Lambda(ya)\bigr\|^2.
\notag
\end{align}
Note that the third equality is because $\rho_C(a)^*\rho_C(a)$ commutes with $\pi_C(c)\in C$ 
(see Proposition~\ref{Ccommutant}).  This observation shows that $\alpha(\pi_C(c))$ is bounded, 
with $\bigl\|\alpha(\pi_C(c))\bigr\|\le\bigl\|\pi_C(c)\bigr\|$.

(2). It is not difficult to see that $\alpha$ preserves multiplication.  Note that for any $c_1,c_2\in{\mathcal C}$ 
and for any $a\in{\mathcal A}$, we have
$$
\alpha\bigl(\pi_C(c_1)\bigr)\alpha\bigl(\pi_C(c_2)\bigr)\Lambda(a)
=\alpha\bigl(\pi_C(c_1)\bigr)\Lambda(c_2a)
=\Lambda(c_1c_2a)=\alpha\bigl(\pi_C(c_1c_2)\bigr)\Lambda(a).
$$
As $\pi_C(c_1c_2)=\pi_C(c_1)\pi_C(c_2)$, and since the vectors $\Lambda(a)$, $a\in{\mathcal A}$, 
are dense in ${\mathcal H}$, it is evident that $\alpha\bigl(\pi_C(c_1)\bigr)\alpha\bigl(\pi_C(c_2)\bigr)
=\alpha\bigl(\pi_C(c_1)\pi_C(c_2)\bigr)$.  

To see if $\alpha$ preserves the involution, note that for any $c\in{\mathcal C}$ and any $a_1,a_2\in{\mathcal A}$, 
we have
\begin{align}
\bigl\langle\alpha(\pi_C(c))\Lambda(a_1),\Lambda(a_2)\bigr\rangle
&=\bigl\langle\Lambda(ca_1),\Lambda(a_2)\bigr\rangle
=\varphi(a_2^*ca_1)=\varphi\bigl((c^*a_2)^*a_1\bigr)
\notag \\
&=\bigl\langle\Lambda(a_1),\alpha(\pi_C(c)^*)\Lambda(a_2)\bigr\rangle,
\notag
\end{align}
since $\pi_C(c^*)=\pi_C(c)^*$.  Since $a_1,a_2\in{\mathcal A}$ are arbitrary, we see that 
$\alpha\bigl(\pi_C(c)\bigr)^*=\alpha\bigl(\pi_C(c)^*\bigr)$.

This means $\alpha:\pi_C({\mathcal C})\to{\mathcal B}({\mathcal H})$ is a ${}^*$-representation, which is 
automatically bounded.  It follows that $\alpha$ extends to $C=\overline{\pi_C({\mathcal C})}^{\|\ \|}$, giving us 
the ${}^*$-representation of the $C^*$-algebra $C$.

(3). The non-degeneracy of $\alpha$  is easy to see, using the fact that ${\mathcal C}{\mathcal A}={\mathcal A}$. 
As a result, it is clear that $\alpha$ naturally extends to the level of the multiplier algebra $M(C)$
\end{proof}

An analogous procedure can be carried out for $B$. First, we can consider $\rho_B:{\mathcal A}\to{\mathcal L}({\mathcal H}_B,{\mathcal H})$, 
as follows:

\begin{prop}\label{rho_B(a)bounded}
Let $\rho_B:{\mathcal A}\to{\mathcal L}({\mathcal H}_B,{\mathcal H})$, by 
$$
\rho_B(a)\Lambda_B(x)=\Lambda(xa), \quad {\text { for $a\in{\mathcal A}$, 
$x\in{\mathcal B}$.}}
$$
Then $\rho_B(a)$ is a bounded element in ${\mathcal L}({\mathcal H}_B,{\mathcal H})$, for any 
$a\in{\mathcal A}$. 
\end{prop}

\begin{proof}
The proof is essentially no different from that of Proposition~\ref{rho_C(a)bounded}. Use again 
Propositions~\ref{muphinupsi} and \ref{DeltaonBandC}, but this time use the right invariance of $\psi$.
\end{proof}

We can also define a ${}^*$-representation of the $C^*$-algebra $B$ into ${\mathcal B}({\mathcal H})$:

\begin{prop}\label{BonH}
For any $b\in{\mathcal B}$, regarded as an element of the $C^*$-algebra $B$ by $b=\pi_B(b)$, define 
$\beta\bigl(\pi_B(b)\bigr)\in{\mathcal L}({\mathcal H})$, by 
$$\beta\bigl(\pi_B(b)\bigr)\Lambda(a)=\Lambda(ba),\quad a\in{\mathcal A}.$$
Then
\begin{enumerate}
\item $\beta\bigl(\pi_B(b)\bigr)\in{\mathcal B}({\mathcal H})$, for any $b\in{\mathcal B}$.
\item $\beta$ extends to a (bounded) $C^*$-representation $\beta:B\to{\mathcal B}({\mathcal H})$.
\item $\beta:B\to{\mathcal B}({\mathcal H})$ becomes a non-degenerate ${}^*$-representation.  
It also extends to the ${}^*$-representation at the level of the multiplier algebra $M(B)$.
\end{enumerate}
\end{prop}

\begin{proof}
Proof can be carried out using essentially the same argument as in Propositions~\ref{Ccommutant} 
and \ref{ConH} above, in the construction of the $C^*$-representation $\alpha:C\to{\mathcal B}({\mathcal H})$.
\end{proof}

\subsection{The $C^*$-algebra $A$} \label{sub3.2}

Unlike the subalgebras ${\mathcal B}$ and ${\mathcal C}$, it is not clear at this stage whether the elements of ${\mathcal A}$ 
can be similarly all regarded as bounded operators on ${\mathcal H}$.  For this purpose, let us define the following operator $W$, 
using the left invariance property of $\varphi$.  This will help us construct the left regular representation of ${\mathcal A}$:

\begin{prop}\label{W}
There exists a bounded operator $W\in{\mathcal B}({\mathcal H}\otimes{\mathcal H})$ satisfying the following:
\begin{enumerate}
\item For any $a,b\in{\mathcal A}$, we have
$$
W^*\bigl(\Lambda(a)\otimes\Lambda(b)\bigr)
=(\Lambda\otimes\Lambda)\bigl((\Delta b)(a\otimes1)\bigr).
$$
\item
For any $a,b\in{\mathcal A}$, we have
$$
W\bigl(\Lambda(a)\otimes\Lambda(b)\bigr)=(\Lambda\otimes\Lambda)\bigl((S^{-1}\otimes\operatorname{id})
(\Delta b)(a\otimes1)\bigr),
$$
where $S$ denotes the antipode on $({\mathcal A},\Delta)$.
\item For any $a,b\in{\mathcal A}$, we have
$$W^*W\bigl(\Lambda(a)\otimes\Lambda(b)\bigr)=(\Lambda\otimes\Lambda)\bigl(E(a\otimes b)\bigr).$$
\end{enumerate}
\end{prop}

\begin{proof}
Let $a,b\in{\mathcal A}$ be arbitrary.  Define $W^*:\Lambda(a)\otimes\Lambda(b)
\mapsto W^*\bigl(\Lambda(a)\otimes\Lambda(b)\bigr)$, by 
\begin{equation}\label{(W*defn)}
W^*\bigl(\Lambda(a)\otimes\Lambda(b)\bigr)
=(\Lambda\otimes\Lambda)\bigl((\Delta b)(a\otimes1)\bigr).
\end{equation}
It is well-defined on $(\Lambda\otimes\Lambda)({\mathcal A}\odot{\mathcal A})$, but at present we do not know it is bounded.

From the definition of $W^*$ given in Equation~\eqref{(W*defn)}, we can find an expression for its adjoint operator, $W$.  To compute this, 
let $c,d\in{\mathcal A}$ be arbitrary.  Then
\begin{align}
&\bigl\langle W(\Lambda(a)\otimes\Lambda(b)),\Lambda(c)\otimes\Lambda(d)\bigr\rangle  \notag \\
&=\bigl\langle\Lambda(a)\otimes\Lambda(b),W^*(\Lambda(c)\otimes\Lambda(d))\bigr\rangle
=\bigl\langle\Lambda(a)\otimes\Lambda(b),(\Lambda\otimes\Lambda)((\Delta d)(c\otimes1))\bigr\rangle
\notag \\
&=(\varphi\otimes\varphi)\bigl((c^*\otimes1)\Delta(d^*)(a\otimes b)\bigr)
=\varphi\bigl(c^*(\operatorname{id}\otimes\varphi)(\Delta(d^*)(1\otimes b))a\bigr).
\notag
\end{align}
Here, we may use the characterization of the antipode map $S$, given in Proposition~\ref{antipodeS}\,(1), so that we have 
$(\operatorname{id}\otimes\varphi)(\Delta(d^*)(1\otimes b))=S^{-1}((\operatorname{id}\otimes\varphi)((1\otimes d^*)(\Delta b)))$.
Then the above becomes 
\begin{align}
&=\varphi\bigl(c^*S^{-1}((\operatorname{id}\otimes\varphi)((1\otimes d^*)(\Delta b)))a\bigr) 
=(\varphi\otimes\varphi)\bigl((c^*\otimes d^*)(S^{-1}\otimes\operatorname{id})(\Delta b)(a\otimes1)\bigr) 
\notag \\
&=\bigl\langle (\Lambda\otimes\Lambda)((S^{-1}\otimes\operatorname{id})(\Delta b)(a\otimes1)),
\Lambda(c)\otimes\Lambda(d)\bigr\rangle.
\notag
\end{align}
As $c,d\in{\mathcal A}$ are arbitrary, this shows that for any $a,b\in{\mathcal A}$, we have 
\begin{equation}\label{(Wcharacterization)}
W\bigl(\Lambda(a)\otimes\Lambda(b)\bigr)=(\Lambda\otimes\Lambda)\bigl((S^{-1}\otimes\operatorname{id})(\Delta b)(a\otimes1))\bigr).
\end{equation}

Next, let us combine Equations~\eqref{(W*defn)} and \eqref{(Wcharacterization)}, to obtain an expression for $W^*W$. 
Observe that for $a,b\in{\mathcal A}$, we have
$$
W^*W\bigl(\Lambda(a)\otimes\Lambda(b)\bigr)=W^*\bigl((\Lambda\otimes\Lambda)((S^{-1}\otimes\operatorname{id})(\Delta b)(a\otimes1)))\bigr).
$$
We may use the Sweedler notation to write 
$(S^{-1}\otimes\operatorname{id})(\Delta b)(a\otimes1))=\sum_{(b)}\bigl[S^{-1}(b_{(1)})a\otimes b_{(2)}\bigr]$.  Then applying $W^*$, it becomes:
$$
W^*W\bigl(\Lambda(a)\otimes\Lambda(b)\bigr)
=(\Lambda\otimes\Lambda)\left(\sum_{(b)}\bigl[b_{(2)}S^{-1}(b_{(1)})a\otimes b_{(3)}\bigr]\right).
$$
Use here a result from the algebraic framework, namely, Proposition~4.3 of \cite{VDWangwha1}, which says 
that $\sum_{(b)}\bigl[b_{(2)}S^{-1}(b_{(1)})a\otimes b_{(3)}\bigr]=E(a\otimes b)$. It follows that 
\begin{equation}\label{(W^*W)}
W^*W\bigl(\Lambda(a)\otimes\Lambda(b)\bigr)=(\Lambda\otimes\Lambda)\bigl(E(a\otimes b)\bigr).
\end{equation}

Recall that $E\in M(B\otimes C)$.  Since we know that the elements of $M(B)$ and $M(C)$ are 
considered as bounded operators by the representations $\alpha$ and $\beta$, respectively, it follows that
$$
W^*W\bigl(\Lambda(a)\otimes\Lambda(b)\bigr)=(\Lambda\otimes\Lambda)\bigl(E(a\otimes b)\bigr)
=(\alpha\otimes\beta)(E)(\Lambda\otimes\Lambda)(a\otimes b).
$$
This shows that $W^*W=E$, where we regard $E$ as the operator 
$(\alpha\otimes\beta)(E)\in{\mathcal B}({\mathcal H}\otimes{\mathcal H})$. In this way, 
we show that $W^*W$ is a bounded operator, which in turn means that $W$ and $W^*$ are also bounded. 
Equations~\eqref{(W*defn)} and \eqref{(Wcharacterization)} characterize the operators $W^*$ and $W$, respectively.
\end{proof}

\begin{rem}
In general, the operator $W$ is not unitary, unless $E=1\otimes1$. In fact, the observation $W^*W=E$ in the above proof 
indicates that $W$ is actually a partial isometry.  We will further discuss this aspect in the next subsection.
\end{rem}

Here is one more result regarding $W$, which will be useful in defining the GNS representation of $A$ in ${\mathcal H}$:

\begin{prop}\label{Wrepresentation}
For any $a,p,q\in{\mathcal A}$, we have: 
$$
(\operatorname{id}\otimes\omega_{\Lambda(p),\Lambda(q)})(W)\Lambda(a)
=\Lambda\bigl((\operatorname{id}\otimes\varphi)(\Delta(q^*)(1\otimes p))a\bigr).
$$
\end{prop}

\begin{proof}
For any $d\in{\mathcal A}$, we have:
\begin{align}
&\bigl\langle(\operatorname{id}\otimes\omega_{\Lambda(p),\Lambda(q)})(W)\Lambda(a),
\Lambda(d)\bigr\rangle   \notag \\
&=\bigl\langle W(\Lambda(a)\otimes\Lambda(p)),\Lambda(d)\otimes\Lambda(q)\bigr\rangle
=\bigl\langle\Lambda(a)\otimes\Lambda(p),W^*(\Lambda(d)\otimes\Lambda(q))\bigr\rangle 
\notag \\
&=\bigl\langle\Lambda(a)\otimes\Lambda(p),(\Lambda\otimes\Lambda)
(\Delta(q)(d\otimes1))\bigr\rangle 
=(\varphi\otimes\varphi)\bigl((d^*\otimes1)\Delta(q^*)(a\otimes p)\bigr)  \notag \\
&=\varphi\bigl(d^*(\operatorname{id}\otimes\varphi)(\Delta(q^*)(1\otimes p))a\bigr) 
=\bigl\langle\Lambda((\operatorname{id}\otimes\varphi)(\Delta(q^*)(1\otimes p))a),
\Lambda(d)\bigr\rangle.
\notag
\end{align}
As  $d\in{\mathcal A}$ was arbitrary, this proves the result.
\end{proof}

Recall that by the fullness assumption of $\Delta$, we know that the elements of the form 
$x=(\operatorname{id}\otimes\varphi)(\Delta(q^*)(1\otimes p))$, where $p,q\in{\mathcal A}$, span all of ${\mathcal A}$. 
Therefore what Proposition~\ref{Wrepresentation} is saying is that for any $x\in{\mathcal A}$, we can find a bounded 
operator $X\in{\mathcal B}({\mathcal H})$ such that $X\Lambda(a)=\Lambda(xa)$, for all $a\in{\mathcal A}$. 
In this way, we can define the GNS-representation $\pi$ of $\varphi$:

\begin{defn}
Define $\pi$ from ${\mathcal A}$ into ${\mathcal B}({\mathcal H})$ by
$$
\pi(x)\Lambda(a)=\Lambda(xa), \quad {\text { for all $x,a\in{\mathcal A}$.}}
$$
Then $\pi$ is an injective ${}*$-homomorphism, which is the GNS representation of ${\mathcal A}$, such that 
$\pi({\mathcal A}){\mathcal H}$ is dense in ${\mathcal H}$.
\end{defn}

The last statement on the non-degeneracy of $\pi$ is a consequence of ${\mathcal A}^2={\mathcal A}$. 
The GNS-representation allows us to properly define our $C^*$-algebra $A$:

\begin{defn}\label{C*algebraA}
Define $A:=\overline{\pi({\mathcal A})}^{\|\ \|}$, as a non-degenerate $C^*$-subalgebra of ${\mathcal B}({\mathcal H})$. 
It can be also characterized as
$$
A=\overline{\bigl\{(\operatorname{id}\otimes\omega)(W):\omega\in{\mathcal B}({\mathcal H})_*\bigr\}}^{\|\ \|}.
$$
\end{defn}

The alternative characterization of $A$ is obtained by noting from Proposition~\ref{Wrepresentation} that for 
$x=(\operatorname{id}\otimes\varphi)\bigl(\Delta(q^*)(1\otimes p)\bigr)$, for any $p,q\in{\mathcal A}$, we have 
$$
\pi(x)=\pi\bigl((\operatorname{id}\otimes\varphi)(\Delta(q^*)(1\otimes p))\bigr)
=(\operatorname{id}\otimes\omega_{\Lambda(p),\Lambda(q)})(W).
$$

As $\pi$ is a non-generate ${}^*$-representation, it can be naturally extended to the level of the multiplier algebra 
$M(A)$.  We will often regard $A=\pi(A)$ and $M(A)=\pi\bigl(M(A)\bigr)$. 

At the ${}^*$-algebra level, we saw that  ${\mathcal B}$ and ${\mathcal C}$ are subalgebras in $M({\mathcal A})$.  
As these algebras are now all represented on ${\mathcal H}$ by left multiplications, and in turn completed to the 
$C^*$-algebras $\beta(B)$, $\alpha(C)$, $\pi\bigl(M(A)\bigr)$, respectively, it is apparent that $\alpha=\pi|_{\mathcal C}$, 
$\beta=\pi|_{\mathcal B}$.  It is thus  natural to regard $B=\beta(B)=\pi(B)\subset M(A)$ and $C=\alpha(C)=\pi(C)\subset M(A)$, 
as operator algebras in ${\mathcal B}({\mathcal H})$.  We also have $M(B)\subset M(A)$ and $M(C)\subset M(A)$. 
As for our idempotent $E$, we may regard $E=(\alpha\otimes\beta)(E)=(\pi\otimes\pi)(E)\in M(A\otimes A)$.
While it is true that in Section~\ref{sec2} we considered the $C^*$-algebras $B$ and $C$ as represented on 
${\mathcal H}_B$ and ${\mathcal H}_C$, respectively, and such aspect may still be needed down the road, we will be 
able to tell from the context on which space they are represented.

\subsection{The comultiplication on $A$} \label{sub3.3}

We next wish to define the  comultiplication at the $C^*$-algebra level, extending the comultiplication on ${\mathcal A}$. 
We have our candidate below using the operator $W$, analogous to the quantum group case. We still need some work 
to clarify that this is indeed a correct definition. 

\begin{defn}\label{comultiplication}
Define the map $\widetilde{\Delta}$ from the ${\mathcal B}({\mathcal H})$ to ${\mathcal B}({\mathcal H}\otimes{\mathcal H})$, 
by $\widetilde{\Delta}(x)=W^*(1\otimes x)W$, for all $x\in{\mathcal B}({\mathcal H})$.
\end{defn}

The next proposition shows that $\widetilde{\Delta}$ is an extension of the  comultiplication on ${\mathcal A}$.

\begin{prop}\label{Deltaextends}
For any $a\in{\mathcal A}$, we have $\widetilde{\Delta}\bigl(\pi(a)\bigr)=(\pi\otimes\pi)(\Delta a)$.
\end{prop}

\begin{proof}
Let  $c,d\in{\mathcal A}$ be arbitrary.  Using the characterization of $W$ given in Proposition~\ref{W}\,(2), we have
\begin{align}
W^*\bigl(1\otimes\pi(a)\bigr)W\bigl(\Lambda(c)\otimes\Lambda(d)\bigr)
&=W^*\bigl(1\otimes\pi(a)\bigr)\bigl((\Lambda\otimes\Lambda)((S^{-1}\otimes\operatorname{id})
(\Delta d)(c\otimes1))\bigr)     \notag \\
&=W^*\bigl((\Lambda\otimes\Lambda)((1\otimes a)(S^{-1}\otimes\operatorname{id})
(\Delta d)(c\otimes1))\bigr)   \notag \\
&=W^*\left((\Lambda\otimes\Lambda)\left(\sum_{(d)}[S^{-1}(d_{(1)})c\otimes ad_{(2)}]\right)\right),
\notag
\end{align}
where we again used the Sweedler notation.  Thus, by applying the characterization of $W^*$ 
given in Proposition~\ref{W}\,(1), the above becomes:
$$
W^*\bigl(1\otimes\pi(a)\bigr)W\bigl(\Lambda(c)\otimes\Lambda(d)\bigr)
=(\Lambda\otimes\Lambda)\left((\Delta a)\sum_{(d)}\bigl[d_{(2)}S^{-1}(d_{(1)})c\otimes d_{(3)}\bigr]\right).
$$
As before use the algebraic result Proposition~4.3 of \cite{VDWangwha1}. Then we have 
\begin{align}
W^*\bigl(1\otimes\pi(a)\bigr)W\bigl(\Lambda(c)\otimes\Lambda(d)\bigr)&=(\Lambda\otimes\Lambda)\bigl((\Delta a)E(c\otimes d)\bigr)
=(\Lambda\otimes\Lambda)\bigl((\Delta a)(c\otimes d)\bigr)  \notag \\
&=(\pi\otimes\pi)(\Delta a)\bigl(\Lambda(c)\otimes\Lambda(d)\bigr).
\notag 
\end{align}
As $c,d\in{\mathcal A}$ can be arbitrary, this shows $\widetilde{\Delta}\bigl(\pi(a)\bigr)
=W^*\bigl(1\otimes\pi(a)\bigr)W=(\pi\otimes\pi)(\Delta a)$.
\end{proof}

In particular, as for our canonical idempotent $E$, which we consider as an operator 
$E=(\alpha\otimes\beta)(E)=(\pi\otimes\pi)(E)\in{\mathcal B}({\mathcal H}\otimes{\mathcal H})$ by the GNS-representation $\pi$, 
the following result is evident.

\begin{prop}\label{E=W^*W}
We have:
\begin{enumerate}
\item $E=\widetilde{\Delta}\bigl(\pi(1_{M({\mathcal A})})\bigr)$
\item $E=W^*W$
\end{enumerate}
\end{prop}

\begin{proof}
As an element of $M({\mathcal A}\odot{\mathcal A})$, we know that $E=\Delta(1_{M({\mathcal A})})$. 
So, by Proposition~\ref{Deltaextends}, we have:
$$
E=(\pi\otimes\pi)(E)=\widetilde{\Delta}\bigl(\pi(1_{M({\mathcal A})})\bigr)
=W^*\bigl(1\otimes\pi(1_{M({\mathcal A})})\bigr)W=W^*W.
$$
Indeed, the observation $W^*W=E\in{\mathcal B}({\mathcal H}\otimes{\mathcal H})$ has been already made in the proof of Proposition~\ref{W}.
\end{proof}

The operator $W$ is essentially like the multiplicative unitary operator (in the sense of \cite{BS}, \cite{Wr7}) in the framework 
of locally compact quantum groups \cite{KuVa}.  In our setting, however, as we have $E\ne1\otimes1$, the operator $W$ is not unitary.  
But then, it turns out that it is a partial isometry. See below:

\begin{prop}\label{lemmaW}
We have
\begin{enumerate}
\item $W^*(1\otimes x)=(\widetilde{\Delta} x)W^*$, for any $x\in A$.
\item $EW^*=W^*$.
\item $W$ is a partial isometry, satisfying $W^*WW^*=W^*$ and $WW^*W=W$.
\end{enumerate}
\end{prop}

\begin{proof}
(1). For $a\in{\mathcal A}$, and for any $c,d\in{\mathcal A}$, note that 
\begin{align}
W^*\bigl(1\otimes\pi(a)\bigr)\bigl(\Lambda(c)\otimes\Lambda(d)\bigr)
&=W^*\bigl(\Lambda(c)\otimes\Lambda(ad)\bigr)=(\Lambda\otimes\Lambda)\bigl(\Delta(ad)(c\otimes1)\bigr)
\notag \\
&=(\pi\otimes\pi)(\Delta a)(\Lambda\otimes\Lambda)\bigl(\Delta(d)(c\otimes1)\bigr)  \notag \\
&=\widetilde{\Delta}\bigl(\pi(a)\bigr)W^*\bigl(\Lambda(c)\otimes\Lambda(d)\bigr).
\notag
\end{align}
As $c,d\in{\mathcal A}$ are arbitrary, this shows that $W^*\bigl(1\otimes\pi(a)\bigr)=\widetilde{\Delta}\bigl(\pi(a)\bigr)W^*$. 
Since the $\pi(a)$, $a\in{\mathcal A}$, are dense in $A$, it follows that we have $W^*(1\otimes x)=(\widetilde{\Delta} x)W^*$, 
for any $x\in A$.

(2). It is evident that the result of (1) will hold true also for all $m\in M(A)$.  So, in particular, if $m=1_{M(A)}
=\pi\bigl(1_{M({\mathcal A})}\bigr)$, we have: 
$$
W^*=W^*\bigl(1\otimes1_{M(A)}\bigr)=\widetilde{\Delta}(1_{M(A)})W^*=EW^*.
$$

(3). We know from Proposition~\ref{E=W^*W} that $E=W^*W$.  Combining this with $EW^*=W^*$, we see
that $W^*WW^*=W^*$.  Also $WW^*W=W$.  So $W$ is a partial isometry.
\end{proof}

For convenience, write $\Delta=\widetilde{\Delta}|_A$, which is reasonable since $\widetilde{\Delta}$ extends $\Delta$ at the 
${}^*$-algebra level on ${\mathcal A}$ and ${\mathcal A}$ is dense in $A$.  The theorem below shows that 
$\Delta:x\mapsto W^*(1\otimes x)W$ determines a valid comultiplication on the $C^*$-algebra $A$.

\begin{thm}\label{comultiplicationrepresentation}
The map $\Delta=\widetilde{\Delta}|_A$ is a ${}^*$-representation of $A$ into $M(A\otimes A)$. It extends to 
a ${}^*$-representation from $M(A)$ into $M(A\otimes A)$, which we may still denote by $\Delta$.  

We also have:
\begin{enumerate}
\item $(\Delta x)(1\otimes y)\in A\otimes A$ and $(x\otimes1)(\Delta y)\in A\otimes A$, for all $x,y\in A$.
\item The following spaces are norm-dense in $A$:
$$
\operatorname{span}\bigl\{(\operatorname{id}\otimes\omega)((\Delta x)(1\otimes y)):\omega\in A^*,x,y\in A\bigr\},
$$
$$
\operatorname{span}\bigl\{(\omega\otimes\operatorname{id})((x\otimes1)(\Delta y)):\omega\in A^*,x,y\in A\bigr\}.
$$
\item The coassociativity condition holds:
$$
(\Delta\otimes\operatorname{id})(\Delta x)=(\operatorname{id}\otimes\Delta)(\Delta x), \quad \forall x\in A.
$$
\end{enumerate}
\end{thm}

\begin{proof}
It is easy to see that $\Delta(x^*)=W^*(1\otimes x^*)W=\bigl(W^*(1\otimes x)W\bigr)^*=\Delta(x)^*$.  In addition, 
for $x,y\in A$, by Proposition~\ref{lemmaW} we have, 
$$
\Delta(x)\Delta(y)=(\Delta x)W^*(1\otimes y)W
=W^*(1\otimes x)(1\otimes y)W=W^*(1\otimes xy)W=\Delta(xy).
$$
This shows that $\Delta$ is a ${}^*$-representation.

Next, let $a,b\in{\mathcal A}$ be arbitrary.  Then 
$$
\bigl(\pi(a)\otimes1\bigr)\Delta\bigl(\pi(b)\bigr)
=\bigl(\pi(a)\otimes1\bigr)(\pi\otimes\pi)(\Delta b)
=(\pi\otimes\pi)\bigl((a\otimes1)(\Delta b)\bigr)\in(\pi\otimes\pi)({\mathcal A}\odot{\mathcal A})\subset A\otimes A,
$$
because we know from Equation~\eqref{(comult1)} that $(a\otimes1)(\Delta b)\in{\mathcal A}\odot{\mathcal A}$ 
at the ${}^*$-algebra level. As $\pi({\mathcal A})$ is dense in $A$, this shows that $(x\otimes1)(\Delta y)\in A\otimes A$, 
for all $x,y\in A$.  Similarly, we can also show that $(\Delta x)(1\otimes y)\in A\otimes A$ for all $x,y\in A$.  These 
observations prove (1).

As an immediate consequence, we can see that for any $x,y,z\in A$, we have$(\Delta x)(y\otimes z)\in A\otimes A$ 
and $(y\otimes z)(\Delta x)\in A\otimes A$, showing that $\Delta(A)\subseteq M(A\otimes A)$.  In other words, we see that 
$\Delta$ is a ${}^*$-representation from $A$ into $M(A\otimes A)$.

Let $m,n\in M(A)$, and consider the expression $(p\otimes q)\widetilde{\Delta}(m)\widetilde{\Delta}(n)
=(p\otimes q)E\widetilde{\Delta}(m)\widetilde{\Delta}(n)$, for arbitrary $p,q\in {\mathcal A}$.  
Here, we used the fact that $EW^*=W^*$ (Proposition~\ref{lemmaW}), from which we have $E\widetilde{\Delta}(m)
=EW^*(1\otimes m)W=W^*(1\otimes m)W=\widetilde{\Delta}(m)$.  
Since $({\mathcal A}\odot{\mathcal A})\Delta({\mathcal A})=({\mathcal A}\odot{\mathcal A})E$
(see Lemma~\ref{canonicalE}), we know $(p\otimes q)E$ can be written as a sum of the expressions $(a\otimes b)(\Delta c)$. 
But then, we know that for any $c\in A$ and any $m\in M(A)$, we have
$$
\Delta(c)\widetilde{\Delta}(m)=(\Delta c)W^*(1\otimes m)W=W^*(1\otimes c)(1\otimes m)W=W^*(1\otimes cm)W=\Delta(cm),
$$
by Proposition~\ref{lemmaW}.  Using this result twice, it follows that 
$$
(a\otimes b)(\Delta c)\widetilde{\Delta}(m)\widetilde{\Delta}(n)=(a\otimes b)\Delta(cm)\widetilde{\Delta}(n)=(a\otimes b)\Delta(cmn)
=(a\otimes b)(\Delta c)\widetilde{\Delta}(mn).
$$
As noted above, any $(p\otimes q)E$ can be written as a sum of the expressions $(a\otimes b)(\Delta c)$. This means that we have:
$$
(p\otimes q)\widetilde{\Delta}(m)\widetilde{\Delta}(n)=(p\otimes q)E\widetilde{\Delta}(m)\widetilde{\Delta}(n)
=(p\otimes q)E\widetilde{\Delta}(mn)=(p\otimes q)\widetilde{\Delta}(mn).
$$
This result is true for arbitrary $p,q\in{\mathcal A}$.  Therefore, 
we can conclude that $\widetilde{\Delta}(m)\widetilde{\Delta}(n)=\widetilde{\Delta}(mn)$, showing that $\widetilde{\Delta}|_{M(A)}$ 
preserves the multiplicativity.  The ${}^*$-property is immediate, and it is evident that $\widetilde{\Delta}\bigl(M(A)\bigr)\subseteq M(A\otimes A)$.  
Therefore, we see that $\widetilde{\Delta}|_{M(A)}$ is a ${}^*$-representation from $M(A)$ to $M(A\otimes A)$, which naturally extends 
the comultiplication on ${\mathcal A}$ and on $A$.  As such, we may from this point on write $\Delta=\widetilde{\Delta}|_{M(A)}$.

Turning to the rest of the results in the proposition, the ``fullness'' property, given in (2), is a consequence of the fullness of $\Delta$ 
at the ${}^*$-algebra level, as the spanned space is exactly ${\mathcal A}$, which is norm-dense in $A$.

Finally, the coassociativity property:  We have already seen that $\Delta$ is well established at the level of both $A$ and $M(A)$, 
and that it extends the comultiplication at the dense ${}^*$-algebra level.  As one can expect, it indeed satisfies the coassociativity 
at the $C^*$-algebra level.  To confirm this, consider $a,b,c\in{\mathcal A}$.  We have:
\begin{align}
&(\Delta\otimes\operatorname{id})\bigl(\Delta(\pi(a))\bigr)
\bigl(\pi(b)\otimes1\otimes\pi(c)\bigr) 
=(\Delta\otimes\operatorname{id})\bigl(\Delta(\pi(a))(1\otimes\pi(c))\bigr)\bigl(\pi(b)\otimes1\otimes1\bigr)
\notag \\
&=(\pi\otimes\pi\otimes\pi)\bigl((\Delta\otimes\operatorname{id})((\Delta a)(1\otimes c))(b\otimes1\otimes1)\bigr)
\notag \\
&=(\pi\otimes\pi\otimes\pi)\bigl((\operatorname{id}\otimes\Delta)((\Delta a)(b\otimes1))(1\otimes1\otimes c)\bigr)
\notag \\
&=(\operatorname{id}\otimes\Delta)\bigl(\Delta(\pi(a))(\pi(b)\otimes1)\bigr)\bigl(1\otimes1\otimes\pi(c)\bigr)
\notag \\
&=(\operatorname{id}\otimes\Delta)\bigl(\Delta(\pi(a))\bigr)
\bigl(\pi(b)\otimes1\otimes\pi(c)\bigr).
\notag
\end{align}
As $A$ acts non-degenerately, this shows that 
$(\Delta\otimes\operatorname{id})\bigl(\Delta(\pi(a))\bigr)=(\operatorname{id}\otimes\Delta)\bigl(\Delta(\pi(a))\bigr)$, 
$\forall a\in{\mathcal A}$.  Since $\pi({\mathcal A})$ is dense in $A$, we see that $(\Delta\otimes\operatorname{id})\Delta
=(\operatorname{id}\otimes\Delta)\Delta$ on $A$.
\end{proof}

In this way, we have shown that $(A,\Delta)$ is a $C^*$-bialgebra, with the comultiplication $\Delta$ satisfying 
all the conditions prescribed in Definition~3.1 of \cite{BJKVD_qgroupoid1}.

\subsection{The multiplicative partial isometry $W$} \label{sub3.4}

We observed above that the behavior of our partial isometry $W$ quite resembles that of a multiplicative unitary operator 
(in the sense of \cite{BS}, \cite{Wr7}) in the quantum group setting.  While it is not unitary, it actually satisfies relations such that 
it can be referred to as a {\em multiplicative partial isometry\/}.  Such a notion has been noted in the finite-dimensional setting 
of weak $C^*$-Hopf algebras \cite{BSzMultIso}.  See also \cite{BJK_mpi}, where an axiomatic discussion on multiplicative 
partial isometries is given for a class of $C^*$-algebraic quantum groupoids.  

Here indeed, we can show that our partial isometry $W$ satisfies the defining axioms given in \cite{BSzMultIso} for being 
a multiplicative partial isometry (The conditions in \cite{BJK_mpi} are slightly different, but equivalent.):

\begin{prop}\label{multiplicativeW}
The operator $W$ satisfies the following conditions:
\begin{align}
W_{12}W_{13}W_{23}&=W_{23}W_{12} \label{(mpi5)}  \\
W_{12}^*W_{12}W_{13}&=W_{13}W_{23}W_{23}^* \label{(mpi6)} \\
W_{23}^*W_{23}W_{12}&=W_{12}W_{23}^*W_{23} \label{(mpi3)} \\
W_{12}W_{12}^*W_{23}&=W_{23}W_{12}W_{12}^* \label{(mpi4)} 
\end{align}
\end{prop}

\begin{proof}
Let $a,b,c\in{\mathcal A}$.  By the characterization of $W^*$ given in Proposition~\ref{W}, we have
\begin{align}
W^*_{12}W^*_{23}\bigl(\Lambda(a)\otimes\Lambda(b)\otimes\Lambda(c)\bigr)
&=W^*_{12}(\Lambda\otimes\Lambda\otimes\Lambda)\bigl((a\otimes(\Delta c))(1\otimes b\otimes 1)\bigr) \notag \\
&=(\Lambda\otimes\Lambda\otimes\Lambda)\bigl((\Delta\otimes\operatorname{id})(\Delta c)((\Delta b)\otimes 1)(a\otimes1\otimes1)\bigr),
\label{(multiplicativeW_eqn1)}
\end{align}
while we also have
\begin{align}
W^*_{23}W^*_{13}W^*_{12}\bigl(\Lambda(a)\otimes\Lambda(b)\otimes\Lambda(c)\bigr)
&=W^*_{23}W^*_{13}(\Lambda\otimes\Lambda\otimes\Lambda)\bigl(((\Delta b)\otimes c)(a\otimes1\otimes1)\bigr) \notag \\
&=W^*_{23}(\Lambda\otimes\Lambda\otimes\Lambda)\bigl((\Delta c)_{13}((\Delta b)\otimes 1)(a\otimes1\otimes1)\bigr)  \notag \\
&=(\Lambda\otimes\Lambda\otimes\Lambda)\bigl((\operatorname{id}\otimes\Delta)(\Delta c)((\Delta b)\otimes 1)(a\otimes1\otimes1)\bigr).
\label{(multiplicativeW_eqn2)}
\end{align}
Comparing Equations~\eqref{(multiplicativeW_eqn1)} and \eqref{(multiplicativeW_eqn2)}, with the knowledge that $\Delta$ is coassociative, 
namely $(\Delta\otimes\operatorname{id})(\Delta c)=(\operatorname{id}\otimes\Delta)(\Delta c)$, we conclude that 
$W^*_{12}W^*_{23}=W^*_{23}W^*_{13}W^*_{12}$.  Or equivalently, we have $W_{12}W_{13}W_{23}=W_{23}W_{12}$.  This is in fact 
the famous ``pentagon equation''.

In the case of a unitary, the other three equations of the proposition, namely Equations~\eqref{(mpi6)}, \eqref{(mpi3)}, \eqref{(mpi4)}, would 
not need a proof.  That is not so in our case, and we need separate proofs for these three cases.  We will just prove Equation~\eqref{(mpi6)} here.

For this, we will quote a result from the algebraic framework, namely Equation~(4.4) of \cite{VDWangwha1} that appear in the proof of 
Proposition~4.3 of that paper: It says 
$$
(c\otimes d)E=\sum_{(c)}c_{(1)}\otimes dS^{-1}(c_{(3)})c_{(2)}, \quad {\text { for $c,d\in{\mathcal A}$.}}
$$
Modifying it a little, we can obtain the following:
\begin{equation}\label{(multiplicativeW_eqn3)}
(c\otimes 1)E(1\otimes b)=\sum_{(c)}c_{(1)}\otimes S^{-1}(c_{(3)})c_{(2)}b, \quad {\text { for $c,b\in{\mathcal A}$.}}
\end{equation}
[There can be a bit more efficient way, using the idempotent map $F_3$, but to avoid digging too deep into the algebraic theory, we will 
make use of the above relation.] 

To prove Equation~\eqref{(mpi6)}, let $a,b,c\in{\mathcal A}$ be arbitrary, and compute.  Note first that 
\begin{align}
W^*_{13}W^*_{12}W_{12}\bigl(\Lambda(a)\otimes\Lambda(b)\otimes\Lambda(c)\bigr)
&=W^*_{13}(\Lambda\otimes\Lambda\otimes\Lambda)\bigl(E(a\otimes b)\otimes c\bigr) \notag \\
&=(\Lambda\otimes\Lambda\otimes\Lambda)\bigl((\Delta c)_{13}(E(a\otimes b)\otimes1)\bigr),
\label{(multiplicativeW_eqn4)}
\end{align}
because $W^*W=(\pi\otimes\pi)(E)$.  Meanwhile, we have
\begin{align}
W_{23}W^*_{23}W^*_{13}\bigl(\Lambda(a)\otimes\Lambda(b)\otimes\Lambda(c)\bigr)
&=W_{23}W^*_{23}(\Lambda\otimes\Lambda\otimes\Lambda)\bigl((\Delta c)_{13}(a\otimes b\otimes1)\bigr) \notag \\
&=W_{23}(\Lambda\otimes\Lambda\otimes\Lambda)\left(\sum_{(c)}c_{(1)}a\otimes c_{(2)}b\otimes c_{(3)}\right)  \notag \\
&=(\Lambda\otimes\Lambda\otimes\Lambda)\left(\sum_{(c)}c_{(1)}a\otimes S^{-1}(c_{(3)})c_{(2)}b\otimes c_{(4)}\right),
\notag
\end{align}
where we are using the characterization of $W$ given in Proposition~\ref{W}\,(2), expressed in the Sweedler notation.
Apply here the algebraic result Equation~\eqref{(multiplicativeW_eqn3)}.  Then it becomes
\begin{equation}\label{(multiplicativeW_eqn5)}
W_{23}W^*_{23}W^*_{13}\bigl(\Lambda(a)\otimes\Lambda(b)\otimes\Lambda(c)\bigr)
=(\Lambda\otimes\Lambda\otimes\Lambda)\bigl((\Delta c)_{13}(E(a\otimes b)\otimes1)\bigr).
\end{equation}
Comparing Equations~\eqref{(multiplicativeW_eqn4)} and \eqref{(multiplicativeW_eqn5)}, we conclude that 
$W^*_{13}W^*_{12}W_{12}=W_{23}W^*_{23}W^*_{13}$.  Or equivalently, we have $W^*_{12}W_{12}W_{13}=W_{13}W_{23}W^*_{23}$.

The other two can be proved similarly, using analogous algebraic results.  But as they are not directly needed below, and since they can be 
proved using an alternative method (see \cite{BJKVD_qgroupoid2}) once the full construction of the quantum groupoid is carried out, we will 
skip the details here.
\end{proof}

\begin{prop}\label{multiplicativeWW}
The operator $W$ satisfies the following conditions:
\begin{align}
W_{23}W_{12}W^*_{23}&=W_{12}W_{13} \label{(mpi1)}  \\
W_{12}^*W_{23}W_{12}&=W_{13}W_{23} \label{(mpi2)}  \\
W_{12}W^*_{23}&=W^*_{23}W_{12}W_{13} \label{(mpi7)}  \\
W_{12}^*W_{23}&=W_{13}W_{23}W^*_{12} \label{(mpi8)}  \\
W_{13}^*W_{13}W_{23}&=W_{23}W^*_{12}W_{12} \label{(mpi9)}  \\
W_{12}W_{13}W^*_{13}&=W_{23}W^*_{23}W_{12} \label{(mpi10)} 
\end{align}
\end{prop}

\begin{proof}
These can be all proved as a quick consequence of the four conditions that appear in Proposition~\ref{multiplicativeW}. Since 
we will be needing only the second one later, we will just give the proof for Equation~\eqref{(mpi2)}.  For more discussion, 
refer to \cite{BJK_mpi}.

From Proposition~\ref{multiplicativeW}, we saw that $W_{23}W_{12}=W_{12}W_{13}W_{23}$.  Multiply here $W^*_{12}$, 
from the left, to obtain: 
$$
W^*_{12}W_{23}W_{12}=W^*_{12}W_{12}W_{13}W_{23}=W_{13}W_{23}W^*_{23}W_{23}=W_{13}W_{23},
$$
using Equation~\eqref{(mpi6)} and the fact that $WW^*W=W^*$, being a partial isometry.
\end{proof}

\subsection{The idempotent $E$} \label{sub3.5}

We already noted that the canonical idempotent element $E$ at the ${}^*$-algebra level can be considered 
as the operator $E=(\pi\otimes\pi)(E)\in M(A\otimes A)\subset {\mathcal B}({\mathcal H}\otimes{\mathcal H})$, 
by the GNS-representation $\pi$.  As such, its properties will be inherited from those at the ${}^*$-algebra level, 
which we gather below. 

\begin{prop}\label{EatC*level}
Consider the canonical idempotent $E$, regarded as $E=(\pi\otimes\pi)(E)\in M(A\otimes A)\subset 
{\mathcal B}({\mathcal H}\otimes{\mathcal H})$, such that $E^*=E$ and $E^2=E$.  We have:
\begin{enumerate}
\item $\overline{\Delta(A)(A\otimes A)}^{\|\ \|}=E(A\otimes A)$ and $\overline{(A\otimes A)\Delta(A)}^{\|\ \|}=(A\otimes A)E$
\item $E(\Delta x)=\Delta x=(\Delta x)E$, for all $x\in A$
\item $E\otimes1$ and $1\otimes E$ commute, and we also have
$$
(\operatorname{id}\otimes\Delta)(E)=(E\otimes1)(1\otimes E)=(1\otimes E)(E\otimes1)=(\Delta\otimes\operatorname{id})(E).
$$
\item There exists a ${}^*$-anti-isomorphism $R=R_{BC}:B\to C$, and together with the KMS weight $\nu$ on $B$ and 
$E\in M(B\otimes C)$, we obtain a separability triple $(E,B,\nu)$, in the sense of \cite{BJKVD_SepId}.
\end{enumerate}
\end{prop}

\begin{proof} 
(1), (2) are the consequences of Lemma~\ref{canonicalE} at the ${}^*$-algebra level.  Since ${\mathcal A}$ is norm 
dense in $A$, the results follow immediately.  As shown in Proposition~3.3 in  \cite{BJKVD_qgroupoid1}, this uniquely 
determines $E$.

(3) is the weak comultiplicativity of the unit, noted already in Equation~\eqref{(weakcomultiplicativity)}.

(4). As $E\in M({\mathcal B}\odot{\mathcal C})$, it can be also considered as an element in $M(B\otimes C)$, 
where $B$ and $C$ are the $C^*$-subalgebras of $M(A)$ we saw earlier.  We saw in Proposition~\ref{Esepid} 
that $E$ satisfies the properties of being a {\em separability idempotent\/}, in the $C^*$-algebraic sense 
(see \cite{BJKVD_SepId}).

\end{proof}

In this way, we have shown that $E\in M(B\otimes C)\subseteq M(A\otimes A)$ satisfies all the conditions 
for being the {\em canonical idempotent\/} for $(A,\Delta)$, as prescribed in Definition~3.7 of \cite{BJKVD_qgroupoid1}.

\subsection{The weight $\tilde{\varphi}$ and the modular automorphism group} \label{sub3.6}

So far, our $\varphi$ has been only a linear functional on ${\mathcal A}$.  Going forward, we will need to consider its extension 
to the operator algebra level.  The full construction of a KMS weight on the $C^*$-algebra $A$, extending the functional 
$\varphi$ and satisfying a proper left invariant condition, needs some additional preparation so will be postponed to a later 
section.  However, we can do some initial work in this subsection.

Observe first that $\Lambda({\mathcal A})\subseteq{\mathcal H}$ obtained from the functional $\varphi$ at the 
${}^*$-algebra level, as given in the beginning part of Section~\ref{sec3}, is a left Hilbert algebra, with respect to 
the multiplication and the ${}*$-structure inherited from ${\mathcal A}$.  We will skip the proof, which is essentially 
no different in nature from that of Proposition~\ref{hilbertalgebra} in Section~\ref{sec2}.

By the general theory on left Hilbert algebras (see \cite{Tk2}), we can associate to $\Lambda({\mathcal A})$ 
a von Neumann algebra $M$.  In our case, in terms of the GNS-representation $\pi$, it would be exactly 
$M=\pi({\mathcal A})''$.  Also by the general theory on left Hilbert algebras, we obtain a normal semi-finite faithful (n.s.f.)
weight $\tilde{\varphi}$ on $M$.  

We can consider the associated spaces ${\mathfrak N}_{\tilde{\varphi}}=\bigl\{x\in M:\tilde{\varphi}(x^*x)<\infty\bigr\}$ 
and ${\mathfrak M}_{\tilde{\varphi}}={\mathfrak N}_{\tilde{\varphi}}^*{\mathfrak N}_{\tilde{\varphi}}$.  The associated 
GNS map $\Lambda_{\tilde{\varphi}}$ is an injective map from ${\mathfrak N}_{\tilde{\varphi}}$ to ${\mathcal H}$ 
(same Hilbert space), which extends $\Lambda$.  By the standard left Hilbert algebra theory, we know that the weight 
$\tilde{\varphi}$ extends the functional $\varphi$.  In particular, we have $\tilde{\varphi}\bigl(\pi(a)^*\pi(a)\bigr)=\varphi(a^*a)$, 
for all $a\in{\mathcal A}$.  For any $x\in{\mathfrak N}_{\tilde{\varphi}}$, there exists a sequence $(a_n)_n$ in ${\mathcal A}$ 
such that $\Lambda(a_n)\xrightarrow{{\text { (in ${\mathcal H}$) }}}\Lambda_{\tilde{\varphi}}(x)$ and 
$\pi(a_n)\xrightarrow{{\text { ($\sigma$-strong-${}^*$) }}}x$.

Denote by $T$ the closure of the involution $\Lambda(x)\mapsto\Lambda(x^*)$ on $\Lambda({\mathcal A})$. 
As before, there exists a polar decomposition, $T=J\nabla^{\frac12}$, where $\nabla=T^*T$ is the modular operator, 
and $J$ is the modular conjugation, which is anti-unitary.  Note also that $J\nabla J=\nabla^{-1}$, so we can also write 
$T=\nabla^{-\frac12}J$.

According to the modular theory in the von Neumann algebra setting (\cite{Tk2}, \cite{Str}), the modular operator defines a strongly 
continuous one-parameter group of automorphisms $\tilde{\sigma}=(\tilde{\sigma}_t)_{t\in\mathbb{R}}$, by $\tilde{\sigma}_t(a)=\nabla^{it}a\nabla^{-it}$, 
for $a\in M$, $t\in\mathbb{R}$.  We have $\tilde{\varphi}\circ\tilde{\sigma}_t=\tilde{\varphi}$, $t\in\mathbb{R}$, and $(\tilde{\sigma}_t)$ 
satisfies a certain KMS boundary condition.  In particular, the weak KMS property at the ${}^*$-algebra level, $\varphi(ab)
=\varphi\bigl(b\sigma(a)\bigr)$, $a,b\in{\mathcal A}$, extends to the von Neumann algebra as $\tilde{\varphi}(xy)
=\tilde{\varphi}\bigl(y\tilde{\sigma}_{-i}(x)\bigr)$, $x\in{\mathfrak M}_{\tilde{\varphi}}$, $y\in{\mathcal D}(\tilde{\sigma}_{-i})$.  
Meanwhile, the modular conjugation $J$ can be characterized by $J\Lambda_{\tilde{\varphi}}(x)
=\Lambda_{\tilde{\varphi}}\bigl(\tilde{\sigma}_{\frac{i}{2}}(x)^*\bigr)$, for $x\in{\mathfrak N}_{\tilde{\varphi}}$. 

From the n.s.f. weight $\tilde{\varphi}$ on the von Neumann $M=\pi({\mathcal A})''$, we can restrict it to the level 
of the $C^*$-algebra $A=\overline{\pi({\mathcal A})}^{\|\ \|}$, to obtain a faithful lower semi-continuous weight $\varphi$ 
on $A$.  However, as was the case for the weights $\tilde{\nu}$ and $\tilde{\mu}$ earlier (Section~\ref{sec2}), the main issues 
are whether the restriction of the modular automorphism group $(\tilde{\sigma}_t)$ to the $C^*$-algebra level would leave 
$A$ invariant, and whether the restriction is norm-continuous.  These are not automatic consequences of the modular theory, 
so this needs more work.  We will return to this matter in Section~\ref{sec5}.  

There is also the matter of properly establishing the left invariance property of the weight $\varphi$, and we also wish 
to construct a right-invariant weight $\psi$.  That is also postponed to Section~\ref{sec5}.

\section{Polar decomposition of the antipode}\label{sec4}

At present, the antipode map, $S$, is considered only at the ${}^*$-algebra level.  It is an anti-isomorphism of 
${\mathcal A}$ onto itself, which can be extended to the multiplier algebra $M({\mathcal A})$, such that 
$S({\mathcal B})={\mathcal C}$ and $S({\mathcal C})={\mathcal B}$. 

To regard it as a map at the operator algebra level, a convenient way is to express it as a polar decomposition.  This is much like 
the case of a locally compact quantum group \cite{KuVD}, \cite{MNW}, \cite{KuVa}, \cite{VDvN}, which has been also adopted to 
the cases of quantum groupoids \cite{LesSMF}, \cite{EnSMF}, \cite{BJKVD_qgroupoid2}. As such, some portion of the construction 
process below would look familiar.

\subsection{The involutive operator $K$ and its polar decomposition}\label{sub4.1}

Let us first define an anti-linear involutive operator $K$ below, which incorporates the antipode at the operator level:

\begin{prop}\label{K}
For $a\in{\mathcal A}$, define: 
$$
K_0\Lambda(a):=\Lambda\bigl(S(a)^*\bigr).
$$
Then:
\begin{enumerate}
\item $K_0$ is a well-defined map, from $\Lambda({\mathcal A})$ into ${\mathcal H}$.
\item $K_0$ is closable, so we may consider its closure $K$. It becomes a closed, densely-defined operator from ${\mathcal H}$ 
into ${\mathcal H}$, such that $\Lambda({\mathcal A})$ forms its core and it has a dense range.
\item $K$ is anti-linear and involutive.
\end{enumerate}
\end{prop}

\begin{proof}
As $S$ is well-defined from ${\mathcal A}$ onto itself, and since $\Lambda({\mathcal A})$ is dense in ${\mathcal H}$, it is clear that $K_0$ is 
well-defined, densely-defined, and has a dense range.  Meanwhile, for $a,b\in{\mathcal A}$, note that 
$$
\bigl\langle \Lambda(S(a)^*),\Lambda(b)\bigr\rangle=\varphi\bigl(b^*S(a)^*\bigr)=\overline{\varphi\bigl(S(a)b\bigr)}
=\overline{(\varphi\circ S)\bigl(S^{-1}(b)a\bigr)}=\overline{\varphi\bigl(S^{-1}(b)a\delta\bigr)}
=\overline{\varphi\bigl(a\delta\sigma(S^{-1}(b))\bigr)},
$$
where we used the fact that $\varphi$ is positive, and $\delta\in M({\mathcal A})$ is the modular element (see the last paragraph 
of \S\,\ref{sub1.5}).  From this, it is not difficult to show that $K_0$ is closable.  We can denote by $K$ its 
closure.  Thus $K$ becomes a closed, densely-defined operator having a dense range, with $\Lambda({\mathcal A})$ forming 
a core.

Finally, note that $K$ is anti-linear because $S$ is, and note also that $K$ is involutive by the property of the antipode:
$$
K^2\Lambda(a)=K\Lambda\bigl(S(a)^*\bigr)=K\Lambda\bigl(S^{-1}(a^*)\bigr)=\Lambda\bigl(S(S^{-1}(a^*))^*\bigr)
=\Lambda(a),
$$
for any $a\in{\mathcal A}$.
\end{proof}

We can consider the polar decomposition of this linear involutive operator $K$, given by $K=IL^{\frac12}$, where $L=K^*K$ is 
a positive operator, and $I$ is a conjugate linear isomorphism. We will have $I^*=I$, $I^2=1$, and $ILI=L^{-1}$.  Also 
$K=L^{-\frac12}I$.  In addition, we can consider the ``scaling group'' $(\tau_t)_{t\in\mathbb{R}}$, a norm-continuous 
automorphism group on ${\mathcal B}({\mathcal H})$ given by $\tau_t(\,\cdot\,)=L^{it}\,\cdot\,L^{-it}$.  

Here is a useful result that relates the operator $W$ with $L$ and $\nabla$, where $\nabla$ is the operator arising from 
the modular theory for the n.s.f. weight $\tilde{\varphi}$, such that $T=J\nabla^{\frac12}$.

\begin{prop}\label{WLNabla}
For any $a,b\in{\mathcal A}$, we have:
$$
W(L\otimes\nabla)\bigl(\Lambda(a)\otimes\Lambda(b)\bigr)=(L\otimes\nabla)W\bigl(\Lambda(a)\otimes\Lambda(b)\bigr).
$$
\end{prop}

\begin{proof}
For $a,b\in{\mathcal A}$, we have:
\begin{align}
W^*(K\otimes T)\bigl(\Lambda(a)\otimes\Lambda(b)\bigr)&=W^*\bigl(\Lambda(S(a)^*)\otimes\Lambda(b^*)\bigr)  \notag \\
&=(\Lambda\otimes\Lambda)\bigl(\Delta(b^*)(S(a)^*\otimes1)\bigr) 
=(\Lambda\otimes\Lambda)\bigl([(S(a)\otimes1)(\Delta b)]^*\bigr)  \notag \\
&=(\Lambda\otimes\Lambda)\bigl([(S\otimes\operatorname{id})((S^{-1}\otimes\operatorname{id})(\Delta b)(a\otimes1))]^*\bigr)
\notag \\
&=(K\otimes T)(\Lambda\otimes\Lambda)\bigl((S^{-1}\otimes\operatorname{id})(\Delta b)(a\otimes1)\bigr)  \notag \\
&=(K\otimes T)W\bigl(\Lambda(a)\otimes\Lambda(b)\bigr),
\label{(WLNabla_eqn1)}
\end{align}
using Proposition~\ref{W}. On the other hand, we also have 
\begin{align}
W(K\otimes T)\bigl(\Lambda(a)\otimes\Lambda(b)\bigr)&=W\bigl(\Lambda(S(a)^*)\otimes\Lambda(b^*)\bigr)  
=(\Lambda\otimes\Lambda)\bigl((S^{-1}\otimes\operatorname{id})(\Delta(b)^*)(S(a)^*\otimes1)\bigr)  \notag \\
&=(\Lambda\otimes\Lambda)\bigl([(S(a)\otimes1)(S\otimes\operatorname{id})(\Delta b)]^*\bigr)
\notag \\
&=(\Lambda\otimes\Lambda)\bigl([(S\otimes\operatorname{id})((\Delta b)(a\otimes1))]^*\bigr)
=(K\otimes T)(\Lambda\otimes\Lambda)\bigl((\Delta b)(a\otimes1)\bigr)  \notag \\
&=(K\otimes T)W^*\bigl(\Lambda(a)\otimes\Lambda(b)\bigr),
\label{(WLNabla_eqn2)}
\end{align} 
from which we have
\begin{equation}\label{(WLNabla_eqn3)}
W(K^*\otimes T^*)\bigl(\Lambda(a)\otimes\Lambda(b)\bigr)=(K^*\otimes T^*)W^*\bigl(\Lambda(a)\otimes\Lambda(b)\bigr), 
\quad {\text { for $a,b\in{\mathcal A}$,}}
\end{equation}
by taking the adjoint.

Combining Equations~\eqref{(WLNabla_eqn1)} and \eqref{(WLNabla_eqn3)}, we have:
\begin{align}
W(L\otimes\nabla)\bigl(\Lambda(a)\otimes\Lambda(b)\bigr)&=W(K^*K\otimes T^*T)\bigl(\Lambda(a)\otimes\Lambda(b)\bigr)
=(K^*\otimes T^*)W^*(K\otimes T)\bigl(\Lambda(a)\otimes\Lambda(b)\bigr) \notag \\
&=(K^*K\otimes T^*T)W\bigl(\Lambda(a)\otimes\Lambda(b)\bigr)
=(L\otimes\nabla)W\bigl(\Lambda(a)\otimes\Lambda(b)\bigr).
\notag 
\end{align}
\end{proof}

Note, by the way, that since we are working with unbounded operators $L$ and $\nabla$, which are only densely-defined, 
the above result does not necessarily mean $W(L\otimes\nabla)=(L\otimes\nabla)W$.  As it is possible that some of the 
elements contained in $\operatorname{Ker}(W)$ may not be contained in ${\mathcal D}(L\otimes\nabla)$, we would have 
${\mathcal D}\bigl(W(L\otimes\nabla)\bigr)\subsetneq{\mathcal D}\bigl((L\otimes\nabla)W\bigr)$. So, to be precise, we can 
only write $W(L\otimes\nabla)\subseteq(L\otimes\nabla)W$, or $W^*(L\otimes\nabla)\subseteq(L\otimes\nabla)W^*$.

If $W$ was a unitary operator, then there exists a clever method one can use to quickly establish $W(L\otimes\nabla)=(L\otimes\nabla)W$ 
(see, for instance, Lemma~5.9 in \cite{KuVa}), which helps us proceed along. That, however, is not the case here.  We need a more 
roundabout approach.  We can almost verbatim follow the steps carried out in Propositions~4.13 -- 4.17 in \cite{BJKVD_qgroupoid2}, 
and we obtain the following.  We will skip the details of the proof, but see the remark following the proposition:

\begin{prop}\label{Lnabla_main}
\begin{enumerate}
\item For any $z\in\mathbb{C}$, we have:
$$
W(L^z\otimes\nabla^z)\subseteq(L^z\otimes\nabla^z)W \quad 
{\text { and }} \quad
W^*(L^z\otimes\nabla^z)\subseteq(L^z\otimes\nabla^z)W^*.
$$
\item Let $t\in\mathbb{R}$.  Then the following is true on the whole space ${\mathcal H}\otimes{\mathcal H}$:
$$
W(L^{it}\otimes\nabla^{it})=(L^{it}\otimes\nabla^{it})W \quad 
{\text { and }} \quad
W^*(L^{it}\otimes\nabla^{it})=(L^{it}\otimes\nabla^{it})W^*.
$$
\end{enumerate}
\end{prop}

\begin{rem}
Our Proposition~\ref{WLNabla} gives rise to the same conclusion as in Proposition~4.12 in \cite{BJKVD_qgroupoid2}. 
The contexts are different, so of course the proofs are not the same.  Nonetheless, the steps taken in the paragraphs 
following that proposition can be taken almost word for word, up to Propositions~4.17 in \cite{BJKVD_qgroupoid2}.  That is 
because this is more about dealing with the unbounded operators in general than any specifics about quantum groupoids. 

The results (1), (2) in Proposition~\ref{Lnabla_main} above are basically Propositions~4.17 and 4.18 of 
\cite{BJKVD_qgroupoid2}.  The key point is that for any $z\in\mathbb{C}$, it turns out that 
$(L^z\otimes\nabla^z)|_{\operatorname{Ran}(E)}$ acts as an operator on ${\operatorname{Ran}(E)}$, where $E=W^*W$, 
and also that $(L^z\otimes\nabla^z)|_{\operatorname{Ran}(E)^{\perp}}$ acts as an operator on 
${\operatorname{Ran}(E)^{\perp}}=\operatorname{Ker}(W)$.  Since ${\mathcal H}\otimes{\mathcal H}
={\operatorname{Ran}(W^*W)}\otimes\operatorname{Ker}(W)$, the partial isometry $W$ behaves as a unitary map 
from $\operatorname{Ran}(W^*W)$ onto $\operatorname{Ker}(W^*)^{\perp}=\operatorname{Ran}(WW^*)$, and similarly 
for the partial isometry $W^*$, allowing us the result that resembles the case when $W$ is unitary.

Since $W$ is not unitary, we only have ``$\subseteq$'' in general.  On the other hand, the situation improves 
if $z\in\mathbb{C}$ is purely imaginary, or $z=it$, $t\in\mathbb{R}$. If so, the operators $L^{it}$ and $\nabla^{it}$ are bounded, 
so the domain of $L^{it}\otimes\nabla^{it}$ becomes the whole space ${\mathcal H}\otimes{\mathcal H}$.
\end{rem}

Next is another useful result, relating the operator $W$ with the anti-unitary operators $I$ (coming from $K=IL^{\frac12}$) 
and $J$ (coming from $T=J\nabla^{\frac12}$):

\begin{prop}\label{IJWequality}
We have:
$$
(I\otimes J)W(I\otimes J)=W^*.
$$
\end{prop}

\begin{proof}
Let $a,b\in{\mathcal A}$ be arbitrary. By Proposition~\ref{Lnabla_main}\,(1), when $z=\frac12$, we have:
\begin{align}
&(I\otimes J)W(L^{\frac12}\otimes\nabla^{\frac12})\big(\Lambda(a)\otimes\Lambda(b)\bigr)
=(I\otimes J)(L^{\frac12}\otimes\nabla^{\frac12})W\big(\Lambda(a)\otimes\Lambda(b)\bigr)  \notag \\
&=(K\otimes T)W\big(\Lambda(a)\otimes\Lambda(b)\bigr)=W^*(K\otimes T)\big(\Lambda(a)\otimes\Lambda(b)\bigr)
\notag \\
&=W^*(I\otimes J)(L^{\frac12}\otimes\nabla^{\frac12})\big(\Lambda(a)\otimes\Lambda(b)\bigr), 
\notag
\end{align}
where we used Equation~\eqref{(WLNabla_eqn1)}.  Even though these expressions involve unbounded operators, note that 
$(L^{\frac12}\otimes\nabla^{\frac12})$ has a dense range in ${\mathcal H}\otimes{\mathcal H}$, while $(I\otimes J)W$ and 
$W^*(I\otimes J)$ are bounded operators, so no domain issues.  Therefore, we can conclude that 
$$
(I\otimes J)W=W^*(I\otimes J),
$$
on the whole space ${\mathcal H}\otimes{\mathcal H}$. This is equivalent to saying that $(I\otimes J)W(I\otimes J)=W^*$.
\end{proof}

\subsection{Antipode map in terms of its polar decomposition}\label{sub4.2}

Now that we have gathered some results regarding the operators $K$ and $T$, including their polar decompositions, we can 
carry on with the construction of the polar decomposition for the antipode.

Note first that the scaling group on ${\mathcal B}({\mathcal H})$ can be defined at the level of the $C^*$-algebra $A$ and also 
at the level of the von Neumann algebra $M$:

\begin{prop}\label{scalinggroup}
For $x\in A$, define $\tau_t(x):=L^{it}xL^{-it}$, for $t\in\mathbb{R}$.  Then $\tau_t(A)=A$, and $(\tau_t)_{t\in\mathbb{R}}$ is 
a norm-continuous one-parameter group of automorphisms on $A$, referred to as the scaling group on $A$.

Similarly, $M\in z\mapsto\tau_t(z)=L^{it}zL^{-it}$, $t\in\mathbb{R}$, defines a strongly-continuous one-parameter group of
automorphisms on $M$.
\end{prop}

\begin{proof}
From Proposition~\ref{Lnabla_main}\,(2), we know $(L^{it}\otimes\nabla^{it})W=W(L^{it}\otimes\nabla^{it})$, for all $t\in\mathbb{R}$. 
Multiply $(\operatorname{id}\otimes\nabla^{-it})$ from the left and multiply $(L^{-it}\otimes\operatorname{id})$ from the right, to obtain:
$$
(L^{it}\otimes\operatorname{id})W(L^{-it}\otimes\operatorname{id})=(\operatorname{id}\otimes\nabla^{-it})W(\operatorname{id}\otimes\nabla^{it}).
$$
Then we can see that for any $\omega\in{\mathcal B}({\mathcal H})_*$, we have:
\begin{equation}\label{(scalingeqn)}
L^{it}(\operatorname{id}\otimes\omega)(W)L^{-it}=\bigl(\operatorname{id}\otimes
\omega(\nabla^{-it}\,\cdot\,\nabla^{it})\bigr)(W).
\end{equation}

Recall that the elements of the form $x=(\operatorname{id}\otimes\omega)(W)$ generate the $C^*$-algebra $A=\overline{\pi({\mathcal A})}^{\|\ \|}$ 
(see Definition~\ref{C*algebraA}) as well as the von Neumann algebra $M=\pi({\mathcal A})''$.  Therefore, we can see quickly from 
Equation~\eqref{(scalingeqn)} that $\tau_t(A)=A$ and $\tau_t(M)=M$, for all $t\in\mathbb{R}$.  The norm-continuity of 
$\bigl(\tau_t|_A\bigr)_{t\in\mathbb{R}}$ and the strong continuity of $\bigl(\tau_t|_M\bigr)_{t\in\mathbb{R}}$ also follow immediately.
\end{proof}

The result of Proposition~\ref{scalinggroup} indicates that the elements of the form $(\operatorname{id}\otimes\omega)(W)$, 
$\omega\in{\mathcal B}({\mathcal H})_*$, span a core for $\tau$.  We can also say that the elements $\pi(a)$, $a\in{\mathcal A}$, are analytic elements 
of $\tau$.  

\begin{prop}\label{analyticfortau}
Let $a\in{\mathcal A}$  Then $\pi(a)$ is an analytic element of $\tau$, and that 
$$
\tau_{ni}\bigl(\pi(a)\bigr)=\pi\bigl(S^{-2n}(a)\bigr), \quad {\text { for any $n\in\mathbb{Z}$.}}
$$
\end{prop}

\begin{proof}
Let $a\in{\mathcal A}$.  For any $b\in{\mathcal A}$, note that
\begin{align}
\pi(a)K\Lambda(b)&=\pi(a)\Lambda\bigl(S(b)^*\bigr)=\Lambda\bigl(aS(b)^*\bigr)=\Lambda\bigl([S(b)a^*]^*\bigr)
=\Lambda\bigl([S(S^{-1}(a^*)b)]^*\bigr)   \notag  \\
&=\Lambda\bigl([S(S(a)^*b)]^*\bigr)=K\Lambda\bigl(S(a)^*b\bigr)=K\pi\bigl(S(a)^*\bigr)\Lambda(b).
\notag
\end{align}
From this, we can conclude that $\pi(a)K\subseteq K\pi\bigl(S(a)^*\bigr)$.  Similarly, we have 
$\pi(a)K^*\subseteq K^*\pi\bigl(S(a^*)\bigr)$.

Recall that $L=K^*K$.  The above observations mean that we have the following:
\begin{align}
\pi(a)L&=\pi(a)K^*K\subseteq K^*\pi\bigl(S(a^*)\bigr)K  \notag \\
&\subseteq K^*K\pi\bigl(S(S(a^*))^*\bigr)
=K^*K\pi\bigl(S^{-1}(S(a^*)^*)\bigr)=K^*K\pi\bigl(S^{-1}(S^{-1}(a)\bigr)
=L\pi\bigl(S^{-2}(a)\bigr).
\notag 
\end{align}
It follows that for any $n\in\mathbb{Z}$, we have:
\begin{equation}\label{(analyticfortau_eqn)}
\pi(a)L^n\subseteq L^n\pi\bigl(S^{-2n}(a)\bigr),
\end{equation}
which is true for all $a\in{\mathcal A}$. 

In this way, we can see that for any $a\in{\mathcal A}$, we have $\pi(a)\in{\mathcal D}(\tau_{ni})$, and that 
$\tau_{ni}\bigl(\pi(a)\bigr)=\pi\bigl(S^{-2n}(a)\bigr)$, for all $n\in\mathbb{Z}$. Continuing, we can also see that 
$\pi(a)\in{\mathcal D}(\tau_z)$, for any $z\in\mathbb{C}$.
\end{proof}

The other main ingredient for the polar decomposition of the antipode is the following ${}^*$-anti-isomorphism $R_A:A\to A$, 
called the unitary antipode.

\begin{prop}\label{unitaryantipode}
For $x\in A$, define $R_A(x)$ by $R_A(x)=Ix^*I$.  It is a ${}^*$-anti-isomporphism from $A$ onto $A$, and is involutive.  In particular, 
for any $\omega\in{\mathcal B}({\mathcal H})_*$, we have:
$$
R:(\operatorname{id}\otimes\omega)(W)\mapsto(\operatorname{id}\otimes\theta)(W),
$$
where $\theta\in{\mathcal B}({\mathcal H})_*$ is such that $\theta=\bar{\omega}(J\,\cdot\,J)$.

It has a natural extension $\widetilde{R}_M:M\to M$, which is a ${}^*$-anti-isomporphism from $M$ onto $M$.
\end{prop}

\begin{proof}
Since $I$ is an anti-unitary, it is evident that $R_A:x\mapsto Ix^*I$ is anti-multiplicative and that 
$R_A(x^*)=R_A(x)^*$, for all $x\in A$.  

To see if $R_A$ is indeed a map from $A$ onto itself, recall first Proposition~\ref{IJWequality}, saying 
$(I\otimes J)W=W^*(I\otimes J)$. It may be re-written as 
\begin{equation}\label{(unitaryantipode_eqn)}
(I\otimes1)W^*(I\otimes1)=(1\otimes J)W(1\otimes J).
\end{equation}
Consider $x=(\operatorname{id}\otimes\omega)(W)$, for $\omega\in{\mathcal B}({\mathcal H})_*$.  Then 
by Equation~\eqref{(unitaryantipode_eqn)}, we have:
\begin{align}
Ix^*I&=I(\operatorname{id}\otimes\bar{\omega})(W^*)I
=(\operatorname{id}\otimes\bar{\omega})\bigl((I\otimes1)W^*(I\otimes1)\bigr)
\notag \\
&=(\operatorname{id}\otimes\bar{\omega})\bigl((1\otimes J)W(1\otimes J)\bigr)
=(\operatorname{id}\otimes\theta)(W),
\notag
\end{align}
where $\theta\in{\mathcal B}({\mathcal H})_*$ is such that $\theta=\bar{\omega}(J\,\cdot\,J)$.  Since $A$ (and $M$) are generated by the elements 
of the form $(\operatorname{id}\otimes\omega)(W)$, for $\omega\in{\mathcal B}({\mathcal H})_*$, we can see that the map $x\mapsto Ix^*I$ sends 
a dense subspace of $A$ into a dense subspace of $A$.  But then, since $R_A$ is a ${}^*$-anti-homomorphism, it must be bounded.  We can 
thus conclude that $R_A$ extends to all of $A$, becoming a ${}^*$-anti-automorphism from $A$ onto $A$.  We can say the same for 
$\widetilde{R}_M:M\to M$, which is a ${}^*$-anti-automorphism from $M$ onto $M$.
\end{proof}

\begin{rem}
Since $R_A$ is a ${}^*$-anti-isomorphism, it can be naturally extended to the level of the multiplier algebra $M(A)$.  Or as a restriction of 
$\widetilde{R}_M:M\to M$ to the level of $M(A)$.  For this map $R_A:M(A)\to M(A)$, it can be shown (see \cite{BJKVD_qgroupoid2}) that 
$R_A(B)=C$ and that the restriction of $R_A$ to $B$ coincides with the map $R_{BC}:B\to C$ that appears in Proposition~\ref{EatC*level}. 
For this reason, we have been rather lazy with the notations in Sections~\ref{sec2} and \ref{sec3}, and we will continue to do so rest of the way: 
The ${}^*$-anti-isomorphism will be all written as $R$, whether it is considered as the unitary antipode $R_A$ on $A$, or as $R_{BC}:B\to C$, 
or its inverse $R_{CB}:C\to B$.
\end{rem}

\begin{prop}\label{Rtaucoomute}
$R$ and $\tau_t$ commute, for all $t\in\mathbb{R}$.  That is, $R\circ\tau_t=\tau_t\circ R$.
\end{prop}

\begin{proof} 
We know that $ILI=L^{-1}$.  Since $I$ is conjugate linear, we thus have $IL^{it}I=L^{it}$.  Therefore, we have
$$
R\bigl(\tau_t(x)\bigr)=I\tau_t(x)^*I=IL^{it}x^*L^{-it}I=L^{it}Ix^*L^{-it}=\tau_t\bigl(R(x)\bigr).
$$
\end{proof}

Consider $\tau_{-\frac{i}2}$, the analytic generator of $\tau=(\tau_t)$ at $t=-\frac{i}2$.  From Proposition~\ref{analyticfortau}, we can deduce that 
it is a densely-defined map. Using this map, we are now ready to describe the polar decomposition of the antipode. 

\begin{thm}\label{antipodepolardecomposition}
Consider the densely-defined map $S$ on $A$, defined by
$$
S:=R\circ\tau_{-\frac{i}{2}}.
$$
This is the {\em antipode\/}, extending the antipode map at the level of ${\mathcal A}$.  In fact, for any $a\in{\mathcal A}$, we have:
$$
\pi\bigl(S(a)\bigr)=R\bigl(\tau_{-\frac{i}2}(\pi(a))\bigr).
$$

At the von Neumann algebra level, on $M=\pi({\mathcal A})''$, we can consider the extended antipode map, 
$\widetilde{S}=\widetilde{R}\circ\tau_{-\frac{i}2}$.
\end{thm}

\begin{proof}
We saw in the proof of Proposition~\ref{analyticfortau} that for any $a\in{\mathcal A}$, we have $\pi(a)K\subseteq K\pi\bigl(S(a)^*\bigr)$. Since 
$K=IL^{\frac12}=L^{-\frac12}I$ (shown earlier, which can be proved using $ILI=L^{-1}$), it becomes 
$$
\pi(a)L^{-\frac12}I\subseteq L^{-\frac12}I\pi\bigl(S(a)^*\bigr).
$$
Then we have
$$
\pi(a)L^{-\frac12}\subseteq L^{-\frac12}I\pi\bigl(S(a)^*\bigr)I=L^{-\frac12}R\bigl(\pi(S(a))\bigr).
$$
It follows that $\tau_{-\frac{i}2}\bigl(\pi(a)\bigr)=R\bigl(\pi(S(a))\bigr)$, or equivalently, $R\bigl(\tau_{-\frac{i}2}(\pi(a))\bigr)=\pi\bigl(S(a)\bigr)$, 
because $R$ is involutive.
\end{proof}

Theorem~\ref{antipodepolardecomposition} agrees with Definition~4.25 of \cite{BJKVD_qgroupoid2}.  In this way, the antipode 
map can be properly considered at the operator algebra level.  The definition also resembles the quantum group case, such as 
Definition~5.21 of \cite{KuVa}.

We can refer to sections 4 and 5 of \cite{BJKVD_qgroupoid2} for more properties of the antipode map at the $C^*$-algebra 
level, once our full construction is done.  Nonetheless, some of those results are actually needed for our purposes in a later 
section.  Among those, here are some results involving $R$, $\tau$, $\tilde{\sigma}$, $\Delta$.  The following four propositions 
(Propositions~\ref{Deltasigma_t}, \ref{Deltatau_t}, \ref{DeltaR}, \ref{tausigmaR_corollaryE}) are none other than 
Propositions~5.1, 5.2, 5.4, and 5.5 of \cite{BJKVD_qgroupoid2}). 

\begin{prop}\label{Deltasigma_t}
For all $x\in A$ and $t\in\mathbb{R}$, we have:
$\Delta\bigl(\tilde{\sigma}_t(x)\bigr)=(\tau_t\otimes\tilde{\sigma}_t)(\Delta x)$.
\end{prop}

\begin{proof}
Let $x\in A$. Then
\begin{align}
\Delta\bigl(\tilde{\sigma}_t(x)\bigr)&=W^*\bigl(1\otimes\tilde{\sigma}_t(x)\bigr)W=W^*\bigl(1\otimes\nabla^{it}x\nabla^{-it}\bigr)W
\notag \\
&=W^*(L^{it}\otimes\nabla^{it})(1\otimes x)(L^{-it}\otimes\nabla^{-it})W  \notag \\
&=(L^{it}\otimes\nabla^{it})W^*(1\otimes x)W(L^{-it}\otimes\nabla^{-it})
=(\tau_t\otimes\tilde{\sigma}_t)(\Delta x),
\notag
\end{align}
by Proposition~\ref{Lnabla_main}\,(2).
\end{proof}

\begin{prop}\label{Deltatau_t}
For all $x\in A$ and $t\in\mathbb{R}$, we have:
$\Delta\bigl(\tau_t(x)\bigr)=(\tau_t\otimes\tau_t)(\Delta x)$.
\end{prop}

\begin{proof}
For any $a\in A$, by the coassociativity of $\Delta$ and by Proposition~\ref{Deltasigma_t}, we have
$$
(\Delta\otimes\operatorname{id})\Delta\bigl(\tilde{\sigma}_t(a)\bigr)
=(\operatorname{id}\otimes\Delta)\Delta\bigl(\tilde{\sigma}_t(a)\bigr)
=(\tau_t\otimes\tau_t\otimes\tilde{\sigma}_t)\bigl((\operatorname{id}\otimes\Delta)(\Delta a)\bigr).
$$
Applying again the proposition (in the left) and the coassociativity (in the right), this becomes:
$$
\bigl((\Delta\circ\tau_t)\otimes\tilde{\sigma}_t)(\Delta a)
=(\tau_t\otimes\tau_t\otimes\tilde{\sigma}_t)\bigl((\Delta\otimes\operatorname{id})(\Delta a)\bigr).
$$
By applying the automorphism $\tilde{\sigma}_{-t}$ to the third leg, this becomes:
$$
\bigl((\Delta\circ\tau_t)\otimes\operatorname{id}\bigr)(\Delta a)
=\bigl((\tau_t\otimes\tau_t)\circ\Delta\otimes\operatorname{id}\bigr)(\Delta a).
$$
Apply now $\operatorname{id}\otimes\operatorname{id}\otimes\omega$, for $\omega\in{\mathcal A}^*$.  Then we have:
$$
\Delta\bigl(\tau_t((\operatorname{id}\otimes\omega)(\Delta a))\bigr)
=(\tau_t\otimes\tau_t)\bigl(\Delta((\operatorname{id}\otimes\omega)(\Delta a))\bigr).
$$
Note that the elements of the form $(\operatorname{id}\otimes\omega)(\Delta a)$, for $a\in A$, $\omega\in A^*$, generate 
all of $A$.  It follows that $\Delta\bigl(\tau_t(x)\bigr)=(\tau_t\otimes\tau_t)(\Delta x)$, for all $x\in A$.
\end{proof}

\begin{prop}\label{DeltaR}
For all $x\in A$, we have: 
$$
(R\otimes R)(\Delta x)=\Delta^{\operatorname{cop}}\bigl(R(x)\bigr),
$$
where $\Delta^{\operatorname{cop}}$ denotes the coopposite comultiplication, given by $\Delta^{\operatorname{cop}}(x)
=\varsigma\bigl(\Delta(x)\bigr)$.

At the von Neumann algebra level, we have: 
$(\widetilde{R}\otimes\widetilde{R})(\Delta x)=\Delta^{\operatorname{cop}}\bigl(\widetilde{R}(x)\bigr)$, for $x\in M$.
\end{prop}

\begin{proof}
We know that the elements of the form $(\operatorname{id}\otimes\omega)(W)$, for $\omega\in{\mathcal B}({\mathcal H})_*$, 
generate all of $A$ (and $M$). Therefore, to prove the above result, it is sufficient to prove the following:
$$
(R\otimes R)\bigl(\Delta((\operatorname{id}\otimes\omega)(W))\bigr)
=\Delta^{\operatorname{cop}}\bigl(R((\operatorname{id}\otimes\omega)(W))\bigr).
$$

Remembering the definition of $R$ (see Proposition~\ref{unitaryantipode}), we have:
\begin{align}
(R\otimes R)\bigl(\Delta((\operatorname{id}\otimes\omega)(W))\bigr)
&=(I\otimes I)\Delta\bigl((\operatorname{id}\otimes\omega)(W)^*\bigr)(I\otimes I) \notag \\
&=(I\otimes I)W^*\bigl(1\otimes(\operatorname{id}\otimes\bar{\omega})(W^*)\bigr)W(I\otimes I) \notag \\
&=(I\otimes I)(\operatorname{id}\otimes\operatorname{id}\otimes\bar{\omega})(W^*_{12}W^*_{23}W_{12})(I\otimes I)  \notag \\
&=(I\otimes I)(\operatorname{id}\otimes\operatorname{id}\otimes\bar{\omega})(W^*_{23}W^*_{13})(I\otimes I)  \notag \\
&=(\operatorname{id}\otimes\operatorname{id}\otimes\bar{\omega})\bigl((I\otimes I\otimes1)W^*_{23}W^*_{13}(I\otimes I\otimes1)\bigr),
\notag
\end{align}
where we used Equation~\eqref{(mpi2)}.  Since $(I\otimes J)W^*(I\otimes J)=W$ from Proposition~\ref{IJWequality}, we see 
that $(I\otimes1)W^*(I\otimes1)=(1\otimes J)W(1\otimes J)$. Using this, the above expression becomes:
\begin{equation}\label{(DeltaR_eqn1)}
(R\otimes R)\bigl(\Delta((\operatorname{id}\otimes\omega)(W))\bigr)
=(\operatorname{id}\otimes\operatorname{id}\otimes\bar{\omega})\bigl((1\otimes1\otimes J)W_{23}W_{13}(1\otimes1\otimes J)\bigr).
\end{equation}

Meanwhile, we also have
\begin{align}
\Delta\bigl(R((\operatorname{id}\otimes\omega)(W)\bigr)&=\Delta\bigl(I(\operatorname{id}\otimes\omega)(W)^*I\bigr)
=W^*\bigl(1\otimes I(\operatorname{id}\otimes\bar{\omega})(W^*)I\bigr)W   \notag \\
&=(\operatorname{id}\otimes\operatorname{id}\otimes\bar{\omega})\bigl(W^*_{12}(1\otimes I\otimes1)W^*_{23}(1\otimes I\otimes1)W_{12}\bigr)
\notag \\
&=(\operatorname{id}\otimes\operatorname{id}\otimes\bar{\omega})\bigl(W^*_{12}(1\otimes1\otimes J)W_{23}(1\otimes1\otimes J)W_{12}\bigr)
\notag \\
&=(\operatorname{id}\otimes\operatorname{id}\otimes\bar{\omega})\bigl((1\otimes1\otimes J)W^*_{12}W_{23}W_{12}(1\otimes1\otimes J)\bigr)
\notag \\
&=(\operatorname{id}\otimes\operatorname{id}\otimes\bar{\omega})\bigl((1\otimes1\otimes J)W_{13}W_{23}(1\otimes1\otimes J)\bigr),
\notag
\end{align}
using again $(I\otimes1)W^*(I\otimes1)=(1\otimes J)W(1\otimes J)$ and also using Equation~\eqref{(mpi2)}. Applying the flip map, we then have:
\begin{equation}\label{(DeltaR_eqn2)}
\Delta^{\operatorname{cop}}\bigl(R((\operatorname{id}\otimes\omega)(W)\bigr)
=(\operatorname{id}\otimes\operatorname{id}\otimes\bar{\omega})\bigl((1\otimes1\otimes J)W_{23}W_{13}(1\otimes1\otimes J)\bigr).
\end{equation}

Comparing Equations~\eqref{(DeltaR_eqn1)} and \eqref{(DeltaR_eqn2)}, we prove the result: 
$(R\otimes R)\bigl(\Delta((\operatorname{id}\otimes\omega)(W))\bigr)=\Delta^{\operatorname{cop}}\bigl(R((\operatorname{id}\otimes\omega)(W))\bigr)$.
\end{proof}

\begin{prop}\label{tausigmaR_corollaryE}
As an immediate consequence of Propositions~\ref{Deltasigma_t}, \ref{Deltatau_t}, 
\ref{DeltaR}, we have, for all $t\in\mathbb{R}$, the following results: 
\begin{enumerate}
  \item $(\tau_t\otimes\sigma_t)(E)=E$
  \item $(\tau_t\otimes\tau_t)(E)=E$
  \item $(R\otimes R)(E)=\varsigma E$
\end{enumerate}
\end{prop}

\begin{proof}
The results follow from the earlier propositions, with $E=\Delta(1)$.
\end{proof}

\section{Left and right invariant weights}\label{sec5}

We have so far established our $C^*$-algebra $A$; the comultiplication $\Delta$; the base $C^*$-algebra $B$ (and $C$) 
as a $C^*$-subalgebra of $M(A)$; a KMS weight $\nu$ on $B$ (and a KMS weight $\mu$ on $C$); as well as the canonical idempotent $E$.  
We have formulated a polar decomposition for the antipode $S$, making it properly defined as an unbounded map at the $C^*$-algebra level.  
In view of the definition of a $C^*$-algebraic quantum groupoid of separable type (Definition~4.8 of  \cite{BJKVD_qgroupoid1} and 
Definition~1.2 of  \cite{BJKVD_qgroupoid2}), what remains is showing the existence of a suitable left invariant weight and a right invariant weight. 

Some initial work was carried out in \S\ref{sub3.6}, where we constructed an n.s.f. weight $\tilde{\varphi}$ at the von Neumann algebra level, 
extending the functional $\varphi$ at the ${}^*$-algebra level.  We looked at some of its modular theoretic data, such as the operators $\nabla$, $J$, 
and the modular automorphism group $(\tilde{\sigma}_t)_{t\in\mathbb{R}}$.  However, there are still a good amount of work remaining: We need to 
obtain a KMS weight $\varphi$ at the $C^*$-algebra level; We also need to establish the left invariance of the weight $\varphi$; Finally, there 
is also the task of constructing a right invariant weight $\psi$.  These are the tasks that we will carry out in the subsections below.

\subsection{The left invariance property of $\tilde{\varphi}$}\label{sub5.1}

Consider the weight $\tilde{\varphi}$, which we constructed in subsection~\S\ref{sub3.6}.  It is an n.s.f. weight at the 
von Neumann algebra level, which we noted to extend the functional $\varphi$ on ${\mathcal A}$.  Note also that by 
the general modular theory (see Theorem VI.1.26 of \cite{Tk2}), for any $x\in{\mathfrak N}_{\tilde{\varphi}}$, there exists 
a sequence $(a_n)_n$ in ${\mathcal A}$ such that $\Lambda(a_n)\xrightarrow{{\text { (in ${\mathcal H}$) }}}
\Lambda_{\tilde{\varphi}}(x)$ and $\pi(a_n)\xrightarrow{{\text { ($\sigma$-strong-${}^*$) }}}x$.  As such, we expect that 
the left invariance property of $\varphi$ (Definition~\ref{invariantintegrals}) would carry over to $\tilde{\varphi}$. 
Let us discuss this matter here. 

\begin{prop}\label{lemma_invariance}
\begin{enumerate}
\item Let $x\in{\mathfrak N}_{\tilde{\varphi}}$ and let $\omega\in A^*_+$.  Then $(\omega\otimes\operatorname{id})\bigl(\Delta(x^*x)\bigr)
\in\mathfrak M_{\tilde{\varphi}}^+$.
\item Let $x\in{\mathfrak N}_{\tilde{\varphi}}$ and let $\omega\in A^*$.  Then $(\omega\otimes\operatorname{id})(\Delta x)
\in\mathfrak N_{\tilde{\varphi}}$.
\end{enumerate}
\end{prop}

\begin{proof}
(1). Without loss of generality, we may write $\omega\in A^*_+$ as $\omega=\theta(y^*\,\cdot\,y)$, for some $\theta\in A^*_+$ and $y\in A$. 
Find sequences $(a_n)$ and $(b_n)$ in ${\mathcal A}$, such that $(b_n)_{n=1}^{\infty}\to y$, $(a_n)_{n=1}^{\infty}\to x$, 
and $\bigl(\Lambda(a_n)\bigr)_{n=1}^{\infty}\to\Lambda_{\tilde{\varphi}}(x)$. For convenience, we may regard $a_n=\pi(a_n)$ 
and $b_n=\pi(b_n)$, so that they can be also viewed as elements in $\pi({\mathcal A})$. Then we will have
$\bigl((\theta(b_n^*\,\cdot\,b_n)\otimes\operatorname{id})\Delta(a_n^*a_n)\bigr)_{n=1}^{\infty}\to (\omega\otimes\operatorname{id})\Delta(x^*x)$.

Apply $\tilde{\varphi}$.  Since $\tilde{\varphi}$ extends the functional $\varphi$ on ${\mathcal A}$, we have
$$
\tilde{\varphi}\bigl((\theta(b_n^*\,\cdot\,b_n)\otimes\operatorname{id})\Delta(a_n^*a_n)\bigr)
=(\theta\otimes\varphi)\bigl((b_n^*\otimes1)\Delta(a_n^*a_n)(b_n\otimes1)\bigr)
=\theta(b_n^*c_nb_n),
$$
where we wrote $c_n=(\operatorname{id}\otimes\varphi)\bigl(\Delta(a_n^*a_n)\bigr)\in M({\mathcal C})$, by the left invariance property of $\varphi$. 
As $\Delta$ is a ${}^*$-representation, we have $\bigl\|(\pi\otimes\pi)(\Delta(a^*a))\bigr\|\le\bigl\|\pi(a^*a)\bigr\|$, for $a\in{\mathcal A}$.  
So for each $n$, we can see that $\bigl\|\pi(c_n)\bigr\|\le\bigl|\varphi(a_n^*a_n)\bigr|$.  

Since $\tilde{\varphi}$ is an n.s.f. weight, we have
\begin{align}
\tilde{\varphi}\bigl((\omega\otimes\operatorname{id})\Delta(x^*x)\bigr)
&\le\liminf\bigl(\tilde{\varphi}((\theta(b_n^*\,\cdot\,b_n)\otimes\operatorname{id})\Delta(a_n^*a_n))\bigr)_{n=1}^{\infty}  \notag \\
&\le\liminf\bigl(\theta(b_n^*c_nb_n)\bigr)_{n=1}^{\infty}
\le\liminf\bigl|\varphi(a_n^*a_n)\bigr|\theta(b_n^*b_n)=\varphi(x^*x)\|\omega\|.
\notag
\end{align}
This shows that $(\omega\otimes\operatorname{id})\bigl(\Delta(x^*x)\bigr)\in\mathfrak M_{\tilde{\varphi}}^+$.

(2). By standard theory on linear functionals (see, for instance Lemma~4.4 of \cite{BJKVD_qgroupoid1} or Proposition~3.6.7 of 
\cite{Pedersenbook}), there exists a unique positive linear functional $|\omega|\in A^*$ such that $\bigl\||\omega|\bigr\|=\|\omega\|$ 
and $\overline{\omega(a)}\omega(a)=\bigl|\omega(a)\bigr|^2\le\|\omega\||\omega|(a^*a)$, for $a\in A$.  Then we have
$$
(\operatorname{\omega}\otimes\operatorname{id})\bigl(\Delta x\bigr)^*(\operatorname{\omega}\otimes\operatorname{id})\bigl(\Delta x\bigr)
\le\|\omega\|(|\omega|\otimes\operatorname{id})\bigl(\Delta(x^*x)\bigr),
$$
for all $x\in A$.  By using (1), we see that 
$$
\tilde{\varphi}\bigl((\operatorname{\omega}\otimes\operatorname{id})(\Delta x)^*(\operatorname{\omega}\otimes\operatorname{id})(\Delta x)\bigr)
\le\|\omega\|\tilde{\varphi}\bigl((|\omega|\otimes\operatorname{id})\Delta(x^*x)\bigr)
\le\|\omega\|^2\tilde{\varphi}(x^*x).
$$
This shows that $(\omega\otimes\operatorname{id})(\Delta x)\in\mathfrak N_{\tilde{\varphi}}$.
\end{proof}

The results of the above proposition are true for any functional $\omega$.  It suggests that for any $x\in{\mathfrak M}_{\tilde{\varphi}}$, we have 
$\Delta x\in{\mathfrak M}_{\operatorname{id}\otimes\tilde{\varphi}}$.  Moreover, we can actually show the left invariance of $\tilde{\varphi}$ 
at the von Neumann algebra level.  See below:

\begin{prop}\label{leftinvariant_phiext}
Let $x\in{\mathfrak M}_{\tilde{\varphi}}$.  Then $\Delta x\in{\mathfrak M}_{\operatorname{id}\otimes\tilde{\varphi}}$.  We also have:
$$
(\operatorname{id}\otimes\tilde{\varphi})\bigl(\Delta x\bigr)\in\pi({\mathcal C})'',
$$
where $\pi({\mathcal C})''$ is the enveloping von Neumann algebra of $C$, containing $M(C)$.
\end{prop}

\begin{proof}
Let $y\in{\mathfrak N}_{\tilde{\varphi}}$, and let $z\in A$ be arbitrary.  Find sequences $(a_n)$ and $(b_n)$ in ${\mathcal A}$, 
such that $\bigl(\pi(a_n)\bigr)_{n=1}^{\infty}\to y$, $\bigl(\pi(b_n)\bigr)_{n=1}^{\infty}\to z$, 
and $\bigl(\Lambda(a_n)\bigr)_{n=1}^{\infty}\to\Lambda_{\tilde{\varphi}}(y)$. Note that 
$$
\bigl\|(\operatorname{id}\otimes\Lambda_{\tilde{\varphi}})(\Delta(\pi(a_n))(\pi(b_n)\otimes1))\bigr\|^2
=\bigl\|\pi\bigl((\operatorname{id}\otimes\varphi)((b_n^*\otimes1)\Delta(a_n^*a_n)(b_n\otimes1))\bigr)\bigr\|
=\bigl\|\pi(b_n^*c_nb_n)\bigr\|,
$$
where we wrote $c_n=(\operatorname{id}\otimes\varphi)\bigl(\Delta(a_n^*a_n)\bigr)\in M({\mathcal C})$, by the left invariance of $\varphi$. 

As observed in the proof of Proposition~\ref{lemma_invariance}, we know that $\bigl\|\pi(c_n)\bigr\|\le\bigl|\varphi(a_n^*a_n)\bigr|$, for each $n$. 
As $\bigl(\Lambda(a_n)\bigr)_{n=1}^{\infty}\to\Lambda_{\tilde{\varphi}}(y)$, which also means $\bigl(\varphi(a_n^*a_n)\bigr)_{n=1}^{\infty}
\to\tilde{\varphi}(y^*y)$, this implies that the sequence $\bigl(\pi(c_n)\bigr)_{n=1}^{\infty}$ must be Cauchy, hence convergent. Combining 
these results, we can see that $\Delta(y)(z\otimes1)\in{\mathfrak N}_{\operatorname{id}\otimes\tilde{\varphi}}$, or equivalently that 
$(z^*\otimes1)\Delta(y^*y)(z\otimes1)\in{\mathfrak M}_{\operatorname{id}\otimes\tilde{\varphi}}$, such that 
\begin{align}
(\operatorname{id}\otimes\tilde{\varphi})\bigl((z^*\otimes1)\Delta(y^*y)(z\otimes1)\bigr)
&=\lim\bigl(\pi\bigl((\operatorname{id}\otimes\varphi)((b_n^*\otimes1)\Delta(a_n^*a_n)(b_n\otimes1))\bigr)\bigr)_{n=1}^{\infty} \notag \\
&=\lim\pi(b_n^*c_nb_n)_{n=1}^{\infty}=z^*cz,
\notag
\end{align}
where $c=\lim_{n=1}^{\infty}\pi(c_n)$.  Since each $c_n\in M({\mathcal C})$, it is evident $c=\lim_{n=1}^{\infty}\pi(c_n)\in\pi({\mathcal C})''$.  

This is true for any $z\in A$.  In other words, for $y\in{\mathfrak N}_{\tilde{\varphi}}$, we have 
$\Delta(y^*y)\in{\mathfrak M}_{\operatorname{id}\otimes\tilde{\varphi}}$, such that
$$
(\operatorname{id}\otimes\tilde{\varphi})\bigl(\Delta(y^*y)\bigr)=\lim_{n=1}^{\infty}\pi(c_n)\in\pi({\mathcal C})''.
$$
By polarization, it follows that 
for any $x\in{\mathfrak M}_{\tilde{\varphi}}$, we have $\Delta x\in{\mathfrak M}_{\operatorname{id}\otimes\tilde{\varphi}}$, and that 
$$
(\operatorname{id}\otimes\tilde{\varphi})(\Delta x)\in\pi({\mathcal C})''.
$$
\end{proof}

\subsection{The weight $\tilde{\psi}$}

Again at the von Neumann algebra level, let us define another weight $\tilde{\psi}$, as follows:

\begin{defn}\label{weightpsi}
Define the weight $\tilde{\psi}=\tilde{\varphi}\circ\widetilde{R}$.  Then $\tilde{\psi}$ is an n.s.f. weight on $M=\pi({\mathcal A})''$.
\end{defn}

Since $\widetilde{R}:M\to M$ is an involutive ${}^*$-anti-isomorphism, it is not too difficult to see that it is an n.s.f.~weight, 
with its associated modular automorphism group $\tilde{\sigma}':=\sigma^{\tilde{\psi}}$ satisfying $\tilde{\sigma}'_t
=\widetilde{R}\circ\tilde{\sigma}_{-t}\circ\widetilde{R}$, for all $t$.

Here is a result that is analogous to Proposition~\ref{Deltasigma_t}, this time in terms of the modular automorphism group 
$(\tilde{\sigma}'_t)$. 

\begin{prop}\label{Deltasigma'_t}
For all $x\in A$ and $t\in\mathbb{R}$, we have: 
$\Delta\bigl(\tilde{\sigma}'_t(x)\bigr)=(\tilde{\sigma}'_t\otimes\tau_{-t})(\Delta x)$.
\end{prop}

\begin{proof}
We can use the fact $\tilde{\sigma}'_t=\widetilde{R}\circ\tilde{\sigma}_{-t}\circ\widetilde{R}$, together with the result of 
Proposition~\ref{DeltaR}, saying $(\widetilde{R}\otimes\widetilde{R})(\Delta x)
=\Delta^{\operatorname{cop}}\bigl(\widetilde{R}(x)\bigr)$.  See below:
\begin{align}
\Delta\bigl(\tilde{\sigma}'_t(x)\bigr)&=\Delta\bigl((\widetilde{R}\circ\tilde{\sigma}_{-t}\circ\widetilde{R})(x)\bigr)
=\varsigma(\widetilde{R}\otimes\widetilde{R})\bigl(\Delta(\tilde{\sigma}_{-t}(\widetilde{R}(x)))\bigr) 
=\varsigma(\widetilde{R}\otimes\widetilde{R})\bigl((\tau_{-t}\otimes\tilde{\sigma}_{-t})(\Delta(\widetilde{R}(x)))\bigr)  \notag \\
&=(\widetilde{R}\otimes\widetilde{R})\bigl((\tilde{\sigma}_{-t}\otimes\tau_{-t})(\Delta^{\operatorname{cop}}(\widetilde{R}(x)))\bigr)
=\bigl((\widetilde{R}\otimes\widetilde{R})\circ(\tilde{\sigma}_{-t}\otimes\tau_{-t})\circ(\widetilde{R}\otimes\widetilde{R})\bigr)(\Delta x) 
\notag \\
&=(\tilde{\sigma}'_t\otimes\tau_{-t})(\Delta x),
\notag
\end{align}
because $\tau_t$ commutes with $\widetilde{R}$.
\end{proof}

As a consequence of Propositions~\ref{Deltasigma_t} and \ref{Deltasigma'_t}, we obtain the following result, which is essentially a converse result 
of Proposition~\ref{DeltaonBandC}, providing us with useful characterizations of $M(B)$ and $M(C)$.

\begin{prop}\label{DeltaonBandCconv}
Let $x,y\in M(A)$.  We have:
\begin{enumerate}
  \item $x\in M(C)$ if and only if $\Delta x=(x\otimes1)E=E(x\otimes1)$,
  \item $y\in M(B)$ if and only if $\Delta y=E(1\otimes y)=(1\otimes y)E$.
\end{enumerate}
\end{prop}

\begin{proof}
We can use mostly the same proof as in Proposition~5.22 in \cite{BJKVD_qgroupoid2}, with only some minor modifications.  Even though 
our weights $\tilde{\varphi}$ and $\tilde{\psi}$ are at the von Neumann algebra level at present, the results still stand because the proof characterizes 
$x\in M(C)$ as a double centralizer for $C$, and similarly for $y\in M(B)$.  The first result uses Proposition~\ref{Deltasigma_t}, and the second 
one uses Proposition~\ref{Deltasigma'_t}.
\end{proof}

Here are a couple more useful results, regarding the base algebras $M(B)$ and $M(C)$:

\begin{prop}\label{tauonBandC}
The scaling group $(\tau_t)$ leaves both $M(B)$ and $M(C)$ invariant.
\end{prop}

\begin{proof}
See Proposition~5.23\,(1) of \cite{BJKVD_qgroupoid2} for the proof. Since it uses only the characterizations of $M(B)$ and $M(C)$ 
noted in Proposition~\ref{DeltaonBandCconv}, as well as the results of Propositions~\ref{Deltatau_t} and \ref{tausigmaR_corollaryE}, this is all right.
\end{proof}

\begin{prop}\label{RonBandC}
We have: $R\bigl(M(B)\bigr)=M(C)$ and $R\bigl(M(C)\bigr)=M(B)$. 

At the von Neumann algebra level, we have: $\widetilde{R}\bigl(\pi({\mathcal B})''\bigr)=\pi({\mathcal C})''$ 
and $\widetilde{R}\bigl(\pi({\mathcal C})''\bigr)=\pi({\mathcal B})''$.
\end{prop}

\begin{proof}
See Proposition~5.23\,(2) of \cite{BJKVD_qgroupoid2} for the proof of the first result. Since it uses only the characterizations of $M(B)$ and $M(C)$ 
noted in Proposition~\ref{DeltaonBandCconv}, as well as the results of Propositions~\ref{DeltaR} and \ref{tausigmaR_corollaryE}, this is all right.

The second result is a direct consequence of the first result.
\end{proof}

\begin{rem}
In section~5 of \cite{BJKVD_qgroupoid2}, there are more results obtained regarding $S$, $R$, $\tau$.  In particular, not only 
$R\bigl(M(B)\bigr)=M(C)$ and $R\bigl(M(C)\bigr)=M(B)$, but in fact, it turns out that $R|_B=R_{BC}:B\to C$, the 
${}^*$-anti-isomorphism that appear in Proposition~\ref{EatC*level}.  Not all the results in \cite{BJKVD_qgroupoid2} 
can be used (or only allowed with an alternative proof), because the full construction is not yet done.  However, eventually 
all these results would be obtained as consequences, once we successfully carry out the construction of our $C^*$-algebraic 
quantum groupoid.
\end{rem}

With these results regarding the isomorphism $\widetilde{R}$ gathered, we can now show that the weight $\tilde{\psi}$ satisfies the right invariance 
property at the von Neumann algebra level:

\begin{prop}\label{rightinvariant_psiext}
Let $x\in{\mathfrak M}_{\tilde{\psi}}$.  Then $\Delta x\in{\mathfrak M}_{\operatorname{id}\otimes\tilde{\psi}}$.  We also have:
$$
(\tilde{\psi}\otimes\operatorname{id})\bigl(\Delta x\bigr)\in\pi({\mathcal B})'',
$$
where $\pi({\mathcal B})''$ is the enveloping von Neumann algebra of $B$, containing $M(B)$.
\end{prop}

\begin{proof}
Let $x\in{\mathfrak M}_{\tilde{\psi}}$. Since $\tilde{\psi}=\tilde{\varphi}\circ\widetilde{R}$, this immediately means that 
$\widetilde{R}(x)\in{\mathfrak M}_{\tilde{\varphi}}$, which in turn means 
$\Delta\bigl(\widetilde{R}(x)\bigr)\in{\mathfrak M}_{\operatorname{id}\otimes\tilde{\varphi}}$, by Proposition~\ref{leftinvariant_phiext}. 
It follows that
\begin{align}
(\tilde{\psi}\otimes\operatorname{id})(\Delta x)
&=\bigl((\tilde{\varphi}\circ\widetilde{R})\otimes\operatorname{id}\bigr)(\Delta x)  \notag \\
&=(\tilde{\varphi}\otimes\widetilde{R})\bigl((\widetilde{R}\otimes\widetilde{R})(\Delta x)\bigr)
=(\tilde{\varphi}\otimes\widetilde{R})\bigl(\Delta^{\operatorname{cop}}(\widetilde{R}(x))\bigr)
=\widetilde{R}\bigl((\operatorname{id}\otimes\tilde{\varphi})(\Delta(\widetilde{R}(x)))\bigr),
\notag
\end{align}
which is a valid element in $\widetilde{R}\bigl(\pi({\mathcal C})''\bigr)$, by Proposition~\ref{leftinvariant_phiext}. 
Since $\widetilde{R}\bigl(\pi({\mathcal C})''\bigr)=\pi({\mathcal B})''$, we prove that 
$(\tilde{\psi}\otimes\operatorname{id})\bigl(\Delta x\bigr)\in\pi({\mathcal B})''$.
\end{proof}

With this result, we see that the weight $\tilde{\psi}$ is a right invariant weight.  One word of caution, however: Unlike $\tilde{\varphi}$, which 
extends the functional $\varphi$ on ${\mathcal A}$, it may not necessarily be the case that $\tilde{\psi}=\tilde{\varphi}\circ\widetilde{R}$ is 
an extension of the functional $\psi$ we began with.  There is no complete uniqueness property for the right invariant functionals and weights, 
so while they are related they may not be exactly same.  

Meanwhile, for the operator algebraic theory to work properly, we not only need the weights $\tilde{\varphi}$ and $\tilde{\psi}$, 
but we actually need a certain condition such that we are able to have a Radon--Nikodym derivative between the two 
weights.  This means that we need an additional requirement, already at the ${}^*$-algebra level.  See the discussion in 
the next subsection.

\subsection{Quasi-invariance assumption}\label{sub5.3}

So far, at least at the von Neumann algebra level, we have established that the weights $\tilde{\varphi}$ and $\tilde{\psi}$ 
are left and right invariant, respectively.  

Unlike the purely algebraic theory, however, for things to work properly in the operator algebraic setting, we need the two 
weights to allow a Radon--Nikodym derivative between them.  Even in the classical setting of locally compact groupoids, 
certain {\em quasi-invariance\/} condition was assumed as a part of the definition (see \cite{Renbook}, \cite{Patbook}).  

Since we are constructing an operator algebraic object from purely algebraic data, some form of an additional condition 
needs to be required at the algebra level to allow this development.  With this in mind, we introduce the following 
{\em quasi-invariance assumption\/} at the purely algebraic level.  This is the one small (but necessary) additional condition 
we are going to require:

\begin{assumption} [Quasi-invariance]
We will assume that $\sigma|_{\mathcal B}$, the restriction of $\sigma$ to the base algebra ${\mathcal B}$, leaves 
${\mathcal B}$ invariant, and that $\nu\circ\sigma|_{\mathcal B}=\nu$.
\end{assumption}

\begin{rem}
This is a purely algebraic condition, where $\nu$ is the distinguished functional on ${\mathcal B}$, and $\sigma$ is the automorphism on 
${\mathcal A}$, as given in Proposition~\ref{phi_modularautomorphism}. 
It can be shown that the restriction $\sigma|_{\mathcal C}$ leaves ${\mathcal C}$ invariant with $\mu\circ\sigma_{\mathcal C}=\mu$, 
because it turns out that $\sigma|_{\mathcal C}=\sigma^{\mu}$.  Similarly, it turns out that $\sigma^{\varphi\circ S}|_{\mathcal B}=\sigma^{\nu}$, 
thus it leaves ${\mathcal B}$ invariant with $\nu\circ\sigma^{\varphi\circ S}|_{\mathcal B}=\nu$.  See Proposition~\ref{sigmarestriction} 
in Appendix (Section~\S\ref{appx}).  On the other hand, we do not know whether similar properties hold for the restrictions $\sigma|_{\mathcal B}$ 
and $\sigma^{\varphi\circ S}|_{\mathcal C}$.  There is no reason why they should.  As such, our quasi-invariance requirement is an extra assumption.
\end{rem}

\begin{rem}
It is possible to actually prove the invariance of ${\mathcal B}$ under $\sigma|_{\mathcal B}$, but not the 
$\nu\circ\sigma|_{\mathcal B}=\nu$ result.  Symmetrically, if the assumption holds, then we can quickly show that 
${\mathcal C}$ is invariant under $\sigma^{\varphi\circ S}|_{\mathcal C}$ and that 
$\mu\circ\sigma^{\varphi\circ S}|_{\mathcal C}=\mu$.  We termed this assumption as ``quasi-invariance'', because 
this assumption indeed gives rise to the quasi-invariance property at the $C^*$-algebraic framework.  See the remarks 
given later in the subsection.  
\end{rem}

In this subsection, we will collect some consequences of the quasi-invariance assumption.  Many of these results refer to 
$\varphi$ and $\sigma$, as well as $\varphi\circ S$ and $\sigma^{\varphi\circ S}$.  See \S\ref{suba.1}, \S\ref{suba.2} of 
Appendix (Section~\S\ref{appx}) for discussion on some algebraic results on these objects.

\begin{lem}\label{lemmaqi}
Under the above quasi-invariance assumption, we have:
\begin{enumerate}
\item  $(\varphi\circ S)\bigl((\sigma\circ S^{-2})(a)\bigr)=(\varphi\circ S)(a)$, for all $a\in{\mathcal A}$.
\item $(\sigma^{\varphi\circ S}\circ\sigma\circ S^{-2})(a)=(\sigma\circ S^{-2}\circ\sigma^{\varphi\circ S})(a)$, for all $a\in{\mathcal A}$.
\end{enumerate}
\end{lem}

\begin{proof}
(1). Let $a\in{\mathcal A}$.  Since $\varphi\circ S$ is a right integral, by Proposition~\ref{phiSresults}\,(1) 
(or Proposition~\ref{muphinupsi}), we have: 
$$
(\varphi\circ S)\bigl((\sigma\circ S^{-2})(a)\bigr)=\nu\bigl(((\varphi\circ S)\otimes\operatorname{id})(\Delta((\sigma\circ S^{-2})(a)))\bigr)
=\nu\bigl((\sigma\circ S^{-2})(((\varphi\circ S)\otimes\operatorname{id})(\Delta a))\bigr).
$$
Note here that we also used Corollary of Proposition~\ref{Deltasigma}.

In the right side of the above equation, since we know $\bigl((\varphi\circ S)\otimes\operatorname{id})(\Delta a)\in M({\mathcal B})$ 
by the right invariance of $\varphi\circ S$, we may regard the maps $\sigma$ and $S^{-2}$ above as $\sigma|_{\mathcal B}$ and 
$S^{-2}|_{\mathcal B}$, respectively.   As $\nu\circ\sigma|_{\mathcal B}=\nu$ (by the quasi-invariance assumption) and 
$\nu\circ S^{-2}|_{\mathcal B}=\nu\circ\sigma^{\nu}=\nu$ (by Proposition~\ref{sigmarestriction}), the equation becomes:
$$
(\varphi\circ S)\bigl((\sigma\circ S^{-2})(a)\bigr)=\nu\bigl(((\varphi\circ S)\otimes\operatorname{id})(\Delta a)\bigr)=(\varphi\circ S)(a).
$$

(2). Let $a,x\in{\mathcal A}$ be arbitrary. We have:
\begin{align}
(\varphi\circ S)(ax)&=(\varphi\circ S)\bigl(x\sigma^{\varphi\circ S}(a)\bigr)
=(\varphi\circ S)\bigl((\sigma\circ S^{-2})(x\sigma^{\varphi\circ S}(a))\bigr)
\notag \\
&=(\varphi\circ S)\bigl((\sigma\circ S^{-2})(x)(\sigma\circ S^{-2}\circ\sigma^{\varphi\circ S})(a)\bigr),
\label{(5a)}
\end{align}
by (1) above and the fact that $\sigma$ and $S^{-2}$ are multiplicative (being isomorphisms on ${\mathcal A}$).

On the other hand, we also have: 
\begin{align}
(\varphi\circ S)(ax)&=(\varphi\circ S)\bigl((\sigma\circ S^{-2})(ax)\bigr)=(\varphi\circ S)\bigl((\sigma\circ S^{-2})(a)(\sigma\circ S^{-2})(x)\bigr)
\notag \\
&=(\varphi\circ S)\bigl((\sigma\circ S^{-2})(x)(\sigma^{\varphi\circ S}\circ\sigma\circ S^{-2})(a)\bigr).
\label{(5b)}
\end{align}

Comparing Equations~\eqref{(5a)} and \eqref{(5b)}, since $x\in{\mathcal A}$ is arbitrary and since $\varphi\circ S$ is faithful, 
it follows that $(\sigma\circ S^{-2}\circ\sigma^{\varphi\circ S})(a)=(\sigma^{\varphi\circ S}\circ\sigma\circ S^{-2})(a)$, for all $a\in{\mathcal A}$.
\end{proof}

This lemma helps us prove the following commutativity results:

\begin{prop} \label{sigmasigma'commute}
Under the quasi-invariance assumption above, we have:
\begin{enumerate}
\item $\sigma\circ S^2=S^2\circ\sigma$
\item $\sigma^{\varphi\circ S}\circ S^2=S^2\circ\sigma^{\varphi\circ S}$
\item $\sigma\circ\sigma^{\varphi\circ S}=\sigma^{\varphi\circ S}\circ\sigma$
\end{enumerate}
\end{prop}

\begin{proof}
(1). Let $a\in{\mathcal A}$ be arbitrary. By Corollary following Proposition~\ref{Deltasigma} (in Appendix), and by using the fact that 
$\sigma\circ S^{-2}=[\sigma^{\varphi\circ S}]^{-1}\circ\sigma\circ S^{-2}\circ\sigma^{\varphi\circ s}$ (from Lemma~\ref{lemmaqi}), 
we have: 
\begin{equation}\label{(5c)}
\bigl(\operatorname{id}\otimes(\sigma\circ S^{-2})\bigr)(\Delta a)=\Delta\bigl(\sigma(S^{-2}(a))\bigr)
=\Delta\bigl(([\sigma^{\varphi\circ S}]^{-1}\circ\sigma\circ S^{-2}\circ\sigma^{\varphi\circ S})(a)\bigr).
\end{equation}

Meanwhile, by Proposition~\ref{Deltasigma}\,(2), where we have $\Delta\bigl(\sigma^{\varphi\circ S}(x)\bigr)=(\sigma^{\varphi\circ S}\otimes S^{-2})
(\Delta x)$, for all $x\in{\mathcal A}$, we know that $\bigl([\sigma^{\varphi\circ S}]^{-1}\otimes S^2\bigr)\bigl(\Delta(\sigma^{\varphi\circ S}(x))\bigr)
=\Delta x$. Letting $x=[\sigma^{\varphi\circ S}]^{-1}(a)$, for $a\in{\mathcal A}$, we have:
\begin{equation}\label{(5d)}
\Delta\bigl([\sigma^{\varphi\circ S}]^{-1}(a)\bigr)=\bigl([\sigma^{\varphi\circ S}]^{-1}\otimes S^2\bigr)(\Delta a).
\end{equation}

Combining Equation~\eqref{(5c)} and \eqref{(5d)}, and using again Corollary of Proposition~\ref{Deltasigma}, we have:
\begin{align}
\bigl(\operatorname{id}\otimes(\sigma\circ S^{-2})\bigr)(\Delta a)&=\Delta\bigl(\sigma(S^{-2}(a))\bigr)
=\bigl([\sigma^{\varphi\circ S}]^{-1}\otimes S^2\bigr)\bigl(\Delta((\sigma\circ S^{-2})(\sigma^{\varphi\circ S}(a)))\bigr)  \notag \\
&=\bigl([\sigma^{\varphi\circ S}]^{-1}\otimes(S^2\circ\sigma\circ S^{-2})\bigr)\bigl(\Delta(\sigma^{\varphi\circ s}(a))\bigr).
\notag
\end{align}
Note that $\Delta\bigl(\sigma^{\varphi\circ s}(a)\bigr)=(\sigma^{\varphi\circ S}\otimes S^{-2})(\Delta a)$, by Proposition~\ref{Deltasigma}. 
As a result, we thus obtain the following:
$$
\bigl(\operatorname{id}\otimes(\sigma\circ S^{-2})\bigr)(\Delta a)
=\bigl(\operatorname{id}\otimes(S^2\circ\sigma\circ S^{-2}\circ S^{-2})\bigr)(\Delta a).
$$

Since $\Delta$ is full, it follows that $\sigma\circ S^{-2}=S^2\circ\sigma\circ S^{-2}\circ S^{-2}$.
Since $S^2$ and $S^{-2}$ are automorphisms on ${\mathcal A}$, we thus have: $S^2\circ\sigma=\sigma\circ S^2$.

(2). Note (see proof of Proposition~\ref{automorphismsigmaS}) that $\sigma^{\varphi\circ S}=S^{-1}\circ\sigma^{-1}\circ S$.  By (1), 
we know $S^2\circ\sigma=\sigma\circ S^2$, or equivalently $\sigma^{-1}\circ S^2=S^2\circ\sigma^{-1}$.  So we have:
$$
\sigma^{\varphi\circ S}\circ S^2=S^{-1}\circ\sigma^{-1}\circ S^3=S\circ\sigma^{-1}\circ S=S^2\circ S^{-1}\circ\sigma^{-1}\circ S
=S^2\circ\sigma^{\varphi\circ S}.
$$

(3). Let $a\in{\mathcal A}$ be arbitrary.  We have:
\begin{align}
\Delta\bigl(\sigma(\sigma^{\varphi\circ S}(a))\bigr)&=(S^2\otimes\sigma)\bigl(\Delta(\sigma^{\varphi\circ S}(a))\bigr)  \notag \\
&=\bigl((S^2\circ\sigma^{\varphi\circ S})\otimes(\sigma\circ S^{-2})\bigr)(\Delta a)
=\bigl((\sigma^{\varphi\circ S}\circ S^2)\otimes(S^{-2}\circ\sigma)\bigr)(\Delta a)  \notag \\
&=(\sigma^{\varphi\circ S}\otimes S^{-2})\bigl(\Delta(\sigma(a))\bigr)=\Delta\bigl(\sigma^{\varphi\circ S}(\sigma(a))\bigr),
\notag
\end{align}
by Proposition~\ref{Deltasigma} and by (1), (2) above (the commutativity of $\sigma$ and $S^2$, and of $\sigma^{\varphi\circ S}$ and $S^2$).  
As $\Delta$ is injective, this shows that $\sigma\circ\sigma^{\varphi\circ S}=\sigma^{\varphi\circ S}\circ\sigma$.
\end{proof}

So far, the results we obtained in this subsection have been all at the ${}^*$-algebra level, under the {\em quasi-invariance 
assumption\/} (formulated at the ${}^*$-algebra level), and using only the algebraic properties summarized in Appendix 
(\S\ref{appx}).  However, the results of Proposition~\ref{sigmasigma'commute} now allows us to obtain some results at 
the operator algebra level:

\begin{prop}\label{weightssigma'sigmacommute}
Under the quasi-invariance assumption, we have the following commutativity results involving the automorphism groups 
$(\tau_t)$, $(\tilde{\sigma}_t)$, and $(\tilde{\sigma}'_t)$.
\begin{enumerate}
\item The automorphism groups $(\tau_t)$ and $(\tilde{\sigma}_t)$ commute.
\item The automorphism groups $(\tilde{\sigma}_t)$ for the weight $\tilde{\varphi}$ and $(\tilde{\sigma}'_t)$ 
for the weight $\tilde{\psi}$ commute.
\end{enumerate}
\end{prop}

\begin{proof}
(1). From Proposition~\ref{sigmasigma'commute}\,(1), we saw that $\sigma\circ S^2=S^2\circ\sigma$, on ${\mathcal A}$. 
Since $\tilde{\sigma}_{-i}\bigl(\pi(a)\bigr)=\sigma(a)$ and $\tau_{-i}\bigl(\pi(a)\bigr)=S^2(a)$, while ${\mathcal A}$ forms 
a core for both, this means that $\tilde{\sigma}_{-i}\circ\tau_{-i}=\tau_{-i}\circ\tilde{\sigma}_{-i}$.

Note that $\tilde{\sigma}_{-i}$ is an analytic generator for the automorphism group $(\tilde{\sigma}_t)$ at $-i$, while 
$\tau_{-i}$ is an analytic generator for $(\tau_t)$. Since the analytic generators of the two automorphism groups commute, 
we can show without difficulty that the automorphism groups $(\tau_t)$ and $(\tilde{\sigma}_t)$ commute.  

For instance, we have $\tilde{\sigma}_{-i}\tau_{-2i}
=\tilde{\sigma}_{-i}\tau_{-i}\tau_{-i}=\tau_{-i}\tilde{\sigma}_{-i}\tau_{-i}=\tau_{-2i}\tilde{\sigma}_{-i}$.  Also, we have 
$(\tilde{\sigma}_{-i}\tau_{-\frac{i}2}\tilde{\sigma}_{i})\circ(\tilde{\sigma}_{-i}\tau_{\frac{i}2}\tilde{\sigma}_{i})
=\tilde{\sigma}_{-i}\tau_{-i}\tilde{\sigma}_{i}=\tau_{-i}$, so $\tilde{\sigma}_{-i}\tau_{\frac{i}2}\tilde{\sigma}_{i}=(\tau_{-i})^{\frac12}=\tau_{-\frac{i}2}$. 
That also means $\tilde{\sigma}_{-i}\tau_{-\frac{i}2}=\tau_{-\frac{i}2}\tilde{\sigma}_{-i}$.  Continuing, we can show that $\tilde{\sigma}_{-i}\tau_s
=\tau_s\tilde{\sigma}_{-i}$, for all $s$.  We can do the same for $\tilde{\sigma}_t$, so we would have 
\begin{equation}\label{(tausigmacommute)}
\tau_s\circ\tilde{\sigma}_t=\tilde{\sigma}_t\circ\tau_s, \ {\text { for all $s,t\in\mathbb{C}$.}}
\end{equation}

(2). In Proposition~\ref{sigmasigma'commute}\,(3), we saw that $\sigma\circ\sigma^{\varphi\circ S}=\sigma^{\varphi\circ S}\circ\sigma$, 
on ${\mathcal A}$.  This means that on ${\mathcal A}$, we have: $\sigma\circ(S^{-1}\circ\sigma^{-1}\circ S)=(S^{-1}\circ\sigma^{-1}\circ S)\circ\sigma$.

Going up to the operator algebra level, recall that $\tilde{\psi}=\varphi\circ\widetilde{R}$, with $\tilde{\sigma}'_t
=\widetilde{R}\circ\tilde{\sigma}_{-t}\circ\widetilde{R}$. In particular, at $t=-i$, we have 
$$
\tilde{\sigma}'_{-i}=\widetilde{R}\circ\tilde{\sigma}_i\circ\widetilde{R}
=\widetilde{R}\circ\tau_{\frac{i}2}\circ\tilde{\sigma}_i\circ\tau_{-\frac{i}2}\circ\widetilde{R}=S^{-1}\circ\tilde{\sigma}_i\circ S,
$$
because $\tau$ commutes with $\widetilde{R}$ (Proposition~\ref{Rtaucoomute}) and also with $\tilde{\sigma}$ (see (1) above). We also used 
the polar decomposition of the antipode.  

Note that $\tilde{\sigma}_i$ coincides with $\sigma^{-1}$ on ${\mathcal A}$.  So on ${\mathcal A}$, the map 
$S^{-1}\circ\tilde{\sigma}_i\circ S$ coincides with $S^{-1}\circ\sigma^{-1}\circ S$, which has been observed to be commuting 
with $\sigma$. What all this means is that $\tilde{\sigma}'_{-i}$ commutes with $\tilde{\sigma}_{-i}$. 

As above, since $\tilde{\sigma}'_{-i}$ is an analytic generator for $(\tilde{\sigma}'_t)$ and since $\tilde{\sigma}_{-i}$ is an 
analytic generator for $(\tilde{\sigma}_t)$, the commutativity of the analytic generators mean the commutativity of the 
automorphism groups $(\tilde{\sigma}'_t)$ and $(\tilde{\sigma}_t)$.  In particular, we have $\tilde{\sigma}'_s\circ\tilde{\sigma}_t
=\tilde{\sigma}_t\circ\tilde{\sigma}'_s$, for all $s,t$.
\end{proof}

The commutativity of the modular automorphism groups of the weights $\tilde{\varphi}$ and $\tilde{\psi}$ is significant, because we have 
the following general result by Vaes:

\begin{prop} \label{vaesRN}
(Vaes \cite{VaRN}): Let $\varphi$ and $\psi$ be two n.s.f. weights on a von Neumann algebra $M$. Then the following 
are equivalent:
\begin{itemize}
\item[(i).] The modular automorphism groups $\sigma^{\psi}$ and $\sigma^{\varphi}$ commute.
\item[(ii).] There exist a strictly positive operator $\tilde{\delta}$ affiliated with $M$ and a strictly positive operator $\lambda$ 
affiliated with the center of $M$ such that $\sigma^{\varphi}_s(\tilde{\delta}^{it})=\lambda^{ist}\tilde{\delta}^{it}$ for all $s,t\in\mathbb{R}$ 
and such that $\psi=\varphi_\delta=\varphi(\tilde{\delta}^{\frac12}\,\cdot\,\tilde{\delta}^{\frac12})$.
\item[(iii).] There exist a strictly positive operator $\delta$ affiliated with $M$ and a strictly positive operator $\lambda$ 
affiliated with the center of $M$ such that $[D\psi:D\varphi]_t=\lambda^{\frac12it^2}\tilde{\delta}^{it}$ for all $t\in\mathbb{R}$.
\end{itemize}
\end{prop}

\begin{proof} 
See Proposition~5.2 in \cite{VaRN}.
\end{proof}

In our case, we should consider the weights $\tilde{\varphi}$, $\tilde{\psi}$, and the commuting modular automorphism groups 
$(\tilde{\sigma}_t)$, $(\tilde{\sigma}'_t)$. The commutativity is a consequence of our (purely algebraic) quasi-invariance 
assumption.  As Vaes's result (Proposition~\ref{vaesRN} above) indicates, this assures us the existence of a suitable 
Radon--Nikodym derivative between our left and right invariant weights, justifying the term ``quasi-invariance''.  Recall that 
in the theory of (classical) locally compact groupoids (see \cite{Renbook}, \cite{Patbook}), the existence of a Radon--Nikodym 
derivative between the left integral and the right integral is typically assumed as part of the definition, which is referred to as t
he {\em quasi-invariance\/} condition.

Note also the existence of a positive operator $\tilde{\delta}$ affiliated with $M$, and also an existence of another positive 
operator $\lambda$.  We re-named Vaes's $\delta$ as $\tilde{\delta}$ here, to avoid confusion with our modular element 
$\delta$ at the ${}^*$-algebra level. But then, it seems reasonable to expect that the operator $\tilde{\delta}$ is closely related 
with the modular element $\delta$. See discussion below. As for $\lambda$, it seems to be a generalization of the ``scaling 
constant'' in the quantum group theory, but now no longer a scalar.  In the rest of the subsection, we will attempt to make 
a connection between the operator $\tilde{\delta}$ and the modular element $\delta\in{\mathcal A}$.  

At the ${}^*$-algebra level, we have been working with the functionals $\varphi$ and $\varphi\circ S$, which are faithful left 
and right integrals, respectively.  On the other hand, at the operator algebra level, we have been considering the weights 
$\tilde{\varphi}$ (which extends $\varphi$) and $\tilde{\psi}=\tilde{\varphi}\circ\widetilde{R}$.  As such, any discussion 
regarding the modular element $\delta$ and the modular operator $\tilde{\delta}$ should begin with clarifying the relationship 
between the functional $\varphi\circ S$ and the weight $\tilde{\psi}$.  For this, we will make use of the polar decomposition 
of the antipode: $S=R\circ\tau_{-\frac{i}2}$.

Let us consider the weight $\tilde{\varphi}\circ\tau_{\frac{i}2}$.  It is not difficult to show that it satisfies the left invariance property:
\begin{prop}\label{phitauleftintegral}
The weight $\tilde{\varphi}\circ\tau_{\frac{i}2}$ is left invariant:
$$
\bigl(\operatorname{id}\otimes(\tilde{\varphi}\circ\tau_{\frac{i}2})\bigr)(\Delta x)\in M({\mathcal C})'', 
\quad {\text { for all $x\in{\mathfrak M}_{\tilde{\varphi}}$.}}
$$
\end{prop}

\begin{proof}
Let $x\in{\mathfrak M}_{\tilde{\varphi}}$. By Proposition~\ref{Deltatau_t}, we have
\begin{align}
\bigl(\operatorname{id}\otimes(\tilde{\varphi}\circ\tau_{\frac{i}2})\bigr)(\Delta x)
&=(\tau_{-\frac{i}2}\otimes\tilde{\varphi})\bigl((\tau_{\frac{i}2}\otimes\tau_{\frac{i}2})(\Delta x)\bigr)
=(\tau_{-\frac{i}2}\otimes\tilde{\varphi})\bigl(\Delta(\tau_{\frac{i}2}(x))\bigr)  \notag \\
&=\tau_{-\frac{i}2}\bigl((\operatorname{id}\otimes\tilde{\varphi})(\Delta(\tau_{\frac{i}2}(x)))\bigr)
\in\tau_{-\frac{i}2}\bigl(M({\mathcal C})''\bigr)=M({\mathcal C})''.
\notag
\end{align}
We used also the left invariance property of $\tilde{\varphi}$ (Proposition~\ref{leftinvariant_phiext}) and also Proposition~\ref{tauonBandC}.
\end{proof}

While $\tilde{\varphi}\circ\tau_{\frac{i}2}$ is a weight at the von Neumann algebra level, it can be considered as a functional 
on ${\mathcal A}$, by $(\tilde{\varphi}\circ\tau_{\frac{i}2})|_{\mathcal A}(a):=\tilde{\varphi}\bigl(\tau_{\frac{i}2}(\pi(a))\bigr)$. 
Note that all the elements $\pi(a)$, $a\in{\mathcal A}$, are analytic elements of $\tau$ (Proposition~\ref{analyticfortau}), 
and the commutativity between $\tau$ and $\tilde{\sigma}$ as noted in Equation~\eqref{(tausigmacommute)} implies that 
all $\tau_{\frac{i}2}\bigl(\pi(a)\bigr)$, $a\in{\mathcal A}$, are analytic for $\tilde{\sigma}$.  As such, we see that 
$\tau_{\frac{i}2}\bigl(\pi(a)\bigr)\in{\mathfrak M}_{\tilde{\varphi}}$, making $(\tilde{\varphi}\circ\tau_{\frac{i}2})|_{\mathcal A}$ 
a valid functional defined on all of ${\mathcal A}$.  This functional turns out to be a left integral.  See below:

\begin{prop}
The functional $(\tilde{\varphi}\circ\tau_{\frac{i}2})|_{\mathcal A}$ is a left integral on ${\mathcal A}$.
\end{prop}

\begin{proof}
Let $a\in{\mathcal A}$. Note that $x=\bigl(\operatorname{id}\otimes(\tilde{\varphi}\circ\tau_{\frac{i}2})|_{\mathcal A}\bigr)
(\Delta a)\in M({\mathcal A})$, and by Proposition~\ref{phitauleftintegral} we have $\pi(x)\in M({\mathcal C})''$. 
Then it must be the case that $\Delta x=E(x\otimes1)$. 
Therefore, by a result at the algebra level (see Proposition~2.9 in \cite{BJKVD_LSthm} and Proposition~2.16 in \cite{VDWangwha2}), this means $x\in M({\mathcal C})$. Or 
$\bigl(\operatorname{id}\otimes(\tilde{\varphi}\circ\tau_{\frac{i}2})|_{\mathcal A}\bigr)(\Delta a)\in M({\mathcal C})$, 
for all $a\in{\mathcal A}$.
\end{proof}

\begin{cor}
There exists an element $y\in M({\mathcal B})$ such that $(\tilde{\varphi}\circ\tau_{\frac{i}2})|_{\mathcal A}(x)=\varphi(xy)$, 
for all $x\in{\mathcal A}$.
\end{cor}

\begin{proof}
This is a consequence of $(\tilde{\varphi}\circ\tau_{\frac{i}2})|_{\mathcal A}$ being a left integral on ${\mathcal A}$. 
See Proposition~\ref{phiandotherleft}.
\end{proof}

\begin{cor}
There exists an element $z\in M(B)$ such that $(\tilde{\varphi}\circ\tau_{\frac{i}2})(\,\cdot\,)=\tilde\varphi(\,\cdot\,z)$.
\end{cor}

\begin{proof}
Take $z=\pi(y)$, where $y\in M({\mathcal B})$ is as in the previous Corollary.
\end{proof}

Let us now consider the weight $\tilde{\psi}=\tilde{\varphi}\circ\widetilde{R}$.  First, from Proposition~\ref{vaesRN}, we saw that we can write 
$$
\tilde{\psi}=\tilde{\varphi}(\tilde{\delta}^{\frac12}\,\cdot\,\tilde{\delta}^{\frac12}),
$$ 
where $\tilde{\delta}$ is the modular operator.  By the polar decomposition $S=\widetilde{R}\circ\tau_{-\frac{i}2}$, we know 
$\widetilde{R}=\tau_{\frac{i}2}\circ S$.  So we can write the above as
$$
(\tilde{\varphi}\circ\tau_{\frac{i}2})\bigl(S(x)\bigr)=\tilde{\varphi}(\tilde{\delta}^{\frac12}x\tilde{\delta}^{\frac12}),\quad x\in{\mathfrak M}_{\tilde{\psi}}.
$$
At the same time, by Corollary above, there exists $z=\pi(y)$, $y\in M({\mathcal B})$, such that 
$$
(\tilde{\varphi}\circ\tau_{\frac{i}2})\bigl(S(x)\bigr)=\tilde{\varphi}\bigl(S(x)z\bigr),\quad x\in{\mathfrak M}_{\tilde{\psi}}.
$$

In particular, if $a\in{\mathcal A}$, then it becomes:
\begin{align}
\varphi\bigl(S(a)y\bigr)&=\tilde{\varphi}\bigl(S(\pi(a))z\bigr)=(\tilde{\varphi}\circ\tau_{\frac{i}2})\bigl(S(\pi(a))\bigr)
=\tilde{\varphi}(\tilde{\delta}^{\frac12}\pi(a)\tilde{\delta}^{\frac12})   \notag \\
&=\tilde{\varphi}\bigl(\pi(a)\tilde{\delta}^{\frac12}\tilde{\sigma}_{-i}(\tilde{\delta}^{\frac12})\bigr)
=\tilde{\varphi}\bigl(\pi(a)\tilde{\delta}^{\frac12}\lambda^{-\frac{i}2}\tilde{\delta}^{\frac12}\bigr)
=\tilde{\varphi}\bigl(\lambda^{-\frac{i}2}\pi(a)\tilde{\delta}\bigr)
\label{(eqn_modular1)}
\end{align}
by Proposition~\ref{vaesRN}, where it is noted that $\tilde{\sigma}_s(\tilde{\delta}^{it})=\lambda^{ist}\tilde{\delta}^{it}$ and that 
$\lambda$ is central.  We took $s=-i$, $t=-\frac{i}2$ in Equation~\eqref{(eqn_modular1)}. Note also that
\begin{equation}\label{(eqn_modular2)}
\varphi\bigl(S(a)y\bigr)=(\varphi\circ S)\bigl(S^{-1}(y)a\bigr)=\varphi\bigl(S^{-1}(y)a\delta\bigr).
\end{equation}

Compare Equations~\eqref{(eqn_modular1)} and \eqref{(eqn_modular2)}.  Since $\varphi$ is faithful on ${\mathcal A}$ (a core for $\tilde{\varphi}$), 
which in turn means the uniqueness of $\delta$ and $y$, this has to mean that $\delta\equiv\tilde{\delta}$ and $S^{-1}(y)\equiv\lambda^{-\frac{i}2}$, 
modulo possibly multiplication by positive real numbers.  In particular, it follows that $\delta$ is positive (self-adjoint) and that 
$y$ is central.  This is not saying that $y$ has to be a scalar, but that $y\in M({\mathcal B})\cap M({\mathcal C})$. Also, this is not saying that 
$\tilde{\delta}$ extends $\delta$, but that $\tilde{\delta}$ is possibly an extension of $p\delta$, where $p$ is a positive real number.  See below:

\begin{prop}\label{deltaisselfadjoint}
Let $\delta$ be the modular element, as defined in Proposition~\ref{modular} and Proposition~\ref{phiSresults}\,(3), such that 
$(\varphi\circ S)(x)=\varphi(x\delta)$, for all $x\in{\mathcal A}$.  Then under the quasi-invariance assumption, it turns out that 
$\delta$ is a positive self-adjoint element in ${\mathcal A}$.
\end{prop}

\begin{proof}
See the discussion given in the preceding paragraphs.
\end{proof}

This observation means that under the quasi-invariance assumption, the modular element $\delta$ is self-adjoint, so 
the results of Proposition~\ref{deltasa} (in Appendix) can be used.  Since the assumption was made at the ${}^*$-algebra 
level and the results are at the ${}^*$-algebra level, this indicates that there is possibly a way to give a direct, purely algebraic 
proof of the self-adjointness of the modular element $\delta$ from the quasi-invariance assumption.  We will not pursue 
that question here.  

Meanwhile, the operator $\lambda$ would be a generalization of the ``scaling constant'' in the quantum group theory 
(see Proposition~6.8 in \cite{KuVa}).  For our current purposes, its role will be downplayed.  In a future paper (such as \cite{BJKVD_qgroupoid3}, 
when we study the duality theory for $C^*$-algebraic quantum groupoids), we will have more occasions to discuss 
further implications of having $\tilde{\delta}$ and $\lambda$.

\subsection{Some additional consequences of the quasi-invariance condition}\label{sub5.4}

We will gather here a few additional technical results that are consequences of our quasi-invariance assumption.  They will 
be useful in the next subsection.

In Proposition~\ref{sigmasigma'delta} (in Appendix \S\ref{appx}), we gathered some results on the modular element 
$\delta$ under the modular automorphisms $\sigma$ and $\sigma^{\varphi\circ S}$.  With the quasi-invariance assumption, 
we can prove another result:

\begin{prop}\label{sigmasigma'delta_ex}
Consider the modular automorphisms $\sigma$ and $\sigma^{\varphi\circ S}$, which can be naturally extended to the 
multiplier algebra level. Under the quasi-invariance assumption, we have:
$$
\sigma^{-1}(a)=\delta[\sigma^{\varphi\circ S}]^{-1}(a)\delta^{-1} \ \ {\text and } \ \ [\sigma^{\varphi\circ S}]^{-1}(a)=\delta^{-1}\sigma^{-1}(a)\delta,
$$  
for any $a\in{\mathcal A}$.
\end{prop}

\begin{proof}
As a consequence of the quasi-invariance assumption, we can use Proposition~\ref{sigmasigma'commute}\,(3), 
the commutativity of $\sigma$ and $\sigma^{\varphi\circ S}$.  Since $\sigma^{\varphi\circ S}\circ\sigma=\sigma\circ\sigma^{\varphi\circ S}$, 
we can see quickly that $[\sigma^{\varphi\circ S}]^{-1}\circ\sigma=\sigma\circ[\sigma^{\varphi\circ S}]^{-1}$.

Applying this commutativity result to Proposition~\ref{sigmasigma'delta}\,(4), we obtain: 
$$
[\sigma^{\varphi\circ S}]^{-1}\bigl(\sigma(x)\bigr)=\sigma\bigl([\sigma^{\varphi\circ S}]^{-1}(x)\bigr)
=\delta^{-1}x\delta.  
$$
Here let $x=\sigma^{-1}(a)$, for $a\in{\mathcal A}$.   Then it becomes: 
$$[\sigma^{\varphi\circ S}]^{-1}(a)=\delta^{-1}\sigma^{-1}(a)\delta,$$ 
true for any $a\in{\mathcal A}$. 
Equivalently, we have: $\sigma^{-1}(a)=\delta[\sigma^{\varphi\circ S}]^{-1}(a)\delta^{-1}$, $\forall a\in{\mathcal A}$.
\end{proof}

In the below is one more consequence of the quasi-invariance assumption and the self-adjointness of $\delta$ (itself a consequence 
of the quasi-invariance):

\begin{prop}\label{sigmasigma'Delta}
Given the quasi-invariance assumption, we have:
$$
(\sigma^{-1}\otimes\sigma^{\varphi\circ S})(\Delta x)=\Delta\bigl(S^{-2}(x)\bigr).
$$
for any $x\in{\mathcal A}$.
\end{prop}

\begin{proof}
Note that for any $a\in{\mathcal A}$, we have $\sigma^{\varphi\circ S}(a)=\delta\sigma(a)\delta^{-1}$ by Proposition~\ref{sigmasigma'delta}\,(3), 
and due to the quasi-invariance, we have $\sigma^{-1}(a)=\delta[\sigma^{\varphi\circ S}]^{-1}(a)\delta^{-1}$, by Proposition~\ref{sigmasigma'delta_ex}. 
We thus have, for any $x\in{\mathcal A}$,
$$
(\sigma^{-1}\otimes\sigma^{\varphi\circ S})(\Delta x)=(\delta\otimes\delta)\bigl([\sigma^{\varphi\circ S}]^{-1}\otimes\sigma\bigr)(\Delta x)
(\delta^{-1}\otimes\delta^{-1}).
$$

Apply here the result $\Delta\bigl([\sigma^{\varphi\circ S}]^{-1}(x)\bigr)=\bigl([\sigma^{\varphi\circ S}]^{-1}\otimes S^2)(\Delta x)$, which is essentially 
Proposition~\ref{Deltasigma}\,(2).  Then the above expression becomes:
\begin{align}
(\sigma^{-1}\otimes\sigma^{\varphi\circ S})(\Delta x)=\dots
&=(\delta\otimes\delta)\bigl[(\operatorname{id}\otimes (\sigma\circ S^{-2}))(\Delta([\sigma^{\varphi\circ S}]^{-1}(x)))\bigr](\delta^{-1}\otimes\delta^{-1}) 
\notag \\
&=(\delta\otimes\delta)\bigl[(\operatorname{id}\otimes (S^{-2}\circ\sigma))(\Delta([\sigma^{\varphi\circ S}]^{-1}(x)))\bigr](\delta^{-1}\otimes\delta^{-1}), 
\notag
\end{align}
as $\sigma$ and $S^{-2}$ commute, again by the quasi-invariance (see Proposition~\ref{sigmasigma'commute}).

Use here the result $\Delta\bigl(\sigma(a)\bigr)=(S^2\otimes\sigma)(\Delta a)$, $a\in{\mathcal A}$, from Proposition~\ref{Deltasigma}, 
which can be also written as $(S^{-2}\otimes\operatorname{id})\Delta\bigl(\sigma(a)\bigr)=(\operatorname{id}\otimes\sigma)(\Delta a)$. 
Then the above expression becomes:
\begin{align}
&=(\delta\otimes\delta)\bigl[(S^{-2}\otimes S^{-2})(\Delta(\sigma([\sigma^{\varphi\circ S}]^{-1}(x))))\bigr](\delta^{-1}\otimes\delta^{-1})  \notag \\
&=(\delta\otimes\delta)\bigl[(S^{-2}\otimes S^{-2})(\Delta(\delta^{-1}x\delta))\bigr](\delta^{-1}\otimes\delta^{-1}),
\notag
\end{align}
where we used the result $\sigma\bigl([\sigma^{\varphi\circ S}]^{-1}(x)\bigr)=\delta^{-1}x\delta$, from Proposition~\ref{sigmasigma'delta}\,(4).

Note that by Proposition~\ref{deltasa} (because $\delta$ is self-adjoint), we have: 
$$
\Delta(\delta^{-1}x\delta)=\Delta(\delta^{-1})(\Delta x)\Delta(\delta)
=(\delta^{-1}\otimes\delta^{-1})E(\Delta x)E(\delta\otimes\delta)
=(\delta^{-1}\otimes\delta^{-1})(\Delta x)(\delta\otimes\delta).
$$
Combining all these observations, we thus have:
\begin{align}
(\sigma^{-1}\otimes\sigma^{\varphi\circ S})(\Delta x)=\dots
&=(\delta\otimes\delta)\bigl[(S^{-2}\otimes S^{-2})(\Delta(\delta^{-1}x\delta))\bigr](\delta^{-1}\otimes\delta^{-1})  \notag \\
&=(\delta\otimes\delta)\bigl[(S^{-2}\otimes S^{-2})((\delta^{-1}\otimes\delta^{-1})(\Delta x)(\delta\otimes\delta))\bigr]
(\delta^{-1}\otimes\delta^{-1})  \notag \\
&=(\delta\otimes\delta)(\delta^{-1}\otimes\delta^{-1})[(S^{-2}\otimes S^{-2})(\Delta x)]
(\delta\otimes\delta)(\delta^{-1}\otimes\delta^{-1})  \notag \\
&=\Delta\bigl(S^{-2}(x)\bigr).
\notag
\end{align}
Here, we used the result that $S(\delta)=\delta^{-1}$ and $S^2(\delta)=\delta$, from Proposition~\ref{deltasa}, 
and the property of the antipode that $(S\otimes S)(\Delta a)=\Delta^{\operatorname{cop}}\bigl(S(a)\bigr)$, 
$\forall a\in{\mathcal A}$, from Proposition~\ref{antipodeS}, applied twice.
\end{proof}

\subsection{The KMS weight $\varphi$}\label{sub5.5}

We have been working with an n.s.f. weight $\tilde{\varphi}$ at the von Neumann algebra level, but the time has come to consider its restriction 
to the $C^*$-algebra level.  By restricting the weight $\tilde{\varphi}$ on the von Neumann algebra $M=\pi({\mathcal A})''$ to the level of the 
$C^*$-algebra $A=\overline{\pi({\mathcal A})}^{\|\ \|}$, we obtain a faithful lower semi-continuous weight $\varphi$ on $A$.  

As noted earlier, the weight $\varphi$ extends the linear functional $\varphi$ on ${\mathcal A}$, in the sense that $\varphi\bigl(\pi(a)\bigr)=\varphi(a)$, 
for $a\in{\mathcal A}$.  For convenience, let us use the same notation for our $C^*$-algebra weight as the linear functional at the ${}^*$-algebra level. 
Denote the associated spaces by ${\mathfrak N}_{\varphi}=\bigl\{x\in A:\varphi(x^*x)<\infty\bigr\}$ and ${\mathfrak M}_{\varphi}
={\mathfrak N}_{\varphi}^*{\mathfrak N}_{\varphi}$.  Write $\Lambda_{\varphi}$ to denote the GNS map from ${\mathfrak N}_{\varphi}$ into 
${\mathcal H}$, where we can take our Hilbert space to be the same as before.

We can consider the operators $T$, $\nabla$, $J$ as before, because the Hilbert space remains the same. However, as noted in \S\ref{sub3.6}, 
we do not know whether the restriction of the modular automorphism group $(\tilde{\sigma}_t)$ to the $C^*$-algebra level would leave $A$ invariant, 
and whether the restriction is norm-continuous.  These are not automatic consequences of the modular theory.  Earlier, for the weights $\nu$ and $\mu$ 
at the base $C^*$-algebra level, we were benefitted by the existence of the canonical idempotent $E$.  However, that is not possible this time. 
We need a different approach.

Let us return back down to the ${}^*$-algebra level, and consider the modular element $\delta\in M({\mathcal A})$. For its properties, 
see Appendix (Section~\ref{appx}).  Note that due to our {\em quasi-invariance\/} assumption (see discussion given in \S\ref{sub5.3}, 
in particular Proposition~\ref{deltaisselfadjoint}), we can use the fact that $\delta$ is positive self-adjoint.  Also, the results of \S\ref{sub5.4} 
can be all used.

Using $\delta$, define a new Hilbert space ${\mathcal H}_{\delta}$, as follows:

\begin{prop}\label{H_delta}
Let $a,b\in{\mathcal A}$.  Then as $\delta$ is a positive element, we can define a positive sesquilinear form:
$$
(a,b)\mapsto\varphi(b^*\delta a).
$$
In this way, we can define a Hilbert space ${\mathcal H}_{\delta}$, together with an injective linear map 
$\Lambda_{\delta}:{\mathcal A}\to{\mathcal H}_{\delta}$, having a dense range in ${\mathcal H}_{\delta}$, 
such that 
$$
\bigl\langle\Lambda_{\delta}(a),\Lambda_{\delta}(b)\bigr\rangle=\varphi(b^*\delta a), \quad 
{\text { for all $a,b\in{\mathcal A}$.}}
$$
\end{prop}

\begin{proof}
As $\delta$ is positive, it is clear that $\varphi(a^*\delta a)$ is positive, for any $a\in{\mathcal A}$. 
Next, assume that $\varphi(a^*\delta a)=0$.  Then by the Schwarz inequality, we have, for any $b\in{\mathcal A}$, 
$$
\bigl|\varphi(b^*\delta a)\bigr|^2\le\varphi(b^*\delta b)\varphi(a^*\delta a),
$$
so $\varphi(b^*\delta a)=0$. Since $\varphi$ is faithful and since this is true for any $b\in{\mathcal A}$, 
this means that $\delta a=0$.  As $\delta$ is invertible, we must have $a=0$.  We see that 
$\varphi(a^*\delta a)=0$ if and only if $a=0$.

We thus have the positive definiteness, and so we obtain an inner product on ${\mathcal A}$.  By completing ${\mathcal A}$ 
with respect to the induced norm, we thereby obtain the Hilbert space ${\mathcal H}_{\delta}$, with the natural 
inclusion $\Lambda_{\delta}:{\mathcal A}\to{\mathcal H}_{\delta}$.
\end{proof}

We define an anti-linear, closed (unbounded) operator $Z$ from ${\mathcal H}$ to ${\mathcal H}_{\delta}$, 
in the following proposition:

\begin{prop}\label{Zoperator}
For $a\in{\mathcal A}$, define:
$$
Z_0\Lambda(a):=\Lambda_{\delta}\bigl(S(a^*)\bigr).
$$
Then:
\begin{enumerate}
\item $Z_0$ is a well-defined map from $\Lambda({\mathcal A})$ into ${\mathcal H}_{\delta}$.
\item $Z_0$ is a closable, so we can consider its closure $Z$. Then $Z$ becomes a closed, densely-defined, injective operator from ${\mathcal H}$ 
into ${\mathcal H}_{\delta}$, such that $\Lambda({\mathcal A})$ forms a core and $Z$ has a dense range.
\item $Z$ is anti-linear.
\item $\Lambda_{\delta}({\mathcal A})$ forms a core for $Z^*$, which is also a densely-defined, injective, and has a dense range, and given by
$$
Z^*\Lambda_{\delta}(a)=\Lambda\bigl(\delta^{-1}S(a)^*\delta\bigr), \quad {\text { for $a\in{\mathcal A}$.}}
$$
\end{enumerate}
\end{prop}

\begin{proof}
As $S$ is well-defined from ${\mathcal A}$ onto itself, and since $\Lambda({\mathcal A})$ is dense in ${\mathcal H}$ while 
$\Lambda_{\delta}({\mathcal A})$is dense in ${\mathcal H}_{\delta}$, with respect to the relevant norms, it is clear that $Z_0$ is well-defined, 
densely-defined, and has a dense range.  Meanwhile, for $a,b\in{\mathcal A}$, note that 
$$
\bigl\langle\Lambda_{\delta}(S(a^*)),\Lambda_{\delta}(b)\bigr\rangle=\varphi\bigl(b^*\delta S(a^*)\bigr)
=\varphi\bigl(S(a^*\delta^{-1}S^{-1}(b^*)\bigr)=\varphi\bigl(S(a^*\delta^{-1}S(b)^*)\bigr),
$$
because $S(\delta)=\delta^{-1}$ (see Proposition~\ref{deltasa} in Appendix) 
and $S\bigl(S(b)^*\bigr)^*=b$.  Since $\varphi\circ S=\varphi(\,\cdot\,\delta)$, 
this becomes:
$$
\bigl\langle\Lambda_{\delta}(S(a^*)),\Lambda_{\delta}(b)\bigr\rangle
=\varphi\bigl(a^*\delta^{-1}S(b)^*\delta\bigr)
=\bigl\langle\Lambda(\delta^{-1}S(b)^*\delta),\Lambda(a)\bigr\rangle.
$$
From this, it is not difficult to show that $Z_0$ is closable.  Let us denote by $Z$ its closure.  Then $Z$ becomes closed, densely-defined, 
and has a dense range, with $\Lambda({\mathcal A})$ forming a core. It is anti-linear since $S$ is.

With $Z^*\Lambda_{\delta}(b)=\Lambda\bigl(\delta^{-1}S(b)^*\delta\bigr)$, for $b\in{\mathcal A}$, this result can be 
expressed as
$$
\bigl\langle Z\Lambda(a),\Lambda_{\delta}(b)\bigr\rangle=\bigl\langle Z^*\Lambda_{\delta}(b),\Lambda(a)\bigr\rangle
=\overline{\bigl\langle\Lambda(a),Z^*\Lambda_{\delta}(b)\bigr\rangle}.
$$
From this, it is easy to notice that $Z$ is injective. It is also apparent that $Z^*$ is itself a closed, injective, anti-linear operator, that is densely-defined 
with a dense range.
\end{proof}

Define $P:=Z^*Z$.  As a consequence of Proposition~\ref{Zoperator}, we see that $P$ is a closed, positive, injective operator on ${\mathcal H}$, 
which is densely-defined and has a dense range.  It is clear that $\Lambda({\mathcal A})$ forms a core for $P$.  Moreover, we have:
\begin{equation}\label{(Poperator)}
P\Lambda(a)=Z^*Z\Lambda(a)=Z^*\Lambda_{\delta}\bigl(S(a^*)\bigr)=\Lambda\bigl(\delta^{-1}S(S(a^*))^*\delta\bigr)
=\Lambda\bigl(\delta^{-1}S^{-2}(a)\delta\bigr),
\end{equation}
because $S(x^*)=S^{-1}(x)^*$, which is applied twice.  Next proposition gives a useful relationship between the operators 
$W$, $\nabla(=T^*T)$, and $P(=Z^*Z)$:

\begin{prop}\label{WnablaP}
For any $a,b\in{\mathcal A}$, we have:
$$
W(\nabla\otimes P)\bigl(\Lambda(a)\otimes\Lambda(b)\bigr)
=(\nabla\otimes\nabla)W\bigl(\Lambda(a)\otimes\Lambda(b)\bigr).
$$
\end{prop}

\begin{proof}
Let $a,b,c,d\in{\mathcal A}$ be arbitrary.  Then 
\begin{align}
&\bigl\langle W(\nabla\otimes P)(\Lambda(a)\otimes\Lambda(b)),\Lambda(c)\otimes\Lambda(d)\bigr\rangle \notag \\
&=\bigl\langle T^*T\Lambda(a)\otimes Z^*Z\Lambda(b),W^*(\Lambda(c)\otimes\Lambda(d))\bigr\rangle
\notag \\
&=\overline{\bigl\langle T\Lambda(a)\otimes Z\Lambda(b),
(T\otimes Z)(\Lambda\otimes\Lambda)((\Delta d)(c\otimes1))\bigr\rangle} \notag \\
&=\bigl\langle(\Lambda\otimes\Lambda_{\delta})((c^*\otimes1)(\operatorname{id}\otimes S)(\Delta(d^*))),
\Lambda(a^*)\otimes\Lambda_{\delta}(S(b^*))\bigr\rangle,
\notag
\end{align}
using the characterization of $W^*$ as in Proposition~\ref{W}, the definitions of $T$ and $Z$, and the fact that $S(b^*)^*=S^{-1}(b)$. 
Note that the inner product is in ${\mathcal H}\otimes{\mathcal H}_{\delta}$.  Continuing, this becomes:
\begin{align}
(RHS)&=(\varphi\otimes\varphi)\bigl((a\otimes S^{-1}(b))(1\otimes\delta)(c^*\otimes1)
(\operatorname{id}\otimes S)(\Delta(d^*))\bigr)  \notag \\
&=(\varphi\otimes\varphi)\bigl((ac^*\otimes S^{-1}(b)\delta)(\operatorname{id}\otimes S)(\Delta(d^*))\bigr).
\notag
\end{align}
We thus have, so far:
\begin{equation}\label{(WnablaP_eq1)}
\bigl\langle W(\nabla\otimes P)(\Lambda(a)\otimes\Lambda(b)),\Lambda(c)\otimes\Lambda(d)\bigr\rangle
=(\varphi\otimes\varphi)\bigl((ac^*\otimes S^{-1}(b)\delta)(\operatorname{id}\otimes S)(\Delta(d^*))\bigr).
\end{equation}

Meanwhile, we have:
\begin{align}
&\bigl\langle (\nabla\otimes\nabla)W(\Lambda(a)\otimes\Lambda(b)),\Lambda(c)\otimes\Lambda(d)\bigr\rangle \notag \\
&=\bigl\langle (T^*T\otimes T^*T)W(\Lambda(a)\otimes\Lambda(b)),\Lambda(c)\otimes\Lambda(d)\bigr\rangle
\notag \\
&=\overline{\bigl\langle (T\otimes T)(\Lambda\otimes\Lambda)((S^{-1}\otimes\operatorname{id})(\Delta b)(a\otimes1)),
T\Lambda(c)\otimes T\Lambda(d)\bigr\rangle}
\notag \\
&=\bigl\langle\Lambda(c^*)\otimes\Lambda(d^*),
(\Lambda\otimes\Lambda)((a^*\otimes1)(S\otimes\operatorname{id})(\Delta(b^*)))\bigr\rangle
\notag \\
&=(\varphi\otimes\varphi)\bigl((S^{-1}\otimes\operatorname{id})(\Delta b)(a\otimes1)(c^*\otimes d^*)\bigr)
=\varphi\bigl(S^{-1}[(\operatorname{id}\otimes\varphi)((\Delta b)(1\otimes d^*))]ac^*\bigr),
\notag
\end{align}
by the characterization of $W$ given in Proposition~\ref{W}\,(2), and the fact that $S^{-1}(x)^*=S(x^*)$.  By using 
a characterization of $S$ given in Proposition~\ref{antipodeS}\,(1), we can go further:
\begin{align}
(RHS)&=\varphi\bigl(S^{-2}[(\operatorname{id}\otimes\varphi)((1\otimes b)\Delta(d^*))]ac^*\bigr)  \notag \\
&=(\varphi\otimes\varphi)\bigl((1\otimes b)(S^{-2}\otimes\operatorname{id})(\Delta(d^*))(ac^*\otimes1)\bigr)  \notag \\
&=(\varphi\otimes\varphi)\bigl((ac^*\otimes b)(\sigma\otimes\operatorname{id})[(S^{-2}\otimes\operatorname{id})
(\Delta(d^*))]\bigr).
\label{(WnablaP_eq2)}
\end{align}
Note here that we used the modular automorphism $\sigma$.  

In Proposition~\ref{sigmasigma'Delta}, we saw that 
$(\sigma^{-1}\otimes\sigma^{\varphi\circ S})(\Delta x)=\Delta\bigl(S^{-2}(x)\bigr)=(S^{-2}\otimes S^{-2})(\Delta x)$, for $x\in{\mathcal A}$. 
As a consequence, it follows that for all $x\in{\mathcal A}$, 
$$
(\sigma\otimes\operatorname{id})\bigl[(S^{-2}\otimes\operatorname{id})(\Delta x)\bigr]
=\bigl(\operatorname{id}\otimes(S^2\circ\sigma^{\varphi\circ S})\bigr)(\Delta x)
=\bigl(\operatorname{id}\otimes(S\circ\sigma^{-1})\bigr)\bigl[(\operatorname{id}\otimes S)(\Delta x)\bigr],
$$
because $\sigma^{\varphi\circ S}=S^{-1}\circ\sigma^{-1}\circ S$ (see proof of Proposition~\ref{automorphismsigmaS}). Then  
Equation~\eqref{(WnablaP_eq2)} becomes:
\begin{align}
&\bigl\langle (\nabla\otimes\nabla)W(\Lambda(a)\otimes\Lambda(b)),\Lambda(c)\otimes\Lambda(d)\bigr\rangle=\dots 
\notag \\
&=(\varphi\otimes\varphi)\bigl((ac^*\otimes b)(\operatorname{id}\otimes(S\circ\sigma^{-1}))
[(\operatorname{id}\otimes S)(\Delta(d^*))]\bigr)
\notag \\
&=\bigl(\varphi\otimes(\varphi\circ S)\bigr)\bigl((ac^*\otimes1)(\operatorname{id}\otimes\sigma^{-1})
[(\operatorname{id}\otimes S)(\Delta(d^*))](1\otimes S^{-1}(b))\bigr)
\notag \\
&=(\varphi\otimes\varphi)\bigl((ac^*\otimes1)(\operatorname{id}\otimes\sigma^{-1})
[(\operatorname{id}\otimes S)(\Delta(d^*))](1\otimes S^{-1}(b)\delta)\bigr)
\notag \\
&=(\varphi\otimes\varphi)\bigl((ac^*\otimes S^{-1}(b)\delta)(\operatorname{id}\otimes S)(\Delta(d^*))\bigr).
\notag
\end{align}
We thus have:
\begin{equation}\label{(WnablaP_eq3)}
\bigl\langle (\nabla\otimes\nabla)W(\Lambda(a)\otimes\Lambda(b)),\Lambda(c)\otimes\Lambda(d)\bigr\rangle
=(\varphi\otimes\varphi)\bigl((ac^*\otimes S^{-1}(b)\delta)(\operatorname{id}\otimes S)(\Delta(d^*))\bigr).
\end{equation}

Compare Equations~\eqref{(WnablaP_eq1)} and \eqref{(WnablaP_eq3)}.  We conclude that 
$$
\bigl\langle W(\nabla\otimes P)(\Lambda(a)\otimes\Lambda(b)),\Lambda(c)\otimes\Lambda(d)\bigr\rangle
=\bigl\langle (\nabla\otimes\nabla)W(\Lambda(a)\otimes\Lambda(b)),\Lambda(c)\otimes\Lambda(d)\bigr\rangle,
$$
true for any $c,d\in{\mathcal A}$.  So we have:
$$
W(\nabla\otimes P)(\Lambda(a)\otimes\Lambda(b))=(\nabla\otimes\nabla)W(\Lambda(a)\otimes\Lambda(b)),
$$
for any $a,b\in{\mathcal A}$.
\end{proof}

Note, by the way that since we are working with unbounded operators $P$ and $\nabla$, which are only 
densely-defined, the above result does not necessarily mean $W(\nabla\otimes P)=(\nabla\otimes\nabla)W$. 
To be precise, considering the domains, this should be written as $W(\nabla\otimes P)\subseteq(\nabla\otimes\nabla)W$.

The situation becomes similar to what we had earlier for $W(L\otimes\nabla)\subseteq(L\otimes\nabla)W$ (in \S\ref{sub4.1}). 
As was in that case, the non-unitarity of $W$ (being only a partial isometry) means the need for a more roundabout approach. 
By following more or less the same procedure (similar to Propositions~4.13 -- 4.17 in \cite{BJKVD_qgroupoid2}), we obtain 
the following:

\begin{prop}\label{nablaP_main}
We have:
\begin{enumerate}
\item The restrictions $(\nabla\otimes P)|_{\operatorname{Ran}(E)}$, $(\nabla\otimes P)|_{\operatorname{Ker}(W)}$, 
$(\nabla\otimes\nabla)|_{\operatorname{Ran}(G)}$, $(\nabla\otimes\nabla)|_{\operatorname{Ker}(W^*)}$ become 
valid operators on the subspaces $\operatorname{Ran}(E)$, $\operatorname{Ker}(W)$, 
$\operatorname{Ran}(G)$, $\operatorname{Ker}(W^*)$, respectively. 
\item For any $z\in\mathbb{C}$, we have:
$$
W(\nabla^z\otimes P^z)\subseteq(\nabla^z\otimes\nabla^z)W \ 
{\text { and }} \ 
W^*(\nabla^z\otimes\nabla^z)\subseteq(\nabla^z\otimes P^z)W^*.
$$
\item Let $t\in\mathbb{R}$.  The following results hold on the whole space ${\mathcal H}\otimes{\mathcal H}$:
$$
(\nabla^{it}\otimes\nabla^{it})W(\nabla^{-it}\otimes P^{-it})=W,
$$
$$
(\nabla^{it}\otimes1)W(\nabla^{-it}\otimes1)=(1\otimes\nabla^{-it})W(1\otimes P^{it}).
$$
\end{enumerate}
\end{prop}

\begin{rem}
We will skip the detailed proof.  Modify the procedure taken in Propositions~4.13 -- 4.17 in \cite{BJKVD_qgroupoid2}, and 
see also the comments made in the remark following Proposition~\ref{Lnabla_main}.  While not exactly the same, the overall 
idea is similar.  This is fundamentally about dealing with the unbounded operators in relation to a partial isometry, not really 
using any specific results on quantum groupoids.  Note that when $z=it$, the operators $\nabla^{it}$ and $P^{it}$ become 
bounded, so the domain issue becomes simpler.
\end{rem}

As a consequence of Proposition~\ref{nablaP_main}, we are now ready to resolve our question on our modular automorphism 
group.  For this, consider $\omega\in{\mathcal B}({\mathcal H})_*$ and let $t\in\mathbb{R}$.  Apply $\operatorname{id}
\otimes\omega$ to the result (3) of Proposition~\ref{nablaP_main}.  Then we have
$$
\nabla^{it}(\operatorname{id}\otimes\omega)(W)\nabla^{-it}=(\operatorname{id}\otimes\theta)(W),
$$
where $\theta\in{\mathcal B}({\mathcal H})_*$ is such that $\theta(X)=\omega(\nabla^{-it}XP^{it})$. for $X\in{\mathcal B}
({\mathcal H})$. 

As elements of the form $(\operatorname{id}\otimes\omega)(W)$, $\omega\in{\mathcal B}({\mathcal H})_*$, 
generate the $C^*$-algebra $A$, the observation above shows that for any $t\in\mathbb{R}$, we have 
$\nabla^{it}a\nabla^{-it}\in A$, for any $a\in A$.  We can also observe that $t\mapsto \nabla^{it}a\nabla^{-it}$ is 
norm-continuous.

We can thus justify the following:

\begin{defn}\label{C*modularautomorphismgroup}
Define the norm-continuous one-parameter group $\sigma=(\sigma_t)$ on the $C^*$-algebra $A$, by
$$
\sigma_t(x)=\nabla^{it}a\nabla^{-it},
$$
for all $t\in\mathbb{R}$, $a\in A$.

The one-parameter group $(\sigma_t)$ is a restriction of the modular automorphism group $(\tilde{\sigma}_t)$ 
for the n.s.f.~weight $\tilde{\varphi}$.  With $(\sigma_t)$, the faithful lower semi-continuous weight $\varphi$ 
becomes a KMS weight on $A$.  Its KMS properties are inherited from the properties of $\tilde{\varphi}$. 
In particular, we have $\varphi\circ\sigma_t=\varphi$; and for any $x\in{\mathcal D}(\sigma_{\frac{i}2})$, we have 
$\varphi(x^*x)=\varphi\bigl(\sigma_{\frac{i}2}(x)\sigma_{\frac{i}2}(x)^*\bigr)$.  Also note that 
$\Lambda_{\varphi}\bigl(\sigma_t(x)\bigr)=\nabla^{it}\Lambda_{\varphi}(x)$, for any $t\in\mathbb{R}$ and any 
$x\in{\mathfrak N}_{\varphi}$.
\end{defn}

\begin{prop}\label{phipsiKMSweights}
Let $\varphi$ denote the weight on the $C^*$-algebra $A$, given by $\varphi=\tilde{\varphi}|_A$.  Similarly, let $\psi$ denote the 
weight on the $C^*$-algebra $A$, given by $\psi=\tilde{\psi}|_A$.

Then $\varphi$ and $\psi$ are KMS weights on $A$.  Their modular automorphism group $(\sigma_t)_{t\in\mathbb{R}}$ for $\varphi$ and 
$(\sigma'_t)_{t\in\mathbb{R}}$ for $\psi$ are restrictions of $(\tilde{\sigma}_t)_{t\in\mathbb{R}}$ and $(\tilde{\sigma}'_t)_{t\in\mathbb{R}}$, 
respectively, and they leave $A$ invariant and are norm-continuous.

We also have $\psi=\varphi\circ R$, where $R$ is the unitary antipode.  
\end{prop}

\begin{proof}
With Definition~\ref{C*modularautomorphismgroup} and the paragraph following it, we showed that $\varphi$ is a KMS weight on $A$, with 
the modular automorphism group $(\sigma_t)_{t\in\mathbb{R}}$. Its properties are inherited from those of the n.s.f. weight $\tilde{\varphi}$.

Since we know $\tilde{\psi}=\tilde{\varphi}\circ\widetilde{R}$ and $\tilde{\sigma}'_t=\widetilde{R}\circ\tilde{\sigma}_{-t}\circ\widetilde{R}$, for all 
$t\in\mathbb{R}$, while $\widetilde{R}$ restricts to the anti-isomorphism $R$ on $A$, it is evident that 
$$
\psi=\varphi\circ R \quad {\text { and }} \quad \sigma'_t=R\circ\sigma_{-t}\circ R, \forall t\in\mathbb{R}.
$$
It is also clear that $(\sigma'_t)_{t\in\mathbb{R}}$ leaves $A$ invariant and is norm-continuous.
\end{proof}

\section{The $C^*$-algebraic locally compact quantum groupoid}\label{sec6}

In \cite{BJKVD_qgroupoid1}, \cite{BJKVD_qgroupoid2}, Van Daele and the author developed a $C^*$-algebraic 
framework of a class of $C^*$-algebraic locally compact quantum groupoids ({\em quantum groupoids of separable type\/}). 
The definition is given below (see Definition~4.8 of \cite{BJKVD_qgroupoid1} and Definition~1.2 of \cite{BJKVD_qgroupoid2}):

\begin{defn}\label{definitionlcqgroupoid}
The data $(A,\Delta,E,B,\nu,\varphi,\psi)$ defines a {\em locally compact quantum 
groupoid of separable type\/}, if
\begin{itemize}
  \item $A$ is a $C^*$-algebra.
  \item $\Delta:A\to M(A\otimes A)$ is a comultiplication on $A$.
  \item $B$ is a non-degenerate $C^*$-subalgebra of $M(A)$.
  \item $\nu$ is a KMS weight on $B$.
  \item $E$ is the canonical idempotent of $(A,\Delta)$.  That is,
  \begin{enumerate}
    \item $\Delta(A)(A\otimes A)$ is dense in $E(A\otimes A)$ and $(A\otimes A)\Delta(A)$ 
is dense in $(A\otimes A)E$;
    \item there exists a $C^*$-subalgebra $C\cong B^{\operatorname{op}}$ contained in $M(A)$, 
with a ${}^*$-anti-isomorphism $R=R_{BC}:B\to C$, so that $E\in M(B\otimes C)$ and the triple 
$(E,B,\nu)$ forms a separability triple;
    \item $E\otimes1$ and $1\otimes E$ commute, and we have:
$$
(\operatorname{id}\otimes\Delta)(E)=(E\otimes1)(1\otimes E)=(1\otimes E)(E\otimes1)
=(\Delta\otimes\operatorname{id})(E).
$$
  \end{enumerate}
  \item $\varphi$ is a KMS weight, and is left invariant.
  \item $\psi$ is a KMS weight, and is right invariant.
  \item There exists a (unique) one-parameter group of automorphisms 
$(\theta_t)_{t\in\mathbb{R}}$ of $B$ such that $\nu\circ\theta_t=\nu$ and that 
$\sigma^{\varphi}_t|_B=\theta_t$, $\forall t\in\mathbb{R}$.
\end{itemize}
\end{defn}

\begin{rem}
We will refer the details to the main papers.  For instance, the notion of the canonical idempotent is summarized in 
Definition~3.7 of \cite{BJKVD_qgroupoid1}.

This definition is similar, but different from that of {\em measured quantum groupoids\/}, in the von Neumann algebra 
setting \cite{LesSMF}, \cite{EnSMF}.  The von Neumann algebra setting may be a bit more general, which is rather 
related to the algebraic framework of multiplier Hopf algebroids \cite{Timm_aqgintegral}.  There are some subtle 
differences between the two locally compact frameworks.
\end{rem}

Let us re-cap the constructions we carried out so far: Starting from a purely algebraic, weak multiplier Hopf ${}^*$-algebra 
with a faithful integral (see Definition~\ref{aqgdefn}), without any additional conditions other than the {\em quasi-invariance 
assumption\/} (see Subsection \S\ref{sub5.3}), we wish to verify that the data indeed gives us a $C^*$-algebraic quantum 
groupoid of separable type, as in Definition~\ref{definitionlcqgroupoid} above. 

Our $C^*$-algebra was defined in Definition~\ref{C*algebraA}, extending the ${}^*$-algebra ${\mathcal A}$. The 
{\em comultiplication\/} $\Delta:A\to M(A\otimes A)$ was given in Definition~\ref{comultiplication} and 
Theorem~\ref{comultiplicationrepresentation}.

The base $C^*$-algebras $B$ and $C$ were defined in Definition~\ref{BandC}. They are equipped with KMS weights 
$\nu$ and $\mu$, respectively, which actually extends the distinguished linear functionals at the ${}^*$-algebra level 
(see Proposition~\ref{KMSweightnu}).  
There exists a $C^*$-anti-isomorphism $R:B\to C$, while the canonical idempotent $E$ at the ${}^*$-algebra level 
extends to the {\em separability idempotent\/} $E\in M(B\otimes C)$.  See Proposition~\ref{Esepid}.

The idempotent $E$ was further shown to satisfy additional properties, making it a valid {\em canonical idempotent\/} 
at the $C^*$-level. See Proposition~\ref{EatC*level}.

Finding a suitable left-invariant weight $\varphi$ and a right-invariant weight $\psi$ was rather tricky. We first extended 
the left integral $\varphi$ at the ${}^*$-algebra level to an n.s.f. weight $\tilde{\varphi}$, but it took a while to establish that 
its restriction to the $C^*$-algebra level is a KMS weight because it relied on results that use the quasi-invariance assumption. 

As for $\psi$, we could not just attempt to extend the right integral at the ${}^*$-algebra level. Instead, we used results at 
the purely algebraic level saying the existence of both a left integral and a right integral imply the existence of the antipode 
(Theorem~3.15 of \cite{BJKVD_LSthm}), then carried out a polar decomposition of the antipode to have it established at 
the operator algebra level, which allowed us obtain various technical results.  In the end, the right weight was chosen to be 
$\varphi\circ R$, where $R$ is the unitary antipode that come from the polar decomposition of the antipode.  This weight 
does not necessarily have to be an extension of the original right integral. 

Along the way, an important role was played by the modular element $\delta$ at the ${}^*$-algebra level, as well as its 
operator algebraic counter part $\tilde{\delta}$, which provided a relationship between the extended weights $\tilde{\varphi}$ 
and $\tilde{\psi}$.  Having the quasi-invariance assumption was necessary for this to work, which was to be expected 
because some form of a quasi-invariance had to be assumed already in the framework of classical locally compact groupoids.

What remains to be shown is the result verifying that the KMS weights $\varphi$ and $\psi$ thus obtained are indeed left 
invariant and right invariant, respectively.  This is not too difficult, because we already have corresponding results at the 
von Neumann algebra level (Proposition~\ref{leftinvariant_phiext} and Proposition~\ref{rightinvariant_psiext}).  See below:

\begin{prop}\label{phipsi_invariantweights}
Let $\varphi$ and $\psi$ be the KMS weights established in Proposition~\ref{phipsiKMSweights}.  Then
\begin{enumerate}
  \item For any $a\in{\mathfrak M}_{\varphi}$, we have 
$\Delta a\in\overline{\mathfrak M}_{\operatorname{id}\otimes\varphi}$ and
$(\operatorname{id}\otimes\varphi)(\Delta a)\in M(C)$.
  \item For any $a\in{\mathfrak M}_{\psi}$, we have 
$\Delta a\in\overline{\mathfrak M}_{\psi\otimes\operatorname{id}}$ and
$(\psi\otimes\operatorname{id})(\Delta a)\in M(B)$.
\end{enumerate}
\end{prop}

\begin{proof}
(1). We showed in Proposition~\ref{leftinvariant_phiext} a corresponding result for the weight $\tilde{\varphi}$. We can just use the same proof, 
replacing $\tilde{\varphi}$ with $\varphi$ and replacing the strong convergence with the norm convergence.  Since $\overline{M({\mathcal C})}^{\|\ \|}
=M(C)$, we can see that for any $a\in{\mathfrak M}_{\varphi}$, we have $\Delta a\in\overline{\mathfrak M}_{\operatorname{id}\otimes\varphi}$, and that
$$
(\operatorname{id}\otimes\varphi)(\Delta a)\in M(C).
$$
(2). Since $\psi=\varphi\circ R$, and since we know $(R\otimes R)(\Delta x)=\Delta^{\operatorname{cop}}\bigl(R(x)\bigr)$ from Proposition~\ref{DeltaR}, 
the right invariance of $\psi$ follows from the left invariance of $\varphi$.  

More specifically, if $a\in{\mathfrak M}_{\psi}$, which means 
$R(a)\in{\mathfrak M}_{\varphi}$, we would have $\Delta\bigl(R(a)\bigr)\in\overline{\mathfrak M}_{\operatorname{id}\otimes\varphi}$, and that
$(\operatorname{id}\otimes\varphi)\bigl(\Delta(R(a))\bigr)\in M(C)$.  But then, we have
\begin{align}
(\psi\otimes\operatorname{id})(\Delta a)&=\bigl((\varphi\circ R)\otimes\operatorname{id}\bigr)(\Delta a)
=R\bigl((\varphi\otimes\operatorname{id})((R\otimes R)(\Delta a))\bigr) \notag \\
&=R\bigl((\operatorname{id}\otimes\varphi)(\Delta(R(a)))\bigr)\in R\bigl(M(C)\bigr)=M(B).
\notag
\end{align}
\end{proof}

Finally, observe the last requirement in Definition~\ref{definitionlcqgroupoid}, about the restriction of $\sigma$ to the base $C^*$-algebra $B$. 
But this is an immediate consequence of none other than the {\em quasi-invariance assumption\/} we required earlier, as $\sigma|_{\mathcal B}$ 
would play the role of an analytic generator for $(\theta_t)$.  

We noted earlier that as we are developing a locally compact theory, we need some form of a quasi-invariance condition, just as in the classical 
locally compact groupoid case \cite{Renbook}, \cite{Patbook}. The last condition in Definition~\ref{definitionlcqgroupoid} is needed for that purpose. 
What we are noticing is that for a purely algebraic object of a weak multiplier Hopf ${}^*$-algebra (with a faithful integral) to allow a construction 
of a $C^*$-algebraic quantum groupoid, some form of the quasi-invariance property is required even at the algebra level, which turns out to be 
the quasi-invariance assumption we required in Section~\S\ref{sub5.3}.

Summarizing, we have the following conclusion:

\begin{thm}
Let $({\mathcal A},\Delta,E)$ be a weak multiplier Hopf ${}^*$-algebra with a single faithful integral, as in Definition~\ref{aqgdefn}. 
With the quasi-invariance assumption (as in \S\ref{sub5.3}), we can construct from it a $C^*$-algebraic quantum groupoid of separable type, 
in the sense of \cite{BJKVD_qgroupoid1}, \cite{BJKVD_qgroupoid2}.
\end{thm}

Now that the construction is done, we can take full advantage of the already-developed theory in \cite{BJKVD_qgroupoid2}. 
There are alternative representations for the $C^*$-algebra and there are multiple equivalent characterizations for the antipode, 
among other results.  Some more relations between the base algebras $B$, $C$ and the total algebra $A$ can be found.

On the other hand, we did not pursue the duality aspect in this paper (see \S\ref{sub1.6}).  Ideally, it would make the picture complete 
if one can confirm that the $C^*$-extension of the dual weak multiplier Hopf ${}^*$-algebra $(\widehat{\mathcal A},\widehat{\Delta})$ is the dual 
in the $C^*$-context of the $C^*$-algebraic quantum groupoid.  We will postpone that project to a future occasion, after the paper on the duality 
theory for the $C^*$-algebraic quantum groupoids (\cite{BJKVD_qgroupoid3}, in preparation) is finished.

\section{Appendix: The modular element at the ${}^*$-algebra level} \label{appx}

In this Appendix, we gather some purely algebraic results for a weak multiplier Hopf ${}^*$-algebra.  As in \S\ref{sub1.5}, 
we assume the existence of a single faithful positive left integral $\varphi$.

In the purely algebraic setting, we noted in Proposition~\ref{phi_modularautomorphism} the existence of the modular 
automorphism $\sigma$ for $\varphi$.  We also noted of the existence of an invertible element $\delta\in M({\mathcal A})$, 
called the {\em modular element}, relating the functionals $\varphi$ and $\varphi\circ S$.  The modular element behaves like 
a modular function in the classical setting.

In this Appendix, we gather some results regarding the functional $\varphi\circ S$, the modular automorphisms for $\varphi$ 
and $\varphi\circ S$, and the modular element $\delta$.  While what appear below are all purely algebraic results, and 
some results are likely already known, the author could not find a good reference for some of these results (especially 
regarding $\delta$), and some results here may be new.  As such, unlike in \S\ref{sec1}, all the proofs are given here.

\subsection{The functional $\varphi\circ S$} \label{suba.1}

Consider the functional $\varphi\circ S$, where $\varphi$ is our faithful left integral $\varphi$, and $S$ is the antipode map.  Since $S$ 
is a bijection on ${\mathcal A}$, we see that $\varphi\circ S$ is also a faithful linear functional on ${\mathcal A}$.  However, there is no reason 
to expect that it is positive.

Recall the existence of the modular automorphism, $\sigma$, for the functional $\varphi$.  This means that there exists a similar object for 
the functional $\varphi\circ S$, which would be the modular automorphism $\sigma^{\varphi\circ S}$ for $\varphi\circ S$.

\begin{prop}\label{automorphismsigmaS}
There exists an automorphism $\sigma^{\varphi\circ S}$ of ${\mathcal A}$, such that
$$
(\varphi\circ S)(ab)=(\varphi\circ S)\bigl(b\sigma^{\varphi\circ S}(a)\bigr), \quad \forall a,b\in{\mathcal A},
$$
and also $(\varphi\circ S)\bigl(\sigma^{\varphi\circ S}(a)\bigr)=(\varphi\circ S)(a)$, for all $a\in{\mathcal A}$. 
\end{prop}

\begin{proof}
For all $a,b\in{\mathcal A}$, using Proposition~\ref{phi_modularautomorphism} and knowing $S$ is an anti-isomorphism, we have
$$
(\varphi\circ S)(ab)=\varphi\bigl(S(b)S(a)\bigr)=\varphi\bigl(\sigma^{-1}(S(a))S(b)\bigr)=(\varphi\circ S)\bigl(bS^{-1}(\sigma^{-1}(S(a)))\bigr).
$$
This shows that with $\sigma^{\varphi\circ S}(a):=(S^{-1}\circ\sigma^{-1}\circ S)(a)$, $a\in{\mathcal A}$, we have the desired automorphism. 
We can also see that 
$$
(\varphi\circ S)\bigl(\sigma^{\varphi\circ S}(a)\bigr)=(\varphi\circ S)\bigl((S^{-1}\circ\sigma^{-1}\circ S)(a)\bigr)=\varphi\bigl(\sigma^{-1}(S(a)\bigr)
=(\varphi\circ S)(a).
$$
\end{proof}

It turns out that $\varphi\circ S$ is a right integral, which is sort of expected:

\begin{prop}\label{phiS_rightinvariant}
The functional $\varphi\circ S$ is right invariant, in the sense of Proposition~\ref{invariantintegrals}:
$$
\bigl((\varphi\circ S)\otimes\operatorname{id}\bigr)(\Delta a)\in M({\mathcal B}), \quad {\text { for all $a\in{\mathcal A}$.}}
$$
\end{prop}

\begin{proof}
Let $a\in{\mathcal A}$. By Proposition~\ref{antipodeS}\,(4), we have:
$$
\bigl((\varphi\circ S)\otimes\operatorname{id}\bigr)(\Delta a)
=S^{-1}\bigl((\varphi\otimes\operatorname{id})((S\otimes S)(\Delta a))\bigr)
=S^{-1}\bigl((\operatorname{id}\otimes\varphi)(\Delta(S(a)))\bigr).
$$
Since $(\operatorname{id}\otimes\varphi)(\Delta(S(a)))\in M({\mathcal B})$ by the left invariance of $\varphi$, and 
since $S^{-1}\bigl(M({\mathcal B})\bigr)=M\bigl({\mathcal C}\bigr)$ by Proposition~\ref{antipodeS}\,(3), this means that 
$\bigl((\varphi\circ S)\otimes\operatorname{id}\bigr)(\Delta a)\in M\bigl({\mathcal C}\bigr)$.
\end{proof}

As $\varphi\circ S$ is a right integral on ${\mathcal A}$, the results for right integrals apply:

\begin{prop}\label{phiSresults}
The function $\varphi\circ S$ is a right integral.  Therefore the following results hold:
\begin{enumerate}
\item $\nu\bigl(((\varphi\circ S)\otimes\operatorname{id})(\Delta x)\bigr)=(\varphi\circ S)(x)$, for all $x\in{\mathcal A}$.
\item For all $a\in{\mathcal A}$, we have
$$
\bigl((\varphi\circ S)\otimes\operatorname{id}\bigr)(\Delta a)
=\bigl((\varphi\circ S)\otimes\operatorname{id}\bigr)\bigl((a\otimes 1)F_1\bigr)
=\bigl((\varphi\circ S)\otimes\operatorname{id}\bigr)\bigl(F_3(a\otimes1)\bigr),
$$
where $F_1=(\operatorname{id}\otimes S)(E)\in M({\mathcal A}\odot{\mathcal A})$ and 
$F_3=(\operatorname{id}\otimes S^{-1})(E)\in M({\mathcal A}\odot{\mathcal A})$.
\item There exists a unique invertible element $\delta\in M({\mathcal A})$ such that $(\varphi\circ S)(x)=\varphi(x\delta)$, for all $x\in{\mathcal A}$. 
We refer to $\delta$ as the modular element.
\item There is an alternative characterization of the antipode map $S$, in terms of the functional $\varphi\circ S$:
$$
S:\bigl((\varphi\circ S)\otimes\operatorname{id}\bigr)\bigl((a\otimes1)(\Delta b)\bigr)
\mapsto\bigl((\varphi\circ S)\otimes\operatorname{id}\bigr)\bigl((\Delta a)(b\otimes1)\bigr),\quad \forall a,b\in{\mathcal A}.
$$
\end{enumerate}
\end{prop}

\begin{proof}
See Proposition~\ref{muphinupsi} for (1), and Proposition~\ref{invarianceF} for (2), which are consequences of $\varphi\circ S$ being 
a right integral.
For (3), see Proposition~\ref{modular} for (2).  The uniqueness of $\delta$ is due to $\varphi$ and $\varphi\circ S$ being faithful.
Finally, (4) is basically Proposition~\ref{antipodeS}\,(2), which is really a result that is true for any right integral (see Proposition~1.5 
of \cite{VDWangwha3}).
\end{proof}

\subsection{Relationships between $\sigma$, $\sigma^{\varphi\circ S}$ and the antipode $S$} \label{suba.2}

Let $\sigma$, $\sigma^{\varphi\circ S}$ be the modular automorphisms for $\varphi$ and $\varphi\circ S$, respectively.  
Recall also that $\sigma^{\varphi\circ S}=S^{-1}\circ\sigma^{-1}\circ S$, which can be seen in the proof of Proposition~\ref{automorphismsigmaS}.

Here are some results regarding their restrictions to the level of the base algebras:

\begin{prop}\label{sigmarestriction}
\begin{enumerate}
\item The restriction of $\sigma$ to ${\mathcal C}$ leaves ${\mathcal C}$ invariant, and we have:
$$
\sigma|_{\mathcal C}=S^2|_{\mathcal C}=S_{\mathcal B}\circ S_{\mathcal C}=\sigma^{\mu}.
$$
\item The restriction of $\sigma^{\varphi\circ S}$ to ${\mathcal B}$ leaves ${\mathcal B}$ invariant, and we have:
$$
\sigma^{\varphi\circ S}|_{\mathcal B}=S^{-2}|_{\mathcal B}=S_{\mathcal B}^{-1}\circ S_{\mathcal C}^{-1}=\sigma^{\nu}.
$$
\item We have:  $\mu\circ\sigma|_{\mathcal C}=\mu$ and $\nu\circ\sigma^{\varphi\circ S}|_{\mathcal B}=\nu$.
\end{enumerate}
\end{prop}

\begin{proof}
(1). Let $y\in{\mathcal C}$ and let $a\in{\mathcal A}$ be arbitrary.  Note that 
\begin{align}
\varphi(ya)&=\mu\bigl((\operatorname{id}\otimes\varphi)(\Delta(ya))\bigr)=\mu\bigl((\operatorname{id}\otimes\varphi)
((y\otimes1)(\Delta a))\bigr)=\mu\bigl(y(\operatorname{id}\otimes\varphi)(\Delta a)\bigr)   \notag \\
&=\mu\bigl((\operatorname{id}\otimes\varphi)(\Delta a)\sigma^{\mu}(y)\bigr)
=\mu\bigl((\operatorname{id}\otimes\varphi)((\Delta a)(\sigma^{\mu}(y)\otimes1))\bigr)
\notag \\
&=\mu\bigl((\operatorname{id}\otimes\varphi)(\Delta(a\sigma^{\mu}(y)))\bigr) 
=\varphi\bigl(a\sigma^{\mu}(y)\bigr).
\notag 
\end{align}
We used here the result of Proposition~\ref{muphinupsi}, and the fact that $\Delta y=(y\otimes1)E=E(y\otimes1)$ 
and $\Delta\bigl(\sigma^{\mu}(y)\bigr)=\bigl(\sigma^{\mu}(y)\otimes1\bigr)E=E\bigl(\sigma^{\mu}(y)\otimes1\bigr)$, 
because $y,\sigma^{\mu}(y)\in{\mathcal C}$ (Proposition~\ref{DeltaonBandC}).

Meanwhile, note that $\varphi(ya)=\varphi\bigl(a\sigma(y)\bigr)$.  Combining the two observations, we see that 
$$
\varphi\bigl(a\sigma(y)\bigr)=\varphi\bigl(a\sigma^{\mu}(y)\bigr).
$$
Since $\varphi$ is faithful and since the result is true for any $a\in{\mathcal A}$, this shows that  $\sigma(y)=
\sigma^{\mu}(y)$, for all $y\in{\mathcal C}$, that is, $\sigma|_{\mathcal C}=\sigma^{\mu}$.  We already know from 
\S\ref{sub1.3} that $\sigma^{\mu}=S_{\mathcal B}\circ S_{\mathcal C}$.  We also know that $S|_{\mathcal B}
=S_{\mathcal B}$ and $S|_{\mathcal C}=S_{\mathcal C}$ (Proposition~\ref{antipodeS}), so we have: 
$\sigma|_{\mathcal C}=\sigma^{\mu}=S^2|_{\mathcal C}$.

(2). The proof for the restriction $\sigma^{\varphi\circ S}|_{\mathcal B}=\sigma^{\nu}=S^{-2}|_{\mathcal B}$ is similar, 
with Proposition~\ref{phiSresults}\,(1).

(3). As it is known that $\mu\circ\sigma^{\mu}=\mu$ and $\nu\circ\sigma^{\nu}=\nu$, the results follow immediately 
from (1) and (2).
\end{proof}

The following results show how $\sigma$ and $\sigma^{\varphi\circ S}$ behave when the comultiplication map is applied: 

\begin{prop}\label{Deltasigma}
We have:
\begin{enumerate}
\item $\Delta\bigl(\sigma(a)\bigr)=(S^2\otimes\sigma)(\Delta a), \quad {\text { for all $a\in{\mathcal A}$,}}$
\item $\Delta\bigl(\sigma^{\varphi\circ S}(a)\bigr)=(\sigma^{\varphi\circ S}\otimes S^{-2})(\Delta a), \quad {\text { for all $a\in{\mathcal A}$.}}$
\end{enumerate}
\end{prop}

\begin{proof}
(1). Let $a,x\in{\mathcal A}$ be arbitrary.  By a characterization of the antipode $S$ given in 
Proposition~\ref{antipodeS}\,(1), we have:
\begin{align}
(\operatorname{id}\otimes\varphi)\bigl((1\otimes x)\Delta((\sigma(a))\bigr)
&=(\operatorname{id}\otimes\varphi)\bigl(S((\Delta x)(1\otimes\sigma(a)))\bigr)   
=S\bigl((\operatorname{id}\otimes\varphi)((\Delta x)(1\otimes\sigma(a)))\bigr)     \notag \\
&=S\bigl((\operatorname{id}\otimes\varphi)((1\otimes a)(\Delta x))\bigr)
=S^2\bigl((\operatorname{id}\otimes\varphi)((\Delta a)(1\otimes x))\bigr)
\notag \\
&=S^2\bigl((\operatorname{id}\otimes\varphi)((1\otimes x)(\operatorname{id}\otimes\sigma)(\Delta a))\bigr)
\notag \\
&=(\operatorname{id}\otimes\varphi)\bigl((1\otimes x)(S^2\otimes\sigma)(\Delta a)\bigr).
\notag
\end{align}
Since $\varphi$ is faithful and since $x\in{\mathcal A}$ is arbitrary, this shows that $\Delta\bigl(\sigma(a)\bigr)
=(S^2\otimes\sigma)(\Delta a)$, for all $a\in{\mathcal A}$.

(2). Proof for $\Delta\bigl(\sigma^{\varphi\circ S}(a)\bigr)=(\sigma^{\varphi\circ S}\otimes S^{-2})(\Delta a)$, $\forall a\in{\mathcal A}$, is similar, 
using an alternative characterization of $S$, namely, $S\bigl(((\varphi\circ S)\otimes\operatorname{id})((a\otimes1)(\Delta x))\bigr)
=((\varphi\circ S)\otimes\operatorname{id})\bigl((\Delta a)(x\otimes1)\bigr)$, noted in Proposition~\ref{phiSresults}\,(4).
\end{proof}

\begin{rem}
For $a\in{\mathcal A}$, we know that $\sigma(a)=\tilde{\sigma}_{-i}\bigl(\pi(a)\bigr)$ and that $\tau_{-i}\bigl(\pi(a)\bigr)
=S^2\bigl(\pi(a)\bigr)=S^2(a)$ (by the polar decomposition of $S$).  As such, the result (1) above is essentially Proposition~\ref{Deltasigma_t}, when $t=-i$.  The second result is analogous to Proposition~\ref{Deltasigma'_t}.
\end{rem}

\begin{cor}
We have: 
$$
\Delta\bigl(\sigma(S^{-2}(a))\bigr)=\bigl(\operatorname{id}\otimes(\sigma\circ S^{-2})\big)(\Delta a), 
\quad \forall a\in{\mathcal A}.
$$
\end{cor}

\begin{proof}
By Proposition~\ref{antipodeS}\,(4), applied twice, we know $\Delta\bigl(S^{-2}(a)\bigr)
=(S^{-2}\otimes S^{-2})(\Delta a)$.  Combine this result with (1) of Proposition~\ref{Deltasigma}.
\end{proof}

\subsection{Some results on the modular element} \label{suba.3}

In Proposition~\ref{phiSresults}\,(3), we noted the existence of a unique invertible element $\delta\in M({\mathcal A})$ 
such that $(\varphi\circ S)(a)=\varphi(a\delta)$, for all $a\in{\mathcal A}$.  In what follows, we will gather some additional results about $\delta$.

Let us begin with a lemma, which gives a similar result for $\varphi\circ S^{-1}$, also a right invariant functional:

\begin{lem} \label{lemmaphiSinverse}
We have: $(\varphi\circ S^{-1})(a)=\varphi(\delta^*a)$, for all $a\in{\mathcal A}$.
\end{lem}

\begin{proof}
Let $a\in{\mathcal A}$ be arbitrary.  Recall from Proposition~\ref{antipodeS}\,(5) that $S\bigl(S(a)^*\bigr)^*=a$.  It follows that $S^{-1}(a)^*=S(a^*)$. 
As $\varphi$ is a positive functional, we thus have:
$$
(\varphi\circ S^{-1})(a)=\varphi\bigl(S^{-1}(a)\bigr)=\overline{\varphi\bigl(S^{-1}(a)^*\bigr)}
=\overline{\varphi\bigl(S(a^*)\bigr)}=\overline{\varphi(a^*\delta)}=\overline{\varphi\bigl((\delta^*a)^*\bigr)}
=\varphi(\delta^*a).
$$
\end{proof}

In the below is a result showing what happens when the antipode map $S$, when extended to the multiplier algebra level, 
is applied to $\delta$:

\begin{prop}\label{Sdelta}
We have: $S(\delta)=(\delta^*)^{-1}$.
\end{prop}

\begin{proof}
Let $x\in{\mathcal A}$ be arbitrary.  As $S$ is anti-multiplicative, we have:
$$
\varphi(x)=(\varphi\circ S)\bigl(S^{-1}(x)\bigr)=\varphi\bigl(S^{-1}(x)\delta\bigr)=(\varphi\circ S^{-1})\bigl(S(\delta)x\bigr)
=\varphi\bigl(\delta^*S(\delta)x\bigr),
$$
where we used the result of Lemma~\ref{lemmaphiSinverse}.

As $\varphi$ is faithful and since $x\in{\mathcal A}$ is arbitrary, it follows that $\delta^*S(\delta)=1$. So we have: 
$S(\delta)^{-1}=\delta^*$ and $S(\delta)=(\delta^*)^{-1}$.
\end{proof}

Next, we gather some results regarding the modular automorphisms $\sigma$ and $\sigma^{\varphi\circ S}$:

\begin{prop}\label{sigmasigma'delta}
Under the modular automorphisms $\sigma$ and $\sigma^{\varphi\circ S}$, which can be naturally extended to the multiplier algebra level, 
we have:
\begin{enumerate}
\item $\sigma^{-1}(\delta)=[\sigma^{\varphi\circ S}]^{-1}(\delta)$;
\item $\sigma\bigl([\sigma^{\varphi\circ S}]^{-1}(\delta)\bigr)=\delta$ and $\sigma^{\varphi\circ S}\bigl(\sigma^{-1}(\delta)\bigr)=\delta$;
\item $\sigma^{\varphi\circ S}(a)=\delta\sigma(a)\delta^{-1}$ and $\sigma(a)=\delta^{-1}\sigma^{\varphi\circ S}(a)\delta$, for any $a\in{\mathcal A}$;
\item $\sigma\bigl([\sigma^{\varphi\circ S}]^{-1}(a)\bigr)=\delta^{-1}a\delta$,  for any $a\in{\mathcal A}$;
\end{enumerate}
\end{prop}

\begin{proof}
(1). Let $a\in{\mathcal A}$.  We have:
$$
(\varphi\circ S)(a)=\varphi(a\delta)=\varphi\bigl(\sigma^{-1}(\delta)a\bigr).
$$
Meanwhile, we have:
$$
(\varphi\circ S)(a)=(\varphi\circ S)(a\delta^{-1}\delta)=(\varphi\circ S)\bigl([\sigma^{\varphi\circ S}]^{-1}(\delta)a\delta^{-1}\bigr)
=\varphi\bigl([\sigma^{\varphi\circ S}]^{-1}(\delta)a\bigr).
$$
Compare the two equations.  Since $\varphi$ is faithful and since $a\in{\mathcal A}$ is arbitrary, this shows that 
$\sigma^{-1}(\delta)=[\sigma^{\varphi\circ S}]^{-1}(\delta)$.  

(2).  As a consequence of (1), we have: $\sigma\bigl([\sigma^{\varphi\circ S}]^{-1}(\delta)\bigr)=\delta$ and 
$\sigma^{\varphi\circ S}\bigl(\sigma^{-1}(\delta)\bigr)=\delta$.

(3). Let $a,x\in{\mathcal A}$. We have:
$$
(\varphi\circ S)(ax)=\varphi(ax\delta)=\varphi\bigl(x\delta\sigma(a)\bigr).
$$
Meanwhile,
$$
(\varphi\circ S)(ax)=(\varphi\circ S)\bigl(x\sigma^{\varphi\circ S}(a)\bigr)=\varphi\bigl(x\sigma^{\varphi\circ S}(a)\delta\bigr).
$$
Compare the two expressions.  Since $\varphi$ is faithful and since $x\in{\mathcal A}$ is arbitrary, this shows that 
$\delta\sigma(a)=\sigma^{\varphi\circ S}(a)\delta$. Or, equivalently, $\sigma^{\varphi\circ S}(a)=\delta\sigma(a)\delta^{-1}$ 
and $\sigma(a)=\delta^{-1}\sigma^{\varphi\circ S}(a)\delta$, true for any $a\in{\mathcal A}$.

(4). We may consider $x=[\sigma^{\varphi\circ S}]^{-1}(a)$, for $a\in{\mathcal A}$, and apply (3).  Then we have: 
$$
\sigma\bigl([\sigma^{\varphi\circ S}]^{-1}(a)\bigr)=\sigma(x)=\delta^{-1}\sigma^{\varphi\circ S}(x)\delta=\delta^{-1}a\delta.
$$
\end{proof}

The following result is a consequence of $\varphi$ and $\varphi\circ S$ being left and right invariant, respectively.  See 
Proposition~\ref{invarianceF} for the definitions of $F_1$, $F_2$ $F_3$, $F_4$, which are elements in $M({\mathcal A}\odot{\mathcal A})$.

\begin{prop}\label{invariancemodular}
Let $a\in{\mathcal A}$.  We have:
$$
(\varphi\otimes\operatorname{id})(\Delta a)=(\operatorname{id}\otimes\varphi)\bigl(F_1(1\otimes a)\bigr)\delta
=\delta^*(\operatorname{id}\otimes\varphi)\bigl((1\otimes a)F_3\bigr).
$$
Here, $F_1=(\operatorname{id}\otimes S)(E)\in M({\mathcal A}\odot{\mathcal A})$ and 
$F_3=(\operatorname{id}\otimes S^{-1})(E)\in M({\mathcal A}\odot{\mathcal A})$.
\end{prop}

\begin{proof}
Let $a,x\in{\mathcal A}$.  Observe:
\begin{align}
\varphi\bigl(x(\varphi\otimes\operatorname{id})(\Delta a)\bigr)&=(\varphi\otimes\varphi)\bigl((1\otimes x)(\Delta a)\bigr)
=\varphi\bigl((\operatorname{id}\otimes\varphi)((1\otimes x)(\Delta a))\bigr) \notag  \\
&=\varphi\bigl(S((\operatorname{id}\otimes\varphi)((\Delta x)(1\otimes a)))\bigr)
=\varphi\bigl(((\varphi\otimes S)\otimes\operatorname{id})(\Delta x)a\bigr)   \notag \\
&=\varphi\bigl(((\varphi\circ S)\otimes\operatorname{id})((x\otimes1)F_1(1\otimes a))\bigr)
=(\varphi\circ S)\bigl(x(\operatorname{id}\otimes\varphi)(F_1(1\otimes a))\bigr)  \notag \\
&=\varphi\bigl(x(\operatorname{id}\otimes\varphi)(F_1(1\otimes a))\delta\bigr).
\notag
\end{align}
Third equality is using the characterization of the antipode given in Proposition~\ref{antipodeS}\,(1). The fifth is an application of 
Proposition~\ref{phiSresults}\,(2), as $\varphi\circ S$ is a right integral.  The last equality is remembering $(\varphi\circ S)(\,\cdot\,)
=\varphi(\,\cdot\,\delta)$.

Since $x\in{\mathcal A}$ is arbitrary and since $\varphi$ is faithful, this shows that 
$$
(\varphi\otimes\operatorname{id})(\Delta a)=(\operatorname{id}\otimes\varphi)(F_1(1\otimes a))\delta, 
\quad {\text { for all $a\in{\mathcal A}$.}}
$$
Taking the adjoint, we have:
$(\varphi\otimes\operatorname{id})(\Delta (a^*))=\delta^*(\operatorname{id}\otimes\varphi)((1\otimes a^*)F_1^*)$, $\forall 
a\in{\mathcal A}$, since $\varphi$ is a positive functional.  Or, equivalently, 
$(\varphi\otimes\operatorname{id})(\Delta a)=\delta^*(\operatorname{id}\otimes\varphi)((1\otimes a)F_1^*)$. Note here 
that since $F_1=(\operatorname{id}\otimes S)(E)$ and since $S(a)^*=S^{-1}(a^*)$ for any $a\in{\mathcal A}$, we have 
$F_1^*=(\operatorname{id}\otimes S^{-1})(E^*)=(\operatorname{id}\otimes S^{-1})(E)=F_3$.  In other words, we have:
$$
(\varphi\otimes\operatorname{id})(\Delta a)=\delta^*(\operatorname{id}\otimes\varphi)((1\otimes a)F_3),
\quad {\text { for all $a\in{\mathcal A}$.}}
$$
\end{proof}

Eventually, we wish to find what $\Delta(\delta)$ is.  At present it is not clear, but the following proposition will help us 
in that direction.

\begin{prop}\label{Deltadelta_prep}
Let $p,q\in{\mathcal A}$, and $\delta\in M({\mathcal A})$ be the modular element.  We have:
$$
(\varphi\otimes\varphi)\bigl((p\otimes q)\Delta(\delta)\bigr)=\bigl((\varphi\circ S)\otimes(\varphi\circ S)\bigr)\bigl((p\otimes q)E\bigr),
$$
$$
(\varphi\otimes\varphi)\bigl(\Delta(\delta^*)(p\otimes q)\bigr)=\bigl((\varphi\circ S)\otimes(\varphi\circ S)\bigr)\bigl(E(p\otimes q)\bigr).
$$
\end{prop}

\begin{proof}
The second result can be obtained immediately from the first one by taking the adjoint.  So let us just prove the first result.

Let $a,b\in{\mathcal A}$ be arbitrary.  By Proposition~\ref{invariancemodular}, we have:
$$
(\varphi\otimes\varphi)\bigl((1\otimes a)(\Delta b)\Delta(\delta)\bigr)=\varphi\bigl(a(\varphi\otimes\operatorname{id})
(\Delta(b\delta))\bigr)=\varphi\bigl(a(\operatorname{id}\otimes\varphi)(F_1(1\otimes b\delta))\delta\bigr).
$$
Since $\varphi(\,\cdot\,\delta)=\varphi\circ S$, this becomes:
$$
(\varphi\otimes\varphi)\bigl((1\otimes a)(\Delta b)\Delta(\delta)\bigr)
=\bigl((\varphi\circ S)\otimes(\varphi\circ S)\bigr)\bigl((a\otimes1)F_1(1\otimes b)\bigr)
=\bigl((\varphi\circ S)\otimes(\varphi\circ S)\bigr)\bigl((1\otimes a)\varsigma F_1(b\otimes1)\bigr),
$$
where $\varsigma$ denotes taking the flip on $M({\mathcal A}\odot{\mathcal A})$.  

Note that $\varsigma F_1=\varsigma\bigl((\operatorname{id}\otimes S)E\bigr)=(S\otimes\operatorname{id})(\varsigma E)$. 
Since $\varsigma E=(S^{-1}\otimes S^{-1})(E)$, we thus have $\varsigma F_1=(\operatorname{id}\otimes S^{-1})(E)=F_3$. 
Apply this result to the above, then we have:
\begin{equation}\label{(Deltadelta_prepeqn1)}
(\varphi\otimes\varphi)\bigl((1\otimes a)(\Delta b)\Delta(\delta)\bigr)
=\bigl((\varphi\circ S)\otimes(\varphi\circ S)\bigr(\bigl((1\otimes a)F_3(b\otimes1)\bigr)
=\bigl((\varphi\circ S)\otimes(\varphi\circ S)\bigr)\bigl((1\otimes a)(\Delta b)\bigr),
\end{equation}
where we used here the right invariance result $\bigl((\varphi\circ S)\otimes\operatorname{id}\bigr)\bigl(F_3(b\otimes1)\bigr)
=\bigl((\varphi\circ S)\otimes\operatorname{id}\bigr)(\Delta b)$, as in Proposition~\ref{phiSresults}\,(2).

Equation~\eqref{(Deltadelta_prepeqn1)} holds true for any $a,b\in{\mathcal A}$.  But note that the elements of the form $(1\otimes a)(\Delta b)$ 
span $({\mathcal A}\odot{\mathcal A})E$.  Therefore, it is equivalent to saying
$$
(\varphi\otimes\varphi)\bigl((p\otimes q)\Delta(\delta)\bigr)=\bigl((\varphi\circ S)\otimes(\varphi\circ S)\bigr)\bigr((p\otimes q)E\bigr), \quad 
{\text { for all $p,q\in{\mathcal A}$.}}
$$
\end{proof}

Before we find $\Delta(\delta)$, we first prove the following lemma, which is yet another consequence of the right invariance 
of the functional $\varphi\circ S$.

\begin{lem}\label{rinvariancelemma}
Let $p,q\in{\mathcal A}$.  Then we have:
$$
\bigl((\varphi\circ S)\otimes\operatorname{id}\bigr)\bigl((p\otimes q)E\bigr)=qS^{-1}\bigl(((\varphi\circ S)\otimes\operatorname{id})(\Delta p)\bigr).
$$
\end{lem}

\begin{proof}
Let $\omega\in{\mathcal A}^*$ be arbitrary.  Then:
$$
\omega\bigl(((\varphi\circ S)\otimes\operatorname{id})((p\otimes q)E)\bigr)
=(\varphi\circ S)\bigl(p(\operatorname{id}\otimes\omega)((1\otimes q)E)\bigr)=(\varphi\circ S)(px),
$$
where $x=(\operatorname{id}\otimes\omega)\bigl((1\otimes q)E\bigr)$.  As $E$ is full, note that such elements span the base 
algebra ${\mathcal B}$.  By Proposition~\ref{phiSresults}\,(1), we have:
$$
(\varphi\circ S)(px)=\nu\bigl(((\varphi\circ S)\otimes\operatorname{id})(\Delta(px))\bigr)
=\nu\bigl(((\varphi\circ S)\otimes\operatorname{id})((\Delta p)(1\otimes x))\bigr)=\nu\bigl(((\varphi\circ S)\otimes\operatorname{id})(\Delta p)x\bigr).
$$
Here, we used the fact that since $x\in{\mathcal B}$, we have $\Delta x=E(1\otimes x)$ by Proposition~\ref{DeltaonBandC}, 
from which it follows that $\Delta(px)=(\Delta p)(\Delta x)=(\Delta p)E(1\otimes x)=(\Delta p)(1\otimes x)$. 

By Proposition~\ref{phiS_rightinvariant}, we know $\bigl((\varphi\circ S)\otimes\operatorname{id}\big)(\Delta p)\in M({\mathcal B})$.
Therefore, combining the results:
$$
\omega\bigl(((\varphi\circ S)\otimes\operatorname{id})((p\otimes q)E)\bigr)
=\nu\bigl(((\varphi\circ S)\otimes\operatorname{id})(\Delta p)x\bigr)=
\omega\bigl(q(\nu\otimes\operatorname{id})([((\varphi\circ S)\otimes\operatorname{id})(\Delta p)\otimes1]E)\bigr).
$$
As $\omega\in{\mathcal A}^*$ is arbitrary, this means that 
$$
\bigl((\varphi\circ S)\otimes\operatorname{id}\bigr)\bigl((p\otimes q)E\bigr)
=q(\nu\otimes\operatorname{id})\bigl([((\varphi\circ S)\otimes\operatorname{id})(\Delta p)\otimes1]E\bigr).
$$

Since $\bigl((\varphi\circ S)\otimes\operatorname{id}\bigr)(\Delta p)\in M({\mathcal B})$, using the map $S_{\mathcal C}$, 
as characterized in Equation~\eqref{(S_Cmap)}, we have:
$\bigl[((\varphi\circ S)\otimes\operatorname{id})(\Delta p)\otimes1\bigr]E=\bigl[(1\otimes S_{\mathcal C}^{-1}
(((\varphi\circ S)\otimes\operatorname{id})(\Delta p))\bigr]E$.  Note also that $(\nu\otimes\operatorname{id})(E)=1$.  Putting all these 
together, we thus obtain:
$$
\bigl((\varphi\circ S)\otimes\operatorname{id}\bigr)\bigl((p\otimes q)E\bigr)
=qS_{\mathcal C}^{-1}\bigl(((\varphi\circ S)\otimes\operatorname{id})(\Delta p)\bigr)
=qS^{-1}\bigl(((\varphi\circ S)\otimes\operatorname{id})(\Delta p)\bigr),
$$
as $S^{-1}$ extends $S_{\mathcal C}^{-1}$. 
\end{proof}

The next proposition provides us with some characterizations of $\Delta(\delta)$ and $\Delta(\delta^*)$:

\begin{prop}\label{Deltadelta}
Let $\delta\in M({\mathcal A})$ be the modular element.  We have:
\begin{itemize}
\item $\Delta(\delta)=\bigl(\delta\otimes S^{-1}(\delta^{-1})\bigr)E=\bigl(\delta\otimes S^{-2}(\delta^*)\bigr)E
=E\bigl(\delta\otimes S^{-2}(\delta^*)\bigr)E
$.
\item $\Delta(\delta^*)=E\bigl(\delta^*\otimes S^2(\delta)\bigr)=E\bigl(\delta^*\otimes S^2(\delta)\bigr)E$
\item $\Delta(\delta^*)=(\delta\otimes\delta^*)E=E(\delta\otimes\delta^*)E$
\item $\Delta(\delta)=E(\delta^*\otimes\delta)=E(\delta^*\otimes\delta)E$
\item $\Delta(\delta)=E(\delta\otimes\delta)E$
\item $\Delta(\delta^*)=E(\delta^*\otimes\delta^*)E$
\end{itemize}
\end{prop}

\begin{proof}
(1). Recall from Proposition~\ref{Deltadelta_prep} that $(\varphi\otimes\varphi)\bigl((p\otimes q)\Delta(\delta)\bigr)
=\bigl((\varphi\circ S)\otimes(\varphi\circ S)\bigr)\bigr((p\otimes q)E\bigr)$, for any $p,q\in{\mathcal A}$.  By Lemma~\ref{rinvariancelemma}, 
we have:
\begin{equation}\label{(Deltadelta_eq1)}
(\varphi\otimes\varphi)\bigl((p\otimes q)\Delta(\delta)\bigr)
=(\varphi\circ S)\bigl(((\varphi\circ S)\otimes\operatorname{id})((p\otimes q)E)\bigr)
=(\varphi\circ S)\bigl(qS^{-1}(((\varphi\circ S)\otimes\operatorname{id})(\Delta p))\bigr).
\end{equation}

Meanwhile, note that
\begin{align} \label{(Deltadelta_eq2)}
\bigl((\varphi\circ S)\otimes\operatorname{id}\bigr)(\Delta p)&=\bigl((\varphi\circ S)\otimes\operatorname{id}\bigr)\bigl(F_3(p\otimes1)\bigr)
=(\varphi\otimes\operatorname{id})\bigl(F_3(p\delta\otimes1)\bigr)  \notag \\
&=(\operatorname{id}\otimes\varphi)\bigl(\varsigma F_3(1\otimes p\delta)\bigr)
=(\operatorname{id}\otimes\varphi)\bigl(F_1(1\otimes p\delta)\bigr)  \notag \\
&=(\varphi\otimes\operatorname{id})\bigl(\Delta(p\delta)\bigr)\delta^{-1}  
=\delta^*(\operatorname{id}\otimes\varphi)\bigl((1\otimes p\delta)F_3\bigr)\delta^{-1}.
\end{align}
The first equality is Proposition~\ref{phiSresults}\,(2); The second used $\varphi\circ S=\varphi(\,\cdot\,\delta)$; In the third and fourth, the flip map 
is applied, together with the observation earlier (see proof of Proposition~\ref{Deltadelta_prep}) that $\varsigma F_1=F_3$; The fifth and sixth used 
the result of Proposition~\ref{invariancemodular}.

Insert into Equation~\eqref{(Deltadelta_eq1)} the result of Equation~\eqref{(Deltadelta_eq2)}.  Then we have:
\begin{align} \label{(Deltadelta_eq3)}
&(\varphi\otimes\varphi)\bigl((p\otimes q)\Delta(\delta)\bigr)=(\varphi\circ S)\bigl(qS^{-1}(((\varphi\circ S)\otimes\operatorname{id})(\Delta p))\bigr)
\notag \\
&=(\varphi\circ S)\bigl(qS^{-1}(\delta^*(\operatorname{id}\otimes\varphi)[(1\otimes p\delta)F_3]\delta^{-1})\bigr)
=\varphi\bigl(qS^{-1}(\delta^*(\operatorname{id}\otimes\varphi)[(1\otimes p\delta)F_3]\delta^{-1})\delta\bigr)
\notag \\
&=\varphi\bigl(qS^{-1}(S(\delta)\delta^*(\operatorname{id}\otimes\varphi)[(1\otimes p\delta)F_3]\delta^{-1})\bigr)
=\varphi\bigl(qS^{-1}((\operatorname{id}\otimes\varphi)[(1\otimes p\delta)F_3]\delta^{-1})\bigr)
\notag \\
&=\varphi\bigl(qS^{-1}(\delta^{-1})S^{-1}((\operatorname{id}\otimes\varphi)[(1\otimes p\delta)F_3])\bigr).
\end{align}
The fourth and sixth equalities used the anti-multiplicativity $S^{-1}$; In the fifth, we used the result 
of Proposition~\ref{Sdelta}, namely $S(\delta)=(\delta^*)^{-1}$.

Note that by applying the flip map and using again $\varsigma F_1=F_3$, we have: 
\begin{equation} \label{(Deltadelta_eq4)}
S^{-1}\bigl((\operatorname{id}\otimes\varphi)[(1\otimes p\delta)F_3]\bigr)
=S^{-1}\bigl((\varphi\otimes\operatorname{id})[(p\delta\otimes1)F_1]\bigr)
=(\varphi\otimes\operatorname{id})\bigl[(p\delta\otimes1)E\bigr],
\end{equation}
because $F_1=(\operatorname{id}\otimes S)(E)$.  Insert the result of Equation~\eqref{(Deltadelta_eq4)} into 
Equation~\eqref{(Deltadelta_eq3)}, to obtain:
\begin{align}
(\varphi\otimes\varphi)\bigl((p\otimes q)\Delta(\delta)\bigr)
&=\varphi\bigl(qS^{-1}(\delta^{-1})(\varphi\otimes\operatorname{id})[(p\delta\otimes1)E]\bigr)  \notag \\
&=(\varphi\otimes\varphi)\bigl((p\delta\otimes qS^{-1}(\delta^{-1}))E\bigr)   \notag \\
&=(\varphi\otimes\varphi)\bigl((p\otimes q)(\delta\otimes S^{-1}(\delta^{-1}))E\bigr).
\notag
\end{align}

Here, note that $\varphi$ is faithful and that $p,q\in{\mathcal A}$ are arbitrary, which means that we have
$$
\Delta(\delta)=\bigl(\delta\otimes S^{-1}(\delta^{-1})\bigr)E.
$$
Or $\Delta(\delta)=\bigl(\delta\otimes S^{-2}(\delta^*)\bigr)E$, by using the fact $\delta^*=S(\delta^{-1})$.  Meanwhile, noting that 
$\Delta(\delta)=E\Delta(\delta)$, we can also write this as $\Delta(\delta)=E\bigl(\delta\otimes S^{-2}(\delta^*)\bigr)E$.

(2). From (1), we know $\Delta(\delta)=\bigl(\delta\otimes S^{-2}(\delta^*)\bigr)E$.  Take the adjoint, to obtain
$\Delta(\delta^*)=E\bigl(\delta^*\otimes S^2(\delta)\bigr)$, using the property of $S$.  Also, as we should have 
$\Delta(\delta^*)=\Delta(\delta^*)E$, we can also write this as $\Delta(\delta^*)=E\bigl(\delta^*\otimes S^2(\delta)\bigr)E$.

(3). From (1), we saw $\Delta(\delta)=\bigl(\delta\otimes S^{-1}(\delta^{-1})\bigr)E$.  As the comultiplication preserves 
multiplication and since $\Delta(1)=E$, we can see from this quickly that
$$
\Delta(\delta^{-1})=E\bigl(\delta^{-1}\otimes[S^{-1}(\delta^{-1})]^{-1}\bigr)=E\bigl(\delta^{-1}\otimes S^{-1}(\delta)\bigr).
$$
As $S(\delta^{-1})=\delta^*$, and by using the result we just obtained on $\Delta(\delta^{-1})$, we have:
\begin{align}
\Delta(\delta^*)&=\Delta\bigl(S(\delta^{-1})\bigr)=(S\otimes S)\Delta^{\operatorname{cop}}(\delta^{-1})
=(S\otimes S)\bigl(\varsigma E(S^{-1}(\delta)\otimes\delta^{-1})\bigr)  \notag \\
&=\bigl(\delta\otimes S(\delta^{-1})\bigr)(S\otimes S)(\varsigma E)=\bigl(\delta\otimes\delta^*\bigr)E.
\notag
\end{align}
Again, as we should have $\Delta(\delta^*)=E\Delta(\delta^*)$, we can also write this as $\Delta(\delta^*)
=E\bigl(\delta\otimes\delta^*\bigr)E$.

(4). From (3), we saw $\Delta(\delta^*)=\bigl(\delta\otimes\delta^*\bigr)E$.  Take the adjoint, to obtain: 
$\Delta(\delta)=E\bigl(\delta^*\otimes\delta\bigr)$.  As before, this can be also written as 
$\Delta(\delta)=E\bigl(\delta^*\otimes\delta\bigr)E$.

(5). Consider $x=(\operatorname{id}\otimes\varphi)\bigl(F_1(1\otimes a)\bigr)
=\bigl(\operatorname{id}\otimes(\varphi\circ S)(S^{-1}(a)\,\cdot\,)\bigr)(E)$, for $a\in{\mathcal A}$.
It is an element in ${\mathcal B}$, 
and it is evident that such elements span ${\mathcal B}$.  Note that since $x\in{\mathcal B}$, we have $\Delta x
=E(1\otimes x)=(1\otimes x)E$.  We thus have:
$(1\otimes x)\Delta(\delta)=(1\otimes x)E\Delta(\delta)=\Delta(x\delta)$.
But note that by Proposition~\ref{invariancemodular}, we have:
$$
x\delta=(\operatorname{id}\otimes\varphi)\bigl(F_1(1\otimes a)\bigr)\delta=
\delta^*(\operatorname{id}\otimes\varphi)\bigl((1\otimes a)F_3\bigr)=\delta^*\tilde{x},
$$
where $\tilde{x}=(\operatorname{id}\otimes\varphi)\bigl((1\otimes a)F_3\bigr)$, also an element of ${\mathcal B}$, 
so we have $\Delta(\tilde{x})=E(1\otimes\tilde{x})=(1\otimes\tilde{x})E$.

Combining these observations, we have:
\begin{align}
(1\otimes x)\Delta(\delta)&=\Delta(x\delta)=\Delta(\delta^*\tilde{x})=\Delta(\delta^*)(1\otimes\tilde{x})  \notag \\
&=E(\delta\otimes\delta^*)E(1\otimes\tilde{x})=E(\delta\otimes\delta^*)(1\otimes\tilde{x})E  \notag \\
&=E(\delta\otimes\delta^*\tilde{x})E=E(\delta\otimes x\delta)E=E(1\otimes x)(\delta\otimes\delta)E  \notag \\
&=(1\otimes x)E(\delta\otimes\delta)E,
\notag
\end{align}
where we used $\Delta(\delta^*)=E(\delta\otimes\delta^*)E$, observed in (3) above. 

As noted above, the elements $x=(\operatorname{id}\otimes\varphi)\bigl(F_1(1\otimes a)\bigr)$, $a\in{\mathcal A}$, span ${\mathcal B}$. 
Note also that ${\mathcal B}$ is non-degenerate.  So this result means that
$$
\Delta(\delta)=E(\delta\otimes\delta)E.
$$

(6). From $\Delta(\delta)=E(\delta\otimes\delta)E$, take the adjoint, to obtain  
$\Delta(\delta^*)=E(\delta^*\otimes\delta^*)E$.
\end{proof}

We do not know whether $\delta$ is self-adjoint, so $\Delta(\delta)$ and $\Delta(\delta^*)$ can be different and we made 
a distinction so far. 

However, under the {\em quasi-invariance assumption\/} (see Section~\S\ref{sub5.3}), it seems that the modular element 
$\delta$ becomes self-adjoint.  See discussion given in \S\ref{sub5.3}, such as Proposition~\ref{vaesRN} and the paragraphs 
following it.  Considering this, we revisit in the below some of the results we obtained so far for the modular element $\delta$, 
when we restrict ourselves to the situation when $\delta$ is self-adjoint. The self-adjointness of $\delta$ makes things simpler.  

\begin{prop}\label{deltasa}
Assume that $\delta$ is self-adjoint.  Then we have the following simpler results:
\begin{enumerate}
\item $(\varphi\circ S)(a)=\varphi(a\delta)$ and $(\varphi\circ S^{-1})(a)=\varphi(\delta a)$, for all $a\in{\mathcal A}$.
\item $S(\delta)=\delta^{-1}$ and $S^2(\delta)=\delta$.
\item $(\varphi\otimes\operatorname{id})(\Delta a)=(\operatorname{id}\otimes\varphi)\bigl(F_1(1\otimes a)\bigr)\delta
=\delta(\operatorname{id}\otimes\varphi)\bigl((1\otimes a)F_3\bigr)$, for all $a\in{\mathcal A}$.
\item $\Delta(\delta)=(\delta\otimes\delta)E=E(\delta\otimes\delta)=E(\delta\otimes\delta)E$.
\end{enumerate}
\end{prop}

\begin{proof}
(1). See Lemma~\ref{lemmaphiSinverse}, and let $\delta^*=\delta$.

(2). See Proposition~\ref{Sdelta}, and let $\delta^*=\delta$.

(3). See Proposition~\ref{invariancemodular}, and let $\delta^*=\delta$.

(4). See Proposition~\ref{Deltadelta}, and let $\delta^*=\delta$.
\end{proof}

\begin{rem}
As discussed in \S\ref{sub5.3} and as indicated above, it seems to be the case that the {\em quasi-invariance Assumption\/} 
leads to the self-adjointness of $\delta$.  The reasoning requires going up to the von Neumann algebra level and 
back down, so not purely algebraic.  It may be possible to find a more direct proof, but as the focus of the current paper 
is on the construction of a $C^*$-algebraic object, we did not make an attempt to develop such a proof.  Moreover, 
in the purely algebraic setting, as there is no reason to have to require the quasi-invariance condition, the need for 
such a proof is probably not too significant.
\end{rem}

\bigskip\bigskip


\begin{thebibliography}{10}

\bibitem{BS}
S.~Baaj and G.~Skandalis, \emph{Unitaires multiplicatifs et dualit\'e pour les produits crois\'es de 
  {$C^*$}-alg\`ebres}, Ann. Scient. \'Ec. Norm. Sup., $4^e$ s\'erie \textbf{t. 26} (1993), 425--488 (French).

\bibitem{BohmGomezLopez}
G.~B{\"o}hm, J.~G{\'o}mez-Torecillas, and E.~L{\'o}pez-Centella, \emph{Weak
  multiplier bialgebras}, Trans. Amer. Math. Soc. \textbf{367} (2015),
  8681--8872.

\bibitem{BNSwha1}
G.~B{\" o}hm, F.~Nill, and K.~Szlach{\' a}nyi, \emph{Weak {H}opf algebras {I}.
  {I}ntegral theory and {$C^*$}-structure}, J. Algebra \textbf{221} (1999),
  385--438.

\bibitem{BSzMultIso}
G.~B{\" o}hm and K.~Szlach{\' a}nyi, \emph{Weak {$C^*$-Hopf} algebras and multiplicative isometries}, 
  J. Operator Theory \textbf{45} (2001), 357--376.

\bibitem{EnSMF}
M.~Enock, \emph{{M}easured {Q}uantum {G}roupoids in {A}ction}, M{\'e}moires de
  la SMF, no. 114, Soc. Math. Fr., 2008.

\bibitem{nosinglefaithfulintegral}
M.C.~Iovanov and L.~Kadison, \emph{When weak quasi {H}opf algebras are
  {F}robenius}, Proc. Amer. Math. Soc. \textbf{138} (2010), 837--845.

\bibitem{BJK_mpi}
B.J.~Kahng, \emph{Multiplicative partial isometries and {$C^*$-algebraic}
  quantum groupoids}, preprint ({\texttt{arXiv:2002.01995}}).

\bibitem{BJKVD_LSthm}
B.J.~Kahng and A.~{Van Daele}, \emph{The {Larson--Sweedler} theorem for weak multiplier {H}opf
  algebras}, Comm. Algebra \textbf{46} (2018), 1--27.

\bibitem{BJKVD_SepId}
B.J.~Kahng and A.~{Van Daele}, \emph{Separability idempotents in {$C^*$}-algebras}, J. Noncommut.
  Geom. \textbf{12} (2018), 997--1040.

\bibitem{BJKVD_qgroupoid1}
B.J.~Kahng and A.~{Van Daele}, \emph{A class of {$C^*$-algebraic} locally compact quantum groupoids
  {Part I}. {M}otivation and definition}, Internat. J. Math. \textbf{29}
  (2018), 1850029.
  
\bibitem{BJKVD_qgroupoid2}
B.J.~Kahng and A.~{Van Daele}, \emph{A class of {$C^*$-algebraic} locally compact quantum groupoids
  {Part II}. {M}ain theory}, Adv. Math. \textbf{354} (2019), 106761.
  
\bibitem{BJKVD_qgroupoid3}
B.J.~Kahng and A.~{Van Daele}, \emph{A class of {$C^*$-algebraic} locally
  compact quantum groupoids {Part III}: {D}uality}, in preparation.  

\bibitem{KuKMS}
J.~Kustermans, \emph{{KMS}-weights on {$C^*$}-algebras}, preprint
  ({\texttt{funct-an/9704008}}).

\bibitem{KuVaweightC*}
J.~Kustermans and S.~Vaes, \emph{Weight theory for {$C^*$}-algebraic quantum
  groups}, preprint ({\texttt{arXiv:math/9901063}}).

\bibitem{KuVa}
J.~Kustermans and S.~Vaes, \emph{Locally compact quantum groups}, Ann. Scient. \'Ec. Norm. Sup.,
  $4^e$ s\'erie \textbf{t. 33} (2000), 837--934.

\bibitem{KuVD}
J.~Kustermans and A.~{Van Daele}, \emph{$C^*$-algebraic quantum groups arising from algebraic quantum groups}, Internat. J. Math.,
  \textbf{8} (1997), 1067--1139.

\bibitem{LesSMF}
F.~Lesieur, \emph{{M}easured {Q}uantum {G}roupoids}, M{\'e}moires de la SMF,
  no. 109, Soc. Math. Fr., 2007.

\bibitem{MNW}
T.~Masuda, Y.~Nakagami,  and  S.~Woronowicz, \emph{$C^*$-algebraic framework for quantum groups}, Internat. J. Math.,
  \textbf{14} (2003), 903--1001.

\bibitem{Patbook}
A.L.T.~Paterson, \emph{{G}roupoids, {I}nverse {S}emigroups, and their
  {O}perator {A}lgebras}, Progress in Mathematics, no. 170, Birkh{\" a}user,
  1998.

\bibitem{Pedersenbook}
G.K.~Pedersen, \emph{{$C^*$-algebras} and their {A}utomorphism {G}roups}, Academic Press,
  1979.

\bibitem{Renbook}
J.N.~Renault, \emph{{A} {G}roupoid {A}pproach to {$C^*$}-algebras}, Lecture
  Notes in Mathematics, no. 793, Springer-Verlag, 1980.

\bibitem{Str}
S.~Str{\u a}til{\u a}, \emph{{M}odular {T}heory in {O}perator {A}lgebras},
  Abacus Press, 1981.

\bibitem{Tk2}
M.~Takesaki, \emph{{T}heory of {O}perator {A}lgebras {II}}, Encyclopaedia of
  Mathematical Sciences, vol. 125, Springer-Verlag, 2002.

\bibitem{Timm_aqgdual}
T.~Timmermann, \emph{On duality of algebraic quantum groupoids}, preprint
  ({\texttt{arXiv:1605.06384}}).

\bibitem{Timm_aqgintegral}
T.~Timmermann, \emph{Integration on algebraic quantum groupoids}, Internat. J. Math.
  \textbf{27} (2016), 139--211.

\bibitem{TimmVD_wmhaalgebroid}
T.~Timmermann and A.~{Van Daele}, \emph{Multiplier {H}opf algebroids arising
  from weak multiplier {H}opf algebras},  \textbf{106} (2015), 73--110.

\bibitem{VaRN}
S.~Vaes, \emph{A {R}adon--{N}ikodym theorem for von {N}eumann algebras}, J.
  Operator Theory \textbf{46} (2001), 477--489.

\bibitem{VD_multiplierHopf}
A.~{Van Daele}, \emph{Multiplier {H}opf algebras}, Trans. Amer. Math. Soc.
  \textbf{342} (1994), no.~2, 917--932.

\bibitem{VD_multiplierHopfduality}
A.~{Van Daele}, \emph{An algebraic framework for group duality}, Adv. Math.
  \textbf{140} (1998), 323--366.

\bibitem{VDsepid}
A.~{Van Daele}, \emph{Separability idempotents and multiplier algebras}, 2013,
  preprint ({\texttt{arXiv:1301.4398v1}} and {\texttt{arXiv:1301.4398v2}}).

\bibitem{VD_aqg}
A.~{Van Daele}, \emph{Algebraic quantum groupoids---an example}, preprint
  ({\texttt{arXiv:1702.04903}}).

\bibitem{VDvN}
A.~{Van Daele}, \emph{Locally compact quantum groups. {A} von {N}eumann algebra approach}, 
 SIGMA \textbf{10} (2014), 082.
  
\bibitem{VDWangwha0}
A.~{Van Daele} and S.~Wang, \emph{Weak multiplier {H}opf algebras. {P}reliminaries, motivation and
  basic examples}, Operator algebras and quantum groups, vol.~98, Banach Center
  Publications, 2012, pp.~367--415.

\bibitem{VDWangwha1}
A.~{Van Daele} and S.~Wang, \emph{Weak multiplier {H}opf algebras. {T}he main theory}, J. Reine
  Angew. Math. (Crelles Journal) \textbf{705} (2015), 155--209.

\bibitem{VDWangwha2}
A.~{Van Daele} and S.~Wang, \emph{Weak multiplier {H}opf algebras {II}. {T}he
  source and target algebras}, preprint ({\texttt{arXiv:1403.7906}}).

\bibitem{VDWangwha3}
A.~{Van Daele} and S.~Wang, \emph{Weak multiplier {H}opf algebras {III}. {I}ntegrals and duality},
  preprint ({\texttt{arXiv:1701.04951}}).

\bibitem{Wr7}
S.~L. Woronowicz, \emph{From multiplicative unitaries to quantum groups},
  Internat. J. Math. \textbf{7} (1996), no.~1, 127--149.

\end{thebibliography}



\end{document}